\documentclass[11pt,reqno,a4paper]{amsart}
\usepackage{amsmath,amssymb}
\usepackage{amsrefs}
\usepackage{mathtools}
\usepackage[hidelinks]{hyperref}
\usepackage{comment}
\usepackage{enumitem}
\usepackage{cases}

\numberwithin{equation}{section}

\newtheorem{theorem}{Theorem}
\newtheorem{proposition}{Proposition}[section]
\newtheorem{lemma}[proposition]{Lemma}
\newtheorem{corollary}[proposition]{Corollary}

\theoremstyle{remark}
\newtheorem{remark}[proposition]{Remark}
\newtheorem*{remark*}{Remark}

\renewcommand{\epsilon}{\varepsilon}
\newcommand{\elliprhs}{\mathfrak{h}}

\DeclarePairedDelimiter\abs{\lvert}{\rvert}
\DeclarePairedDelimiter\norm{\lVert}{\rVert}

\makeatletter
\renewcommand*\env@matrix[1][\arraystretch]{%
  \edef\arraystretch{#1}%
  \hskip -\arraycolsep
  \let\@ifnextchar\new@ifnextchar
  \array{*\c@MaxMatrixCols c}}
\makeatother

\begin{document}

\title{Vortex dynamics for the Gross-Pitaevskii equation}

\author[M.~del Pino]{Manuel del Pino}
\address{\noindent M.~del Pino: Department of Mathematical Sciences University of Bath, Bath BA2 7AY, United Kingdom}
\email{mdp59@bath.ac.uk}

\author[R.~Juneman]{Rowan Juneman}
\address{\noindent R.~Juneman: Department of Mathematical Sciences University of Bath, Bath BA2 7AY, United Kingdom}
\email{rj493@bath.ac.uk}

\author[M.~Musso]{Monica Musso}
\address{\noindent M.~Musso: Department of Mathematical Sciences University of Bath, Bath BA2 7AY, United Kingdom.}
\email{mm2683@bath.ac.uk}

\begin{abstract}
We rigorously establish the formal asymptotics of Neu \cite{neu1990} for Gross-Pitaevskii \mbox{vortex} dynamics in the plane. Given any integer \(n\geq2\), we construct a family of \(n\)-vortex \mbox{solutions} with vortices of degree \(\pm1\), and describe precisely the solution profile and \mbox{associated} vortex dynamics on an arbitrarily large, finite time interval. We compute an asymptotic \mbox{expansion} of the vortex positions in terms of the vortex core size \(\epsilon>0\), and show that the dynamics is governed at leading order as \(\epsilon\to0\) by the classical Helmholtz-Kirchhoff system. Moreover, we show that the first correction to the leading order dynamics is determined by the solution of a linear wave equation, justifying a formal expansion found by Ovchinnikov and Sigal \cite{ovchinnikovsigal1998}.
\end{abstract}

\maketitle

\section{Introduction}

A striking feature of many nonlinear field theories in mathematical physics is the presence of smooth, particle-like solutions called solitons. These solutions take various forms, such as kinks, monopoles or instantons, and are characterized by a non-trivial topological structure which differs from the vacuum \cite{mantonsutcliffe2004}. In two space dimensions, complex-valued field theories give rise to an important class of soliton solutions called \textit{vortices}. Here, the particle-like components correspond to isolated zeroes (vortices) of the field with non-zero topological degree, indicating points where the complex phase becomes singular and the field ``vorticity" becomes highly concentrated. Away from the zero set the complex modulus is approximately constant, thus vortices tend to dominate the global properties of the field which otherwise appears fairly uniform.

Of the many challenges associated to the mathematical analysis of vortices, one of the most notable concerns the description of vortex dynamics. We are interested in this issue for the Gross-Pitaevskii equation in the plane
\begin{equation}
\label{GPeq}
    iu_{t}+\Delta{u}+\frac{1}{\epsilon^{2}}(1-\abs{u}^{2})u=0,\quad\quad (x,t)\in\mathbb{R}^{2}\times(0,T),
\end{equation}where \(u=u(x,t):\mathbb{R}^{2}\times[0,T]\to\mathbb{C}\) is a complex-valued function, \(0<T<\infty\) and \(\epsilon>0\) is a small parameter. Problem (\ref{GPeq}) is accompanied by the natural requirement  \(\abs{u}\to1\) as \(\abs{x}\to\infty\), and arises in quantum physics as the Hamiltonian evolution equation associated to the Ginzburg-Landau energy
\begin{equation}
\label{GLenergy}
    E_{\epsilon}(u):=\int_{\mathbb{R}^{2}}\frac{1}{2}\abs{\nabla u}^{2}+\frac{1}{4\epsilon^{2}}(1-\abs{u}^{2})^{2}\:\:dx.
\end{equation} 

In the asymptotic regime where the vortex core size \(O(\epsilon)\) is much smaller than the inter-vortex distance \(O(1)\), formal computations of Neu \cite{neu1990} have suggested that the vortex motion for (\ref{GPeq}) is related to a system of ODEs, known as the \textit{Helmholtz-Kirchhoff system}. While this relation has been clarified by several authors \cites{collianderjerrard1998,linxin1999,jerrardspirn2008,bethueljerrardsmets2008,jerrardspirn2015} in a measure-theoretic sense, there is currently an absence of explicit examples of solutions to (\ref{GPeq}) which exhibit the Helmholtz-Kirchhoff dynamics. 

The aim of this paper is to address this problem using a constructive approach, leading to the first examples of general multi-vortex solutions for which the solution profile and zero set can be tracked precisely over a non-trivial time interval. Our main results hold in the asymptotic regime mentioned above, and can be stated informally as follows:

\begin{enumerate}[label=(\roman*)]
\item Given any integer \(n\geq2\), there exists a family of dynamic \(n\)-vortex solutions to (\ref{GPeq}) which resemble at each time a product of sharply rescaled degree \(\pm1\) solutions to the Ginzburg-Landau equation, i.e. the Euler-Lagrange equation associated to the energy \(E_{1}(u)\).
\item The solutions depicted in (i) have exactly \(n\) zeroes, whose trajectories are governed at leading order as \(\epsilon\to0\) by the Helmholtz-Kirchhoff system.
\item The first correction to the leading order vortex dynamics in (ii) is determined by the solution of a linear wave equation, corresponding to a lower order adjustment of the phase far from the vortices. 
\end{enumerate}

Expanding on (i), the results presented here provide a time-dependent counterpart to a \mbox{number} of established constructions in the stationary setting \cites{delpinokowalczykmusso2006,pacardriviere2000}. The existence of dynamic solutions satisfying (i)-(ii) addresses, in particular, a problem raised by Pacard and Rivi\`{e}re \cite{pacardriviere2000}*{Section 1.3}.

Regarding (iii), lower order corrections to the leading dynamics are expected to play an important role in determining the long-time behaviour of vortices. A formal investigation of these issues has been carried out by Ovchinnikov and Sigal \cites{ovchinnikovsigal1998,ovchinnikovsigal1998longtime}, who highlighted the significance of wave-like terms describing the interaction between vortices with radiation generated by their induced motion. Our results rigorously detect these vortex-radiation interactions for the first time, providing evidence for their influence on the dynamics which allows one to go beyond the Helmholtz-Kirchhoff approximation.

Singular perturbation methods based on linearization around an approximate solution have been a powerful tool in elliptic problems for the construction of vortex configurations with fine asymptotic information \cites{pacardriviere2000,delpinokowalczykmusso2006,chironpacherie2021,aohuangliuwei2021,liuwei2020,daviladelpinomedinarodiac2022}. This paper develops the analogue of this approach for the evolution equation (\ref{GPeq}), first building an accurate expansion for a solution in powers of \(\epsilon\) and then employing linearized stability estimates to obtain suitable control on the remainder. Both elements in this process turn out to be rather delicate: the high degree of accuracy needed in the first stage, for example, requires the design of a refined improvement procedure taking into account coupled ``inner'' effects near the vortices and ``outer'' effects far from them. In the second stage, a major novelty is the introduction of an adapted quadratic form providing control on the dynamics near a general multi-vortex configuration. 

\subsection{Background and statement of the main results}\label{mainresultssubsect}Setting \(\epsilon=1\), critical points of the energy (\ref{GLenergy}) satisfy the Ginzburg-Landau equation in the plane
\begin{equation}
\label{GLeq}
    \Delta u+(1-\abs{u}^{2})u=0,\quad\text{in }\mathbb{R}^{2}.
\end{equation} 
This problem admits a \textit{standard degree-one vortex solution} of the form \(W(x)=w(r)e^{i\theta}\), where \((r,\theta)\) denote the usual polar coordinates in \(\mathbb{R}^{2}\), and \(w=w(r)\) denotes the unique solution of the problem 
\begin{equation}
\label{wODE}
\begin{gathered}
w''+\frac{1}{r}w'-\frac{1}{r^{2}}w+(1-w^2)w=0,\quad\text{in }(0,\infty),\\
    w(0)=0,\quad w(r)\to1\text{ as }r\to\infty.
\end{gathered}    
\end{equation}An important feature of the degree-one vortex is its stability and minimality properties \cites{ovchinnikovsigal1997,delpinofelmerkowalczyk2004,mironescu1995}. Mironescu \cite{mironescu1996} has proved, in addition, that any solution of (\ref{GLeq}) with degree one at infinity and \(\int_{\mathbb{R}^{2}}(1-\abs{u}^{2})^{2}\:\:dx<\infty\) must equal \(W(x)\) up to a translation and phase shift. We note that the function \(w(r)\) is positive, increasing and has the asymptotic behaviour \(w(r)\sim r\) as \(r\to0^{+}\) and \(w(r)\sim 1-1/(2r^{2})\) as \(r\to\infty\), see for example \cites{chenelliottqi1994,herveherve1994}.

By taking products of translated copies of \(W\) and \(\overline{W}\), or suitable perturbations thereof, it is easy to construct initial data for (\ref{GPeq}) with any prescribed (finite) configuration of degree \(\pm1\) vortices. A natural question then concerns the resulting dynamics of these configurations.

Neu \cite{neu1990} studied this issue in the weakly interacting regime where the typical vortex core size \(O(\epsilon)\) is much smaller than the inter-vortex distance \(O(1)\). Using a formal matched asymptotic expansion, it was proposed that vortices \(\xi_{1}(t),\ldots,\xi_{n}(t)\) evolve in time maintaining their degrees \(d_{j}=\pm1\), and that the limiting dynamics as \(\epsilon\to0\) is governed by the Helmholtz-Kirchhoff system
\begin{equation}
\label{KirchhoffODE}
    \dot{\xi}_{j}=2\sum_{k\neq j}d_{k}\frac{(\xi_{j}-\xi_{k})^{\perp}}{\abs{\xi_{j}-\xi_{k}}^{2}},\quad\quad j=1,\ldots,n.
\end{equation}Here and in what follows we denote \((x_{1},x_{2})^{\perp}:=(-x_{2},x_{1})\). Equations (\ref{KirchhoffODE}) admit a Hamiltonian formulation
\begin{equation}
\label{kirchhofffunct}
\begin{gathered}
    d_{j}\dot{\xi}_{j}=\nabla^{\perp}_{\xi_{j}}K(\xi_{1},\ldots,\xi_{n}),\quad j=1,\ldots,n,\\
    K(\xi_{1},\ldots,\xi_{n}):=2\sum_{j\neq k}d_{j}d_{k}\log\abs{\xi_{j}-\xi_{k}},
\end{gathered}    
\end{equation}and coincide with the well-known motion law for point vortices in an ideal, incompressible fluid. We remark that this law gives rise to globally defined trajectories for generic initial data, however there are exceptional examples which exhibit finite-time collisions \cite{marchioropulvirenti1994}.  

Rigorous results on Gross-Pitaevskii vortex dynamics have been obtained using measure-theoretic methods in a series of works since the late 90s, starting with Colliander and Jerrard \cite{collianderjerrard1998}, Lin and Xin \cite{linxin1999} and Jerrard and Spirn \cite{jerrardspirn2008} for solutions on a bounded domain \(\Omega\subset\mathbb{R}^{2}\) or the torus \(\mathbb{T}^{2}\). In the entire space \(\mathbb{R}^{2}\), the analysis of (\ref{GPeq}) is complicated by the presence of \textit{infinite energy} vortex configurations with non-zero total degree \(d:=d_{1}+\ldots+d_{n}\). Bethuel and Smets \cite{bethuelsmets2007} have shown that the Cauchy problem is globally well-posed in this setting, and the corresponding vortex motion has been considered by Bethuel, Jerrard and Smets \cite{bethueljerrardsmets2008}. In precise terms, the authors of \cite{bethueljerrardsmets2008} study the Jacobian (or vorticity)
\begin{equation*}
    J(u):=\det(\nabla u)=\det
    \begin{pmatrix}
        \partial_{1}u_{1} & \partial_{2}u_{1} \\
        \partial_{1}u_{2} & \partial_{2}u_{2}
    \end{pmatrix},\quad\quad u=u_{1}+iu_{2}\in\mathbb{C},
\end{equation*}associated to solutions \((u_{\epsilon})_{\epsilon>0}\) arising from concentrated \(n\)-vortex initial data, showing that this quantity converges in a suitable space of measures to a sum of Dirac masses 
\begin{equation*}
    J(u_{\epsilon}(\cdot,t))\rightharpoonup\pi\sum_{j=1}^{n}d_{j}\delta_{\xi_{j}(t)},\quad\text{as}\quad\epsilon\to0,\quad\text{for all}\quad t\in[0,T],
\end{equation*}where the limiting vortex trajectories \(\xi_{j}(t)\in\mathbb{R}^{2}\) satisfy (\ref{KirchhoffODE}). Quantitative estimates for the Jacobian which describe the vortex motion for a small, but fixed, value of the parameter \(\epsilon>0\) were later obtained by Jerrard and Spirn \cite{jerrardspirn2015}.

It is relevant to note that the aforementioned rigorous results do not provide information on the solution profile near the vortex cores, a key feature of Neu's original asymptotics, nor the precise character of the corrections to the leading order dynamics. The approach taken in this work is capable, by contrast, of giving a complete description of these details. 

Our first main result asserts the existence of \(n\)-vortex solutions resembling a sharp rescaling of \(W\) or \(\overline{W}\) around each vortex core, with asymptotic vortex dynamics determined by (\ref{KirchhoffODE}).
Let us fix an integer \(n\geq2\), degrees \(d_{1},\ldots,d_{n}\in\{1,-1\}\) and a smooth solution \(\xi^{0}(t)=(\xi_{1}^{0}(t),\ldots,\xi_{n}^{0}(t))\) of (\ref{KirchhoffODE}) with no collisions in \([0,T]\), i.e.
\begin{equation*}
    \min_{t\in[0,T]}\:\abs{\xi_{j}^{0}(t)-\xi_{k}^{0}(t)}>0,\quad\text{for all}\quad j\neq k.
\end{equation*}
We let \(I_{\pm}\) denote the set of indices \(j\) such that \(d_{j}=\pm1\). Then the following holds.

\begin{theorem}
\label{mainthm}
    For points \(\xi_{1}^{0}(t),\ldots,\xi_{n}^{0}(t)\) as above and sufficiently small \(\epsilon>0\), there exists a smooth solution to (\ref{GPeq}) of the form
    \begin{equation}
    \label{mainthmexp}
        u_{\epsilon}(x,t)=\prod_{j\in I_{+}}W\bigg(\frac{x-\xi_{j}(t)}{\epsilon}\bigg)\prod_{j\in I_{-}}\overline{W}\bigg(\frac{x-\xi_{j}(t)}{\epsilon}\bigg)+\phi_{\epsilon}(x,t),
    \end{equation}where \(\xi_{j}(t)\), $j=1,\ldots,n$, denote smooth functions satisfying
    \begin{equation*}
        \xi_{j}(t)=\xi_{j}^{0}(t)+O(\epsilon^{2}\abs{\log{\epsilon}}^{2}),\quad\text{for}\quad t\in[0,T],
    \end{equation*}and \(\phi_{\epsilon}(x,t)\) is a remainder term such that
    \begin{equation*}
        \abs{\phi_{\epsilon}(x,t)}+\epsilon\abs{\nabla\phi_{\epsilon}(x,t)}\leq C\epsilon^{2}\abs{\log\epsilon}^{2},\quad\text{for all}\quad(x,t)\in\mathbb{R}^{2}\times[0,T].
    \end{equation*}We can further estimate
    \begin{equation*}
        \int_{\mathbb{R}^{2}}\epsilon^{2}\abs{\nabla\phi_{\epsilon}}^{2}+
        \sum_{j=1}^n\frac{{\abs{\phi_{\epsilon}}}^{2}}{ \epsilon^{2}+\abs{x-\xi_j(t)}^{2}}\:\:dx   \leq C\epsilon^4\abs{\log\epsilon}^{5},\quad\text{for all}\quad t\in[0,T].
    \end{equation*} 
\end{theorem}

A standard application of the implicit function theorem shows that the solutions in Theorem \ref{mainthm} have exactly \(n\) zeroes at each time, located at points \(O(\epsilon^{2}\abs{\log\epsilon}^{2})\) close to \(\xi_{j}^{0}(t)\) for \(j=1,\ldots,n\). 
As alluded to earlier, these solutions also have infinite energy if the total degree is non-zero. This is related to the corresponding property \(E_{1}(W)=+\infty\) of the degree-one vortex, which holds due to the slow \(O(r^{-1})\) decay of the angular gradient. 

In the next result, we present a refined expansion for the phase of the solutions constructed in Theorem \ref{mainthm}. We identify a small adjustment to the leading order phase away from the vortex cores as the solution of a linear wave equation, validating a formal expansion found by Ovchinnikov and Sigal \cite{ovchinnikovsigal1998}. It is convenient to introduce the following notation: for a smooth right-hand side \(F=F(x,\tau)\) expressed in terms of the rescaled time variable \(\tau:=\sqrt{2}\epsilon^{-1}t\), we write \(\psi=\Box^{-1}F\) for the unique solution of the problem
\begin{equation}
\label{linearwaveeq}
-\partial_{\tau}^{2}\psi+\Delta_{x}\psi=F,\quad\text{in}\quad\mathbb{R}^{2}\times(0,\sqrt{2}\epsilon^{-1}T),
\end{equation}with zero initial data \((\psi,\partial_{\tau}\psi)(\cdot,0)=(0,0)\). 

\begin{theorem}
\label{phasethm}
Let \(u_{\epsilon}(x,t)\) be the solution of (\ref{GPeq}) described in Theorem \ref{mainthm}, and write \(u_{\epsilon}=\abs{u_{\epsilon}}e^{i\varphi_{\epsilon}}\). Let \(\delta:=\tfrac{1}{4}\min\{\abs{\xi_{j}^{0}(t)-\xi_{k}^{0}(t)}:t\in[0,T],j\neq k\}>0\). In the region \(\abs{x-\xi_{j}(t)}\geq2\delta\) for all \(j=1,\ldots,n\),    we can express 
\begin{equation}
\label{phasethmexp}
    \varphi_{\epsilon}(x,t)=\sum_{j=1}^{n}d_{j}\theta\big(x-\xi_{j}(t)\big)+\psi_{1}^{out,1}(x,\tau)+O(\epsilon^{3-\sigma}),
\end{equation}for arbitrarily small \(\sigma>0\), where \(\theta(x)\) denotes the polar angle of \(x\in\mathbb{R}^{2}\), and \(\psi_{1}^{out,1}\) denotes a correction of size \(O(\epsilon^{2}\abs{\log\epsilon}^{2})\). More precisely, we have \(\psi_{1}^{out,1}=\Box^{-1}F^{out,1}\) where
\begin{equation}
\label{Foutexp}
    F^{out,1}=\frac{1}{2}\epsilon^{2}\chi\sum_{j=1}^{n}d_{j}\frac{(x-\xi_{j})^{\perp}}{\abs{x-\xi_{j}}^{2}}\cdot(-\ddot{\xi}_{j})+O_{c}(\epsilon^{2}\abs{\log\epsilon})+\text{lower order terms}.
\end{equation}
\end{theorem} 

In the statement above, \(\chi\) is a smooth cut-off function vanishing in a \(\delta\)-neighbourhood of the points \(\xi_{j}(t)\), \(j=1,\ldots,n\), and equal to one if the distance to these points is greater than \(2\delta\). The notation \(O_{c}(\epsilon^{2}\abs{\log\epsilon})\) denotes terms of size \(O(\epsilon^{2}\abs{\log\epsilon})\) that are compactly supported in the regions \(\abs{x-\xi_{j}(t)}\leq2\delta\) for \(j=1,\ldots,n\), and the ``lower order terms'' decay at least like \(O(\epsilon^{2}\abs{x}^{-2})\) as \(\abs{x}\to\infty\). We refer the reader to (\ref{Fout1def}) for the precise definition of \(F^{out,1}\).

The function \(\psi_{1}^{out,1}(x,\tau)\) in (\ref{phasethmexp}) can be interpreted as radiation generated by vortex motion. Our final result demonstrates that this radiation gives rise to small corrections to the Helmholtz-Kirchhoff dynamics (\ref{KirchhoffODE}), as predicted in \cite{ovchinnikovsigal1998}.

\begin{theorem}
\label{kirchhoffcorrthm}
The zeroes of the solution \(u_{\epsilon}(x,t)\) described in Theorem \ref{mainthm} are located at points 
\begin{equation*}
    {\xi}_{j}^{*}(t)+O(\epsilon^{3-\sigma}),\quad\text{for}\quad j=1,\ldots,n,
\end{equation*}where \({\xi}^{*}(t)=({\xi}_{1}^{*}(t),\ldots,{\xi}_{n}^{*}(t))\) denotes the unique solution of the system
\begin{equation}
\label{kirchhoffcorrsystem}
    \dot{\xi}_{j}^{*}=2\sum_{k\neq j}d_{k}\frac{(\xi_{j}^{*}-\xi_{k}^{*})^{\perp}}{\abs{\xi_{j}^{*}-\xi_{k}^{*}}^{2}}+2\nabla\psi_{1}^{out,1}\big(\xi_{j}^{0}(t),\tau;\xi^{0}\big),\quad j=1,\ldots,n,
\end{equation}with initial data \(\xi_{j}^{*}(0)=\xi_{j}^{0}(0)\), \(j=1,\ldots,n\). Here
\begin{equation*}
\psi_{1}^{out,1}(x,\tau;\xi^{0}):=\Box^{-1}F^{out,1}(x,\tau;\xi^{0}),
\end{equation*}where \(F^{out,1}(x,\tau;\xi^{0})\) denotes the function (\ref{Foutexp}) with parameters fixed at \(\xi(t)=\xi^{0}(t)\).
\end{theorem}

Equations (\ref{kirchhoffcorrsystem}) represent a modification of the Helmholtz-Kirchhoff system with small ``forcing terms'' determined by \(\nabla\psi^{out,1}\). The precise characterization of these lower order corrections appears to be new, and presumably constitutes an important element towards understanding the long-time behaviour of vortices. In this direction, Ovchinnikov and Sigal \cite{ovchinnikovsigal1998longtime} have argued that the radiation emitted by two degree \(+1\) vortices leads to spiral dynamics, with the two vortices rotating around a common centre and moving apart at a rate \(O(t^{1/6})\) as \(t\to\infty\). A rigorous description of this phenomenon has so far been elusive: one expects that the analysis of (\ref{kirchhoffcorrsystem}) will be necessary since all solutions of the unperturbed system (\ref{KirchhoffODE}) are periodic for \(d_{1}=d_{2}=1\).

The proof of our main results can be divided into two key steps. First, we construct an \mbox{accurate} approximate solution to (\ref{GPeq}) which satisfies the equation up to an \(H^{1}(\mathbb{R}^2)\)-error given by a large power of \(\epsilon\). This is done by inverting the elliptic linearized operator around the degree-one \mbox{vortex} near the vortex cores, and by solving wave equations of the form (\ref{linearwaveeq}) to correct the \mbox{error} in the outer region. Sufficient decay for the elliptic problem can only be obtained if certain \mbox{orthogonality} conditions are satisfied: this ultimately determines how the vortex parameters \(\xi(t)\) must be \mbox{adjusted}. 

In the second step, energy estimates are used to control the remainder of the approximation. Ruling out exponentially large terms in \(\epsilon\) requires a linear stability analysis where a particular quadratic form plays a crucial role. Finding a suitable expression for this functional, and proving its positivity, are nontrivial issues we have to address. 

\subsection{Related literature}

Before entering the core of the paper, we make some further \mbox{comments} on connections to previous work.

\begin{itemize}[leftmargin=*]
    \item There is a known relation \cites{jerrardspirn2015,serfaty2017} between the evolution of vortices for the Gross-Pitaevskii equation (\ref{GPeq}) and the dynamics of Euler flows for an ideal, incompressible fluid. A result analogous to Theorem \ref{mainthm} has been obtained for the two-dimensional Euler equations in \cite{daviladelpinomussowei2020}, where \(n\)-vortex solutions are found resembling a superposition of scaled, finite mass solutions of Liouville's equation \(\Delta u+e^{u}=0\). While this construction shares some similarities with the present paper - indeed, the leading order vortex dynamics is described by the same Helmholtz-Kirchhoff system (\ref{KirchhoffODE}) - building vortex solutions to (\ref{GPeq}) represents a more difficult problem. The intuitive reason for this can be traced back to the infinite energy property of the standard degree-one vortex \(W(x)=w(r)e^{i\theta}\), which manifests itself in the slow decay of the initial error of approximation. 
    
    \medskip

    \item In the setting of the two-dimensional Navier-Stokes equations, all vortex structures are eventually destroyed by diffusion. Nevertheless, an interesting subject of investigation concerns the vanishing viscosity limit for solutions arising from point vortex initial data. Precise asymptotics for weakly viscous \(n\)-vortex flows were obtained by Gallay \cite{gallay2011} on time intervals \([0,T]\) independent of the viscosity. More recently, Dolce and Gallay \cite{dolcegallay2024} have studied the long-time dynamics of the viscous vortex dipole.

    \medskip

    \item The gradient flow of the energy (\ref{GLenergy}) corresponds to the  parabolic Ginzburg-Landau equation
    \begin{equation*}
        u_{t}=\Delta u+\frac{1}{\epsilon^{2}}(1-\abs{u}^{2})u,
    \end{equation*}whose vortex motion has been studied extensively within the measure-theoretic framework: see \cites{lin1996,jerrardsoner1998,linxin1999heatflow,sandierserfaty2004} for results addressing the limiting dynamics as \(\epsilon\to0\), and \cites{bethuelorlandismets2005,bethuelorlandismets2007,bethuelorlandismets2007multipledegree,serfaty2007no1,serfaty2007no2} for results handling the issue of vortex collisions. There has also been interest in the ``mixed-flow'' problem
    \begin{equation*}
    (a+ib)u_{t}=\Delta u+\frac{1}{\epsilon^{2}}(1-\abs{u}^{2})u,
    \end{equation*}where \(a,b\in\mathbb{R}\). We refer the reader to \cites{kurzkemelchermoserspirn2009,miot2009} for a description of the leading order vortex dynamics in this case. 

    \medskip

    \item On the unit disc, time periodic rotating vortex solutions to the Gross-Pitaevskii equation have been constructed in \cite{venkatraman2017} using elliptic methods.
    
    \medskip 

    \item Recent progress on the dynamic stability of the standard degree-one vortex has been made by several authors. We mention that Gravejat, Pacherie and Smets \cites{gravejatpacheriesmets2022} have established the orbital stability of \(W(x)\) under the Gross-Pitaevskii evolution. See also \cites{collotgermainpacherie2025,luhrmannschlagshahshahani2025,palaciospusateri2024} for linearized estimates obtained using the distorted Fourier transform. 
    
\end{itemize}

\section{Strategy of the construction}  

\subsection{Construction of an accurate approximate solution}
\label{approxsolsubsection}
For a smooth, complex-valued \mbox{function} \(u=u(y,t)\) expressed in terms of the rescaled space variable \(y:=\epsilon^{-1}x\), we consider the nonlinear operator
\begin{equation}
\label{Soperator}
    S(u):=\epsilon^{2}iu_{t}+\Delta_{y} u+(1-\abs{u}^{2})u.
\end{equation}The aim of the rest of the paper is to construct a solution of \(S(u)=0\) with the profile described in Theorem \ref{mainthm}; in other words, given a smooth, collisionless solution 
\begin{equation*}
    \xi^{0}(t)=\big(\xi^{0}_{1}(t),\ldots,\xi^{0}_{n}(t)\big): [0,T]\to\mathbb{R}^{2n}
\end{equation*}of (\ref{KirchhoffODE}) with degrees \(d_{1},\ldots,d_{n}\in\{1,-1\}\), we seek a function \(u=u(y,t)\) satisfying \(S(u)=0\) with the approximate form
\begin{equation}
\label{uapprox}
   u(y,t)\approx\prod_{j\in I_{+}}W\big(y-\epsilon^{-1}\xi_{j}(t)\big)\prod_{j\in I_{-}}\overline{W}\big(y-\epsilon^{-1}\xi_{j}(t)\big)
\end{equation}as \(\epsilon\to0\), where 
\begin{equation*}
    \xi_{j}(t)\approx \xi_{j}^{0}(t),\quad\text{for all}\quad t\in[0,T],\quad j=1,\ldots,n.
\end{equation*}We recall that \(\overline{W}\) denotes the complex conjugate of \(W\), and \(I_{\pm}\) denotes the set of indices \(j\) such that \(d_{j}=\pm1\).

Our first objective is to build an accurate expansion for an \(n\)-vortex solution which substantially improves the approximation (\ref{uapprox}). This refinement is desirable from a descriptive point of view (i.e. to detect lower-order corrections to the dynamics); more crucially, we will need a high level of accuracy to close the final argument using energy estimates.

\medskip

\textit{Ansatz for the approximation.} It is convenient to decompose the expansion as a product of ``inner'' functions expressed in the translated variable
\begin{equation}
    y_{j}:=y-\tilde{\xi}_{j}(t),\quad\quad\tilde{\xi}_{j}(t)=\epsilon^{-1}\xi_{j}(t),
\end{equation}and an ``outer'' correction expressed in the original variables \((x,t)\). More precisely, for functions \(\Phi^{in}=(\Phi_{1},\ldots,\Phi_{n})\) of the form \(\Phi_{j}=\Phi_{j}(y_{j},t)\), we make the ansatz
\begin{equation}
\label{addmultansatz}
    u(y,t)\approx e^{i\psi^{out}}\prod_{j=1}^{n}\bigg(\eta_{j}(W_{j}+\Phi_{j})+(1-\eta_{j})W_{j}e^{\Phi_{j}/W_{j}}\bigg)
\end{equation}for an accurate approximate solution, where \(\psi^{out}=\psi^{out}(\epsilon y,t)\) denotes the ``outer'' correction alluded to above, and
\begin{equation}
\label{Wjdef}
   W_{j}(y,t):=\begin{cases}
    W(y_{j}),\quad\text{for }j\in I_{+},\\
    \overline{W}(y_{j}),\quad\text{for }j\in I_{-},
   \end{cases}
\end{equation}denotes the leading-order profile around the \(j\)th vortex. Taking \(\eta_{j}\) a smooth cut-off function such that \(\eta_{j}=1\) for \(\abs{y_{j}}\leq1\) and \(\eta_{j}=0\) for \(\abs{y_{j}}\geq2\), the term in parentheses in (\ref{addmultansatz}) interpolates between the additive expression \(W_{j}+\Phi_{j}\) very close to the points \(\tilde{\xi}_{j}(t)\), and the multiplicative expression \(W_{j}e^{\Phi_{j}/W_{j}}\) away from the vortex cores. The advantage of this setup is that error terms associated to the multiplicative ansatz are easier to deal with for \(\abs{y_{j}}\) large, while the additive ansatz is needed to avoid singularities near the zero set.

We localize each ``inner'' correction in a ball of radius \(2\delta\epsilon^{-1}\) around \(\tilde{\xi}_{j}(t)\) by  writing \(\Phi_{j}=\tilde{\eta}_{j}\phi_{j}\) for another smooth cut-off \(\tilde{\eta}_{j}\), namely
\begin{equation*}
    \tilde{\eta}_{j}(y,t):=\eta_{0}(\epsilon\delta^{-1}y_{j}),\quad\quad\eta_{0}(y)=\begin{cases}
        1,\quad\text{if }\abs{y}\leq1,\\
        0,\quad\text{if }\abs{y}\geq2,
    \end{cases}
\end{equation*}where \(\delta>0\) is the number defined by
\begin{equation}
\label{deltadef}
    \delta:=\frac{1}{4}\min\bigg\{\abs{\xi_{j}^{0}(t)-\xi_{k}^{0}(t)}:t\in[0,T],j\neq k\bigg\}.
\end{equation}Our aim is then to find parameters \(\xi(t)\approx\xi^{0}(t)\) and functions \((\phi^{in},\psi^{out})=(\phi_{1},\ldots,\phi_{n},\psi^{out})\) such that the approximation
\begin{equation}
\label{ansatzforapprox}
u_{*}=u_{*}\big(y,t;\xi,\phi^{in},\psi^{out}\big):=e^{i\psi^{out}}\prod_{j=1}^{n}\bigg(\eta_{j}(W_{j}+\tilde{\eta}_{j}\phi_{j})+(1-\eta_{j})W_{j}e^{\tilde{\eta}_{j}\phi_{j}/W_{j}}\bigg)
\end{equation}gives rise to a small, finite error \(S(u_{*})\) in \(H^{1}(\mathbb{R}^{2})\)-norm. This will be done by solving elliptic equations for the functions \(\phi_{j}\), and by inverting the linear wave operator for \(\psi_{1}^{out}=\operatorname{Re}(\psi^{out})\). (Here and in what follows, we write \(\operatorname{Re}(\psi)\) and \(\operatorname{Im}(\psi)\) for the real and imaginary parts of a complex number \(\psi\)).

\medskip

\textit{The elliptic equation for the inner corrections.} Inserting the ansatz (\ref{ansatzforapprox}) into the operator \(S(u)\) and imposing that the error of the initial approximation (\ref{uapprox}) is eliminated at main order, we will derive equations for \(\phi_{j}\) which essentially take the form
\begin{equation}
\label{timedepinnereq}
    \epsilon^{2}i\partial_{t}\phi_{j}+L_{j}[\phi_{j}]=H_{j}(y,t;\xi,\psi^{out}),\quad\text{for}\quad\abs{y_{j}}\leq2\delta\epsilon^{-1},
\end{equation}where \(H_{j}=H_{j}(y,t;\xi,\psi^{out})\) is a smooth function of all its arguments, and
\begin{equation}
\label{Ljdef}
    L_{j}[\phi_{j}]:=\Delta_{y}\phi_{j}+(1-\abs{W_{j}}^{2})\phi_{j}-2\operatorname{Re}(\overline{W}_{j}\phi_{j})W_{j}
\end{equation}denotes the elliptic linearized operator around \(W_{j}\). Fixing ideas with \(\psi^{out}=0\), an approximate solution to (\ref{timedepinnereq}) can be found by neglecting the time-derivative term (which is formally of ``lower order'' due to the factor of \(\epsilon^{2}\)). The construction of an accurate \(n\)-vortex expansion then proceeds, in the region where \(\tilde{\eta}_{j}\) is supported, by solving elliptic equations of the form
\begin{equation}
\label{innereq}
    L_{j}[\phi_{j}]=H_{j}(y,t;\xi).
\end{equation}

To make our argument work, it is crucial that the functions \(\phi_{j}\) obtained via (\ref{innereq}) do not become too large away from the vortex cores, i.e. as \(\abs{y_{j}}\to\infty\): excessive growth of the inner corrections would create error terms in the outer region which cannot be eliminated with a suitable estimate for \(\psi^{out}\). The choice of parameters \(\xi_{j}(t)\) is closely tied to this issue; in particular, we ensure the necessary control by adjusting these points appropriately to avoid certain terms obstructing the invertibility theory.
\medskip 

\textit{The wave equation for the outer correction.} Far from the vortices the ansatz (\ref{ansatzforapprox}) reads
\begin{equation*}
    u_{*}=e^{i\psi^{out}}\prod_{j=1}^{n}W_{j},\quad\quad\psi^{out}=\psi_{1}^{out}+i\psi_{2}^{out},
\end{equation*}where \(\psi_{1}^{out}:=\operatorname{Re}(\psi^{out})\) and \(\psi_{2}^{out}:=\operatorname{Im}(\psi^{out})\) denote lower order corrections to the outer phase and modulus respectively. The linearized equations for \(\psi^{out}\) in this region resemble the system
\begin{align}
    -\epsilon^{2}\partial_{t}\psi_{2}^{out}+\epsilon^{2}\Delta_{x}\psi_{1}^{out}+E_{1}^{out}&=0,\label{outeq1}\\
    \epsilon^{2}\partial_{t}\psi_{1}^{out}+\epsilon^{2}\Delta_{x}\psi_{2}^{out}-2\psi_{2}^{out}+E_{2}^{out}&=0,\label{outeq2}
\end{align}for the real and imaginary parts, where \(E_{1}^{out}\), \(E_{2}^{out}\) denote smooth, real-valued functions of the form
\begin{equation*}
  E_{1}^{out}=E_{1}^{out}(x,t;\xi,\phi^{in}),\quad E_{2}^{out}=E_{2}^{out}(x,t;\xi,\phi^{in}).  
\end{equation*}Neglecting the term \(\epsilon^{2}\Delta_{x}\psi_{2}^{out}\) in (\ref{outeq2}) (which is \textit{a posteriori} of smaller order), we can write 
\begin{equation}
\label{psi2intermspsi1intro}
    \psi_{2}^{out}=\tfrac{1}{2}\left(\epsilon^{2}\partial_{t}\psi_{1}^{out}+E_{2}^{out}\right).
\end{equation}Substituting (\ref{psi2intermspsi1intro}) into (\ref{outeq1}), and writing all functions involved in terms of the rescaled time variable \(\tau:=\sqrt{2}\epsilon^{-1}t\), we obtain the wave equation (\ref{linearwaveeq}) where \(\psi=\psi_{1}^{out}\) and \(F=F^{out}\) is a linear combination of \(E_{1}^{out}\) and \(\partial_{\tau}E_{2}^{out}\). 

An iteration scheme based on this line of reasoning will be implemented, in a precise manner, to obtain successive improvements of the outer approximation in powers of \(\epsilon\). At each stage we solve an equation of the form (\ref{linearwaveeq}) for \(\psi_{1}^{out}\), with \(\psi_{2}^{out}\) then determined algebraically via (\ref{psi2intermspsi1intro}). One notable difficulty we face concerns the elimination of ``radiation terms'' in the error created by preceding wave corrections, which typically decay slowly (or not at all) inside the light cone \(\abs{x}\leq\tau\). We ensure a marginal \(\epsilon^{1/2}\) improvement with each iteration using a delicate combination of pointwise and energy estimates, together with the observation that the largest term in the new error sees only the angular part of the gradient of \(\psi_{1}^{out}\) (which has a good space decay rate).

\subsection{Energy estimates}\label{energyestimatesoutline}

Once an accurate approximate solution \(u_{*}\) has been constructed, the second part of the paper is devoted to energy estimates for the remainder. More precisely, we expand 
\begin{equation}
\label{fullprobexp}
    S(u_{*}+\phi)=S(u_{*})+S'(u_{*})[\phi]+\mathcal{N}(\phi)
\end{equation}where \(S(u_{*})\) denotes the error of the approximation, 
\begin{equation}
\label{Sprimedef}
    S'(u_{*})[\phi]:=\epsilon^{2}i\phi_{t}+\Delta_{y}\phi+(1-\abs{u_{*}}^{2})\phi-2\operatorname{Re}(\overline{u}_{*}\phi)u_{*}
\end{equation}denotes the linearized operator around \(u_{*}\), and 
\begin{equation}
    \mathcal{N}(\phi):=-2\operatorname{Re}(\overline{u}_{*}\phi)\phi-\abs{\phi}^{2}u_{*}-\abs{\phi}^{2}\phi
\end{equation}denotes nonlinear terms, with the intention of showing that the \(\phi\) for which (\ref{fullprobexp}) vanishes (with \(\phi(y,0)=0\)) is uniformly small on \([0,T]\).

A key step for this purpose is to obtain estimates for the linear problem
\begin{equation}
\label{fulllinearizedprob}
\begin{cases}
S'(u_{*})[\phi]=f(y,t),&\quad\text{in}\quad\mathbb{R}^{2}\times[0,T],\\
    \hspace{1.1em}\phi(y,0)=0,&\quad\text{in}\quad\mathbb{R}^{2},  
    \end{cases}
\end{equation}which we achieve via the construction of a carefully designed quadratic form (i.e. energy functional). To motivate its definition, we note that the unscaled (i.e. \(\epsilon=1\)) Gross-Pitaevskii equation (\ref{GPeq}) has three ``formal'' conserved quantities, namely the energy
\(\int_{\mathbb{R}^{2}}\frac{1}{2}\abs{\nabla u}^{2}+\frac{1}{4}(1-\abs{u}^{2})^{2}\), the momentum \(\int_{\mathbb{R}^{2}}\operatorname{Re}(i\nabla u\overline{u})\),
and the mass \(\int_{\mathbb{R}^{2}}\abs{u}^{2}\). Modifying the momentum by a vector field \(A:\mathbb{R}^{2}\times[0,T]\to\mathbb{R}^{2}\), and the mass by a real-valued function \(B:\mathbb{R}^{2}\times[0,T]\to\mathbb{R}\), it is convenient to consider the functional
\begin{equation*}
    \mathcal{F}(u):=\int_{\mathbb{R}^{2}}\frac{1}{2}\abs{\nabla u}^{2}+\frac{1}{4}(1-\abs{u}^{2})^{2}+\int_{\mathbb{R}^{2}}\frac{1}{2}\operatorname{Re}(iA\cdot\nabla u\overline{u})+\int_{\mathbb{R}^{2}}\frac{1}{2}B\abs{u}^{2}.
\end{equation*}
Note that if \(A\), \(B\) are constant then \(\mathcal{F}(u)\) is another formal conserved quantity. Our approach involves 
choosing  \(A\), \(B\) not exactly constant but so that \(\mathcal{F}(u)\) can still be thought of as an ``approximate'' conserved quantity for the Gross-Pitaevskii equation. The key quadratic form for our linear analysis then corresponds to the second variation of \(\mathcal{F}(u)\) at \(u_{*}\), namely
\begin{equation}
\label{quadform}
    D^{2}\mathcal{F}(u_{*})[\phi,\phi]=\int_{\mathbb{R}^{2}}\abs{\nabla\phi}^{2}-(1-\abs{u_{*}}^{2})\abs{\phi}^{2}+2\operatorname{Re}(\overline{u}_{*}\phi)^{2}+\operatorname{Re}(iA\cdot\nabla\phi\overline{\phi})+B\abs{\phi}^{2}.
\end{equation} 

A formal reasoning that goes back to Arnold \cite{arnold1965} (see also \cite{gallaysverak2024}*{Section 1A}) leads us to expect 
 that (\ref{quadform}) is an approximate conserved quantity for the linearized equation
 \begin{equation*}
     S'(u_{*})[\phi]=0\quad\text{in}\quad\mathbb{R}^{2}\times[0,T] \end{equation*}
 provided \(u_{*}\) is an approximate critical point of \(\mathcal{F}(u)\). The latter condition amounts to the first variation of \(\mathcal{F}\) vanishing at \(u_{*}\), where
\begin{equation*}
    D\mathcal{F}(u_{*})[\phi]=\operatorname{Re}\int_{\mathbb{R}^{2}}\big(-\Delta u_{*}-(1-\abs{u_{*}}^{2})u_{*}+iA\cdot\nabla u_{*}+\tfrac{1}{2}i\operatorname{div}(A)u_{*}+Bu_{*}\big)\overline{\phi}.
\end{equation*} Imposing the relation \(D\mathcal{F}(u_{*}) = 0\), and using the property 
\begin{equation*}
    S(u_{*})=\epsilon^{2}i\partial_{t}u_{*}+\Delta_{y} u_{*}+(1-\abs{u_{*}}^{2})u_{*}\approx0,
\end{equation*}we find that \(A\), \(B\) should be chosen so that
\begin{equation}\label{AandBeq}
    \epsilon^{2}i\partial_{t}u_{*}+iA\cdot\nabla u_{*}+\tfrac{1}{2}i\operatorname{div}(A)u_{*}+Bu_{*}=0.
\end{equation}Writing \(u_{*}\) in polar form \(u_{*}=\rho e^{i\varphi}\),  equation (\ref{AandBeq}) can be expressed equivalently as
\begin{align}
    \epsilon^{2}\partial_{t}(\rho^{2})+\operatorname{div}(\rho^{2}A)&=0,\label{Arhoeq}\\ 
    \epsilon^{2}\partial_{t}\varphi+A\cdot\nabla\varphi&=B.\label{Bphieq}
\end{align}

We will find a smooth vector field \(A(y,t)\) so that (\ref{Arhoeq}) is satisfied at main order, with the scalar function \(B(y,t)\) then determined via a small adjustment of (\ref{Bphieq}). More precisely, our choice of \(A\) takes the simple form \(A(y,t):=\epsilon\dot{\xi}_{j}(t)\) close to \(\xi_{j}(t)\), with \(A(y,t):=0\) far from all vortices. In a certain intermediate region, \(A(y,t)\) is defined so that it interpolates smoothly between \(\epsilon\dot{\xi}_{j}(t)\) and \(0\) while keeping the \(\operatorname{div}(\rho^{2}A)\) term small. 

With \(A\), \(B\) constructed in this way, we will prove a lower bound for the quadratic form (\ref{quadform}) for functions satisfying finitely many orthogonality conditions. The addition of an \(\epsilon^{4}\)-weighted \(\int\abs{\phi}^{2}\) term and finite-dimensional contributions then allows us to construct a positive quantity \(\mathcal{Q}(t)\) for (\ref{fulllinearizedprob}) which controls the \(H^{1}(\mathbb{R}^{2})\)-norm of \(\phi(\cdot,t)\) and which satisfies a differential inequality from which Gr\"{o}nwall's inequality can be applied. The ultimate conclusion is an estimate for (\ref{fulllinearizedprob}) which loses a large but finite power of \(\epsilon\) with respect to the right-hand side \(f(y,t)\), and this suffices to control the remainder in the nonlinear problem \(S(u_{*}+\phi)=0\).

\subsection{Outline of the paper}

The organization of the paper is as follows. In Section \ref{ellipticlinearizedsect}, we recall some key results concerning the operator (\ref{Ljdef}) and the linear elliptic equation (\ref{innereq}). These facts are first used in Section \ref{Firstimprovementsect}, where we begin the construction of an accurate \(n\)-vortex expansion. We explain in Section \ref{arbitraryapproxsect} how to iterate the ideas in Section \ref{Firstimprovementsect} to construct an approximation (\ref{ansatzforapprox}) satisfying \(S(u_{*})=O(\epsilon^{m})\) in \(H^{1}(\mathbb{R}^{2})\) for an arbitrary integer \(m\geq1\). The proofs of some relevant estimates for the wave equation are contained in Section \ref{waveestimatesect}. Reserved for Section \ref{linearfullprobsect} is the analysis of the linearized equation (\ref{fulllinearizedprob}) and the quadratic form (\ref{quadform}). Finally in \mbox{Section \ref{solvefullprobsect}} we consider the full problem \(S(u_{*}+\phi)=0\) to complete the proof of Theorems \ref{mainthm}, \ref{phasethm} and \ref{kirchhoffcorrthm}.

\section{The elliptic linearized operator around the degree-one vortex}
\label{ellipticlinearizedsect}

For \(j\in I_{+}\), the operator (\ref{Ljdef}) corresponds to the elliptic linearized operator around the degree-one vortex:
\begin{equation}
\label{ellipticlinop}
    L[\phi]:=\Delta_{y}\phi+(1-\abs{W}^{2})\phi-2\operatorname{Re}(\overline{W}\phi)W.
\end{equation}

The aim of this section is to state some relevant facts about (\ref{ellipticlinop}). We recall, in particular, the bounded functions in the kernel of \(L\), and the key elements of the linear theory for \(L[\phi]=h\) developed in our previous work \cite{delpinojunemanmusso2025}.

\subsection{Kernel of \texorpdfstring{\(L\)}{L} and positivity of the associated quadratic form}\label{onevortexqfsubsect}
The invariance of the Ginzburg-Landau energy (\ref{GLenergy}) under translations and phase shifts implies that \(e^{i\alpha}W(\cdot-\xi)\) satisfies (\ref{GLeq}) for all \(\alpha\in\mathbb{R}\) and \(\xi\in\mathbb{R}^{2}\). Differentiating this equation with respect to the parameters at \((\alpha,\xi)=(0,0)\), we find 
\begin{equation}
\label{kernel1}
    L[iW]=L\left[\frac{\partial W}{\partial y_{1}}\right]=L\left[\frac{\partial W}{\partial y_{2}}\right]=0.
\end{equation}Let \((r,\theta)\) denote the usual polar coordinates in \(\mathbb{R}^{2}\), and recall that \(w(r)=\abs{W}\) denotes the modulus of the degree-one vortex. We can then write the functions in (\ref{kernel1}) as follows: 
\begin{align}
\label{kernel2}
\begin{split}
    iW&=iW\big(1+i0\big),\\
    \frac{\partial W}{\partial y_{1}}&=iW\bigg(-\frac{1}{r}\sin\theta-i\frac{w'}{w}\cos\theta\bigg),\\
    \frac{\partial W}{\partial y_{2}}&=iW\bigg(+\frac{1}{r}\cos\theta-i\frac{w'}{w}\sin\theta\bigg).
\end{split}
\end{align} 

An important property of \(L\) is that the functions (\ref{kernel2}) and their linear combinations represent all bounded solutions of \(L[\phi]=0\) in \(\mathbb{R}^{2}\) \cite{pacardriviere2000}*{Theorem 3.2}. We mention that the corresponding quadratic form is also nonnegative \cite{delpinofelmerkowalczyk2004}: we have
\begin{equation}
\label{basicquadform}
\int_{\mathbb{R}^{2}}\abs{\nabla\phi}^{2}-(1-\abs{W}^{2})\abs{\phi}^{2}+2(\operatorname{Re}(\overline{W}\phi))^{2}\:\:dx\geq0   
\end{equation}for all perturbations \(\phi\) in the energy space
\begin{equation}
\label{energyspace}
    \left\{\phi\in H^{1}_{\text{loc}}(\mathbb{R}^{2}):\int_{\mathbb{R}^{2}}\abs{\nabla\phi}^{2}+(1-\abs{W}^{2})\abs{\phi}^{2}+(\operatorname{Re}(\overline{W}\phi))^{2}\:\:dx<\infty\right\},
\end{equation}with equality in (\ref{basicquadform}) if and only if \(\phi=c_{1}\partial_{{1}}W+c_{2}\partial_{{2}}W\) for some \(c_{1}\), \(c_{2}\in\mathbb{R}\). Note that \(\phi=iW\) is excluded in the latter scenario since \(iW\) does not belong to (\ref{energyspace}).

\subsection{The linear problem with right-hand side} 
\label{linprobsect}

As suggested by (\ref{kernel2}), the linear problem \(L[\phi]=h\) admits a convenient reformulation after factoring out \(iW\). We write \(\phi=iW\psi\) where \(\psi=\psi_{1}+i\psi_{2}\), and \(h=iW\elliprhs\). Then the equation \(L[\phi]=h\) takes the form
\begin{equation}
\label{ellipticlinearpsi}
    \Delta\psi+2\frac{w'}{w}\partial_{r}\psi+\frac{2i}{r^{2}}\partial_{\theta}\psi-2iw^{2}\psi_{2}=\elliprhs.
\end{equation}

In our previous work \cite{delpinojunemanmusso2025}, we developed a solvability theory for (\ref{ellipticlinearpsi}) by \mbox{decomposing} in Fourier modes. We write
\begin{equation}
\label{fmodesum}
\psi=P^{0}\psi+\sum_{k=1}^{\infty}P_{k}^{1}\psi+\sum_{k=1}^{\infty}P_{k}^{2}\psi
\end{equation}where
\begin{align}
\label{fourierprojop}
    \begin{split}
    P^{0}\psi&:=\big(P_{1}^{0}\psi\big)(r)+i\big(P_{2}^{0}\psi\big)(r),\\
    P_{k}^{1}\psi&:=\big(P_{k1}^{1}\psi\big)(r)\cos{k\theta}+i\big(P_{k2}^{1}\psi\big)(r)\sin{k\theta},\\
    P_{k}^{2}\psi&:=\big(P_{k1}^{2}\psi\big)(r)\sin{k\theta}+i\big(P_{k2}^{2}\psi\big)(r)\cos{k\theta},
    \end{split}
\end{align}and similarly for \(\elliprhs\). Then (\ref{ellipticlinearpsi}) reduces to the ODEs
\begin{equation}
\label{fmodesystem}
\Psi''+\left(2\frac{w'}{w}+\frac{1}{r}\right)\Psi'-\frac{1}{r^{2}}\begin{pmatrix}
        k^{2} & 2k \\
        2k & k^{2}+2w^{2}r^{2} 
    \end{pmatrix}\Psi=\tilde{\elliprhs}
\end{equation}for \(k\geq1\), where 
\begin{equation}
    \Psi=\begin{cases}
    \big(P_{k1}^{\nu}\psi,\:P_{k2}^{\nu}\psi\big),\quad\hspace{0.8em}\text{for }\nu=1,\\
    \big(P_{k1}^{\nu}\psi,\:-P_{k2}^{\nu}\psi\big),\quad\text{for }\nu=2,
   \end{cases}\quad
   \tilde{\elliprhs}=\begin{cases}
    \big(P_{k1}^{\nu}\elliprhs,\:P_{k2}^{\nu}\elliprhs\big),\quad\hspace{0.8em}\text{for }\nu=1,\\
    \big(P_{k1}^{\nu}\elliprhs,\:-P_{k2}^{\nu}\elliprhs\big),\quad\text{for }\nu=2.
   \end{cases}
\end{equation}
In the \(k=0\) case we get decoupled equations
\begin{gather}
    \Psi_{1}''+\left(2\frac{w'}{w}+\frac{1}{r}\right)\Psi_{1}'=\tilde{\elliprhs}_{1},\label{fmodeuncoupled1}\\
\Psi_{2}''+\left(2\frac{w'}{w}+\frac{1}{r}\right)\Psi_{2}'-2w^{2}\Psi_{2}=\tilde{\elliprhs}_{2},\label{fmodeuncoupled2} 
\end{gather}for \(\Psi_{1}=P_{1}^{0}\psi\) and \(\Psi_{2}=P_{2}^{0}\psi\), with respective right-hand sides \(\tilde{\elliprhs}_{1}=P_{1}^{0}\elliprhs\) and \(\tilde{\elliprhs}_{2}=P_{2}^{0}\elliprhs\).

\medskip

A standard application of the variation of parameters method (see \cite{delpinojunemanmusso2025}*{Section 2}) gives the following representation formulae for solutions to (\ref{fmodeuncoupled1})-(\ref{fmodeuncoupled2}).

\begin{proposition}
\label{mode0formulaeprop}
    The general solution of (\ref{fmodeuncoupled1}) can be written in the form
    \begin{equation}
    \label{modezero1repformula}
        \Psi_{1}(r)=\int\frac{ds}{w(s)^{2}s}\int w(t)^{2}t\tilde{\elliprhs}_{1}(t)\:dt,
    \end{equation}where the symbols \(\int\) denote arbitrary antiderivatives. Moreover, the homogeneous version of (\ref{fmodeuncoupled2}) admits two linearly independent solutions with the asymptotic behaviour
    \begin{equation}
  \label{mode0kernelprop}
    z_{1,0}(r)=\begin{cases}
        O(1)\text{ as }r\to0^{+},\\
        O\left(r^{-1/2}e^{\sqrt{2}r}\right)\text{ as }r\to\infty,
    \end{cases}
        z_{2,0}(r)=\begin{cases}
            O(r^{-2})\text{ as }r\to0^{+},\\
            O\left(r^{-1/2}e^{-\sqrt{2}r}\right)\text{ as }r\to\infty,
        \end{cases}
\end{equation}
and the inhomogeneous problem can then be solved via the formula 
    \begin{equation}
    \label{modezero2repformula}
    \Psi_{2}(r)=z_{2,0}(r)\int w(s)^{2}s\:\tilde{\elliprhs}_{2}(s)z_{1,0}(s)\:ds
    -z_{1,0}(r)\int w(s)^{2}s\:\tilde{\elliprhs}_{2}(s)z_{2,0}(s)\:ds.
    \end{equation}
\end{proposition}

Obtaining a representation formula for solutions to (\ref{fmodesystem}) requires more work due to the coupled nature of the system. We first recall a classification result for the homogeneous equation (i.e. setting \(\tilde{\elliprhs}=0\)).

\begin{proposition}[\cite{delpinojunemanmusso2025}*{Propositions 3.4 and 4.4}]
    For \(k\geq2\), the homogeneous version of (\ref{fmodesystem}) admits four linearly independent solutions \(z_{1,k}\), \(z_{2,k}\), \(z_{3,k}\), \(z_{4,k}\) with the asymptotic behaviour
    \begin{equation}
\label{modegeq2kernelasymprop}
\begin{aligned}[c]
z_{1,k}(r)&=\begin{pmatrix}[1.3]
    O(r^{-k})\\
    O(r^{-k})
    \end{pmatrix},\text{ as }r\to0^{+},\\
z_{2,k}(r)&=\begin{pmatrix}[1.3]
    O(r^{k-2})\\
    O(r^{k-2})
    \end{pmatrix},\text{ as }r\to0^{+},\\
z_{3,k}(r)&=\begin{pmatrix}[1.3]
    O(r^{k})\\
    O(r^{k})
    \end{pmatrix},\text{ as }r\to0^{+},\\
z_{4,k}(r)&=\begin{pmatrix}[1.3]
    O(r^{-2-k})\\
    O(r^{-2-k})
    \end{pmatrix},\text{ as }r\to0^{+},    
\end{aligned}
\quad\quad
\begin{aligned}[c]
z_{1,k}(r)&=\begin{pmatrix}[1.3]
    O(r^{-k})\\
    O(r^{-k-2})
    \end{pmatrix},\text{ as }r\to\infty,\\
z_{2,k}(r)&=\begin{pmatrix}[1.3]
    O(r^{k})\\
    O(r^{k-2})
    \end{pmatrix},\text{ as }r\to\infty,\\
z_{3,k}(r)&=\begin{pmatrix}[1.3]
    O\big(r^{-5/2}e^{\sqrt{2}r}\big)\\
    O\big(r^{-1/2}e^{\sqrt{2}r}\big)
    \end{pmatrix},\text{ as }r\to\infty,\\
z_{4,k}(r)&=\begin{pmatrix}[1.3]
    O\big(r^{-5/2}e^{-\sqrt{2}r}\big)\\
    O\big(r^{-1/2}e^{-\sqrt{2}r}\big)
    \end{pmatrix},\text{ as }r\to\infty.    
\end{aligned}
\end{equation}The same statement holds for \(k=1\) with the slight modification
\begin{equation}
\label{z21mod}
    z_{2,1}(r)=\begin{pmatrix}[1.3]
    O(r^{-1}\log r)\\
    O(r^{-1}\log r)
    \end{pmatrix},\text{ as }r\to0^{+},
\end{equation}and in this case we have explicitly \(z_{1,1}(r)=\big(1/r,-w'/w\big)\) (\textit{cf.} (\ref{kernel2})).
\end{proposition}

The solvability of the inhomogeneous problem is then addressed in the following.

\begin{proposition}[\cite{delpinojunemanmusso2025}*{Propositions 3.8 and 4.8}]  
\label{modegeq1formulaeprop}
    For \(k\geq1\), the general solution of (\ref{fmodesystem}) can be written in the form
    \begin{equation}
    \label{modegeq1repformula}
\begin{split}
    \Psi(r)=&\quad\left(\int w(s)^{2}s\:\tilde{\elliprhs}(s)\cdot z_{2,k}(s)\:ds\right)z_{1,k}(r)\\
    &-\left(\int w(s)^{2}s\:\tilde{\elliprhs}(s)\cdot z_{1,k}(s)\:ds\right)z_{2,k}(r)\\
    &+\left(\int w(s)^{2}s\:\tilde{\elliprhs}(s)\cdot z_{4,k}(s)\:ds\right)z_{3,k}(r)\\
    &-\left(\int w(s)^{2}s\:\tilde{\elliprhs}(s)\cdot z_{3,k}(s)\:ds\right)z_{4,k}(r),
\end{split}
\end{equation}where \(\cdot\) denotes the usual dot product in \(\mathbb{R}^{2}\), and the symbols \(\int\) denote arbitrary antiderivatives.
\end{proposition}

\begin{remark}
\label{summingfouriermodes}
   Using the explicit formulae (\ref{modezero1repformula}), (\ref{modezero2repformula}), (\ref{modegeq1repformula}) and a suitable choice for the limits of integration, one can construct a solution of the linear problem (\ref{ellipticlinearpsi}) with \(iW\psi\) bounded at the origin for any right-hand side \(\elliprhs\) with finitely many nonzero modes in its Fourier expansion (and with \(iW\elliprhs\) bounded at \(r=0\)). Moreover, sharp estimates for \(\psi\) can be deduced using (\ref{mode0kernelprop}), (\ref{modegeq2kernelasymprop}), (\ref{z21mod}). Existence and estimates extend to the general case (i.e. to functions supported in all Fourier modes) using a contradiction argument and maximum principle type estimates. See \cite{delpinojunemanmusso2025}*{Section 5} for a detailed discussion in the case \(\operatorname{Re}(\elliprhs)=O(r^{-2})\) and \(\operatorname{Im}(\elliprhs)=O(1)\) as \(r\to\infty\).
\end{remark}

\begin{remark}
\label{derivativeestimates}
    Once (\ref{ellipticlinearpsi}) is solved and estimates for \(\psi\) are established, one can also obtain bounds for the gradient and higher-order derivatives. We note that, in terms of real and imaginary parts \(\psi=\psi_{1}+i\psi_{2}\) and \(\elliprhs=\elliprhs_{1}+i\elliprhs_{2}\), the linear problem \(L[iW\psi]=iW\elliprhs\) reads
    \begin{align}
    \Delta\psi_{1}+2\frac{w'}{w}\partial_{r}\psi_{1}-\frac{2}{r^{2}}\partial_{\theta}\psi_{2}&=\elliprhs_{1},\label{linearizedpsieq1}\\
    \Delta\psi_{2}+2\frac{w'}{w}\partial_{r}\psi_{2}+\frac{2}{r^{2}}\partial_{\theta}\psi_{1}-2w^{2}\psi_{2}&=\elliprhs_{2}.\label{linearizedpsieq2}
\end{align}Neglecting the \(\partial_{\theta}\) terms, the operator on the left-hand side of (\ref{linearizedpsieq1}) behaves roughly like \(\Delta\psi_{1}\) for large \(r\), while the operator on the left-hand side of (\ref{linearizedpsieq2}) behaves like \(\Delta\psi_{2}-2\psi_{2}\). Using the explicit representation formula for \((-\Delta+2)^{-1}\) and rescaled Schauder estimates, one can typically show that \(D^{\ell}\psi_{1}\) and \(D^{\ell}\psi_{2}\) decay \(r^{\ell}\) times faster than \(\psi_{1}\) and \(\psi_{2}\) (respectively), provided estimates of this form also hold for the derivatives of the right-hand side. (We recall that the representation formula for \((-\Delta+2)^{-1}\) is given by convolution with \(\tfrac{1}{2\pi}K_{0}\big(\sqrt{2}(\cdot)\big)\), where \(K_{0}\) is the modified Bessel function of second kind).

\end{remark}

\begin{remark}
\label{conjugatesymmetry}
    Note that \(\phi\) and \(h\) satisfy the linear problem \(L[\phi]=h\) if and only if \(\bar{\phi}\) and \(\bar{h}\) satisfy the equation \(\bar{L}[\bar{\phi}]=\bar{h}\), where
    \begin{equation}
    \label{Lbardef}
        \bar{L}[\phi]:=\Delta_{y}\phi+(1-\abs{\overline{W}}^{2})\phi-2\operatorname{Re}(W\phi)\overline{W}
        \end{equation}denotes the elliptic linearized operator around the degree \(-1\) vortex \(\overline{W}\). In particular, all statements for \(L\) have a corresponding analogue for (\ref{Lbardef}) after conjugating the relevant functions involved.
\end{remark}

\section{Error of the approximation and first improvement}
\label{Firstimprovementsect}

In this section we begin the construction of an accurate \(n\)-vortex expansion, starting with the product of vortices
\begin{equation}
\label{firstapprox}
   U_{\xi}(y,t):=\prod_{j\in I_{+}}W\big(y-\tilde{\xi}_{j}(t)\big)\prod_{j\in I_{-}}\overline{W}\big(y-\tilde{\xi}_{j}(t)\big)
\end{equation}as a first approximation. Our initial task is to compute the error \(S(U_{\xi})\) associated to (\ref{firstapprox}). We then recall the ansatz (\ref{ansatzforapprox}) for an improvement of the approximation, and describe how to choose the functions \(\phi_{j}\) and \(\psi^{out}\) to improve the size of \(S(U_{\xi})\) by ``one power of \(\epsilon\)''.

\subsection{First error}\label{firsterrorsect} Recall the definition of \(W_{j}\) from (\ref{Wjdef}). A basic first step in the construction is to understand how well \(U_{\xi}=\prod_{j}W_{j}\) fits the equation \(S(u)=0\); our starting point is the following global expression for the error \(S(U_{\xi})\), where \(w_{j}:=\abs{W_{j}}\) and
\begin{equation*}
\nabla_{y}\theta_{j}:=\frac{(y-\tilde{\xi}_{j})^{\perp}}{\abs{y-\tilde{\xi}_{j}}^{2}}.
\end{equation*}
\begin{lemma}
\label{firsterrorlemma}
For any choice of smooth parameters \(\xi(t):=\big(\xi_{1}(t),\ldots,\xi_{n}(t)\big)\), we have
\begin{equation*}
    S(U_{\xi})=iU_{\xi}\big(R_{1}+iR_{2}\big),
\end{equation*}where
\begin{align}
\label{R1def}
    R_{1}=R_{1}(y,t;\xi):=&\sum_{j=1}^{n}\frac{\nabla_{y} w_{j}}{w_{j}}\cdot(-\epsilon\dot{\xi}_{j})+2\sum_{j=1}^{n}\sum_{k\neq j}\frac{\nabla_{y} w_{j}}{w_{j}}\cdot d_{k}\nabla_{y}\theta_{k},\\
    \begin{split}
    \label{R2def}
    R_{2}=R_{2}(y,t;\xi):=&\sum_{j=1}^{n}d_{j}\nabla_{y}\theta_{j}\cdot(-\epsilon\dot{\xi}_{j})+\sum_{j=1}^{n}\sum_{k\neq j}d_{j}d_{k}\nabla_{y}\theta_{j}\cdot \nabla_{y}\theta_{k}\\
    &-\bigg(1-\prod_{j=1}^{n}w_{j}^{2}-\sum_{j=1}^{n}\big(1-w_{j}^{2}\big)\bigg)-\sum_{j=1}^{n}\sum_{k\neq j}\frac{\nabla_{y} w_{j}}{w_{j}}\cdot\frac{\nabla_{y} w_{k}}{w_{k}}.
    \end{split}
\end{align}
\end{lemma}

\begin{proof}
    By direct calculation, we have
    \begin{equation*}
    \partial_{t}U_{\xi}=U_{\xi}\bigg(\sum_{j=1}^{n}\frac{\partial_{t}W_{j}}{W_{j}}\bigg)=U_{\xi}\bigg(\sum_{j=1}^{n}\frac{\nabla_{y}W_{j}}{W_{j}}\cdot\big(-\epsilon^{-1}\dot{\xi}_{j}\big)\bigg),
    \end{equation*}
    \begin{equation}
    \label{laplacianfirstapprox}\Delta_{y}U_{\xi}=\operatorname{div}_{y}\bigg(U_{\xi}\bigg(\sum_{j=1}^{n}\frac{\nabla_{y}W_{j}}{W_{j}}\bigg)\bigg)=U_{\xi}\bigg(\sum_{j=1}^{n}\frac{\Delta_{y}W_{j}}{W_{j}}+\sum_{j=1}^{n}\sum_{k\neq j}\frac{\nabla_{y}W_{j}}{W_{j}}\cdot\frac{\nabla_{y}W_{k}}{W_{k}}\bigg).
    \end{equation}
Using \(\Delta_{y} W_{j}+(1-\abs{W_{j}}^{2})W_{j}=0\), the expression (\ref{laplacianfirstapprox}) can be written as
    \begin{equation*}
        \Delta_{y}U_{\xi}=U_{\xi}\bigg(\sum_{j=1}^{n}\sum_{k\neq j}\frac{\nabla_{y}W_{j}}{W_{j}}\cdot\frac{\nabla_{y}W_{k}}{W_{k}}-\sum_{j=1}^{n}(1-\abs{W_{j}}^{2})\bigg).
    \end{equation*}It then follows, using the definition (\ref{Soperator}) of \(S\), that
    \begin{equation*}
        S(U_{\xi})=U_{\xi}\bigg(\sum_{j=1}^{n}\frac{\nabla_{y}W_{j}}{W_{j}}\cdot(-i\epsilon\dot{\xi}_{j})+\sum_{j=1}^{n}\sum_{k\neq j}\frac{\nabla_{y}W_{j}}{W_{j}}\cdot\frac{\nabla_{y}W_{k}}{W_{k}}+\bigg(1-\prod_{j=1}^{n}\abs{W_{j}}^{2}-\sum_{j=1}^{n}(1-\abs{W_{j}}^{2})\bigg)\bigg).   
    \end{equation*}We conclude by writing the above expression in terms of its real and imaginary parts and factoring out \(i\), noting that we have the decomposition    \begin{equation*}
    \frac{\nabla_{y} W_{j}}{W_{j}}=\frac{\nabla_{y}w_{j}}{w_{j}}+id_{j}\nabla_{y}\theta_{j}
    \end{equation*}for the gradient of \(W_{j}\).
\end{proof}

Next we record a precise expansion for the functions (\ref{R1def}) and (\ref{R2def}) close to the zeroes of \(U_{\xi}\). For a fixed \(j\in\{1,\ldots,n\}\), we work with polar coordinates \((r,\theta)\) for \(y_{j}:=y-\tilde{\xi}_{j}\), and we write simply \(w=w_{j}\) and \(\nabla\theta=\nabla_{y}\theta_{j}\). We also assume
\begin{equation*}
    \abs{\xi_{j}(t)-\xi_{k}(t)}\geq2\delta,\quad\text{for all}\quad t\in[0,T],\quad k\neq j,
\end{equation*}where \(\delta>0\) is the number defined in (\ref{deltadef}).

\begin{lemma}
\label{refinederrorexp}
    In the region \(\abs{y-\tilde{\xi}_{j}}\leq\delta\epsilon^{-1}\), we have
    \begin{align}
        R_{1}=&\epsilon\frac{\nabla w}{w}\cdot\bigg(-\dot{\xi}_{j}+2\sum_{k\neq j}d_{k}\frac{(\xi_{j}-\xi_{k})^{\perp}}{\abs{\xi_{j}-\xi_{k}}^{2}}\bigg)+R_{j1}+O(\epsilon^{3}r^{-1}),\label{R1expansion}\\
        R_{2}=&\epsilon d_{j}\nabla\theta\cdot\bigg(-\dot{\xi}_{j}+2\sum_{k\neq j}d_{k}\frac{(\xi_{j}-\xi_{k})^{\perp}}{\abs{\xi_{j}-\xi_{k}}^{2}}\bigg)+R_{j2}+O\big(\epsilon^{3}(r^{-1}+r)\big),\label{R2expansion}    
        \end{align}where \(R_{j1}=O(\epsilon^{2}(1+r^{2})^{-1})\) contains only mode 2 terms in its Fourier expansion, and \(R_{j2}=O(\epsilon^{2})\) contains only mode 2 and mode 0 terms in its Fourier expansion. 
\end{lemma}

\begin{remark*}
    To clarify, the statement above means that \(R_{j1}\) is a linear combination of \(\sin(2\theta)\) and \(\cos(2\theta)\), and \(R_{j2}\) is a linear combination of \(\sin(2\theta)\), \(\cos(2\theta)\) and a radial function of \(r\).
\end{remark*}    

The proof of Lemma \ref{refinederrorexp} makes use of the following preliminary result.

\begin{lemma}
\label{trigexpansions}
    For all \(y,\zeta\in\mathbb{R}^{2}\) with \(\abs{y}<\abs{\zeta}\), we have
    \begin{align*}
        y\cdot\bigg(\frac{(y+\zeta)^{\perp}}{\abs{y+\zeta}^{2}}-\frac{\zeta^{\perp}}{\abs{\zeta}^{2}}\bigg)=&\sum_{m=2}^{\infty}(-1)^{m-1}\frac{\abs{y}^{m}}{\abs{\zeta}^{m}}\sin\big(m(\theta-q)\big)\\
        y^{\perp}\cdot\bigg(\frac{(y+\zeta)^{\perp}}{\abs{y+\zeta}^{2}}-\frac{\zeta^{\perp}}{\abs{\zeta}^{2}}\bigg)=&\sum_{m=2}^{\infty}(-1)^{m-1}\frac{\abs{y}^{m}}{\abs{\zeta}^{m}}\cos\big(m(\theta-q)\big)
        \end{align*}where \(\theta\) denotes the polar argument of \(y\) and \(q\) the polar argument of \(\zeta\).
\end{lemma}

\begin{proof}
   A proof of the first claim can be found in \cite{gallay2011}*{Lemma 3}. The second claim follows using a similar argument.
\end{proof}

\begin{proof}[Proof of Lemma \ref{refinederrorexp}]
    In the region considered, (\ref{R1def}) and the properties \(w'(r)\sim r^{-3}\)  and \(\abs{\nabla\theta}\sim r^{-1}\) as \(r\to\infty\) (see (\ref{derivativesofw})) imply that the expression for \(R_{1}\) takes the form
    \begin{equation}
    \label{leadingorderR1}
    R_{1}=\frac{\nabla w}{w}\cdot\bigg(-\epsilon\dot{\xi}_{j}+2\sum_{k\neq j}d_{k}\nabla\theta_{k}\bigg)+O(\epsilon^{3}r^{-1}).    
    \end{equation}Since
    \begin{equation*}
    2\sum_{k\neq j}d_{k}\nabla\theta_{k}=2\epsilon\sum_{k\neq j}d_{k}\frac{(\epsilon y_{j}+\xi_{j}-\xi_{k})^{\perp}}{\abs{\epsilon y_{j}+\xi_{j}-\xi_{k}}^{2}},
    \end{equation*}we can then use Lemma \ref{trigexpansions} to find that the leading order term of (\ref{leadingorderR1}) can be written as
    \begin{equation*}
        \epsilon\frac{\nabla w}{w}\cdot\bigg(-\dot{\xi}_{j}+2\sum_{k\neq j}d_{k}\frac{(\xi_{j}-\xi_{k})^{\perp}}{\abs{\xi_{j}-\xi_{k}}^{2}}\bigg)+\frac{w'}{w}\sum_{k\neq j}\frac{2d_{k}}{r}\sum_{m=2}^{\infty}\frac{(-1)^{m-1}\epsilon^{m}r^{m}}{\abs{\xi_{j}-\xi_{k}}^{m}}\sin\big(m(\theta-q_{jk})\big),
        \end{equation*}where \(q_{jk}=q_{jk}(t)\) denotes the polar angle of \(\xi_{j}(t)-\xi_{k}(t)\). Extracting the first term in the sum over \(m\) then gives (\ref{R1expansion}) with
        \begin{equation}
        \label{Rj1def}
            R_{j1}:=-2\epsilon^{2}\frac{w'r}{w}\sum_{k\neq j}\frac{d_{k}}{\abs{\xi_{j}-\xi_{k}}^{2}}\sin\big(2(\theta-q_{jk})\big).
        \end{equation}

        Let us now consider the expansion for \(R_{2}\). In the region considered we have
        \begin{equation*}
            \begin{split}
                R_{2}=&d_{j}\nabla\theta\cdot\bigg(-\epsilon\dot{\xi}_{j}+2\sum_{k\neq j}d_{k}\nabla\theta_{k}\bigg)+\sum_{k\neq j}d_{k}\nabla\theta_{k}\cdot(-\epsilon\dot{\xi}_{k})+\sum_{\substack{k\neq j\\ l\neq k,j}}d_{k}d_{l}\nabla\theta_{k}\cdot\nabla\theta_{l}\\
                &-\bigg(1-\prod_{k=1}^{n}w_{k}^{2}-\sum_{k=1}^{n}\big(1-w_{k}^{2}\big)\bigg)+O\big(\epsilon^{3}r^{-1}(1+r^2)^{-1}\big).
            \end{split}
        \end{equation*}Since 
        \begin{equation*}
           1-\prod_{k=1}^{n}w_{k}^{2}-\sum_{k=1}^{n}\big(1-w_{k}^{2}\big)=-\sum_{m=2}^{n}(-1)^{m}\sum_{k_{1}<\ldots<k_{m}}(1-w_{k_{1}}^{2})\ldots(1-w_{k_{m}}^{2})
        \end{equation*}and \(1-w^{2}\sim r^{-2}\) as \(r\to\infty\), we find (after making a Taylor expansion) that
        \begin{equation*}
        \sum_{k\neq j}d_{k}\nabla\theta_{k}\cdot(-\epsilon\dot{\xi}_{k})+\sum_{\substack{k\neq j\\ l\neq k,j}}d_{k}d_{l}\nabla\theta_{k}\cdot\nabla\theta_{l}-\bigg(1-\prod_{k=1}^{n}w_{k}^{2}-\sum_{k=1}^{n}\big(1-w_{k}^{2}\big)\bigg)=R_{j2}^{(0)}+O(\epsilon^{3}r)   
        \end{equation*}for \(\abs{y-\tilde{\xi}_{j}}\leq\delta\epsilon^{-1}\), where \(R_{j2}^{(0)}=R_{j2}^{(0)}(r,t)\) is a radial function of size \(O(\epsilon^{2})\). Concerning the first term
        \begin{equation*}
        d_{j}\nabla\theta\cdot\bigg(-\epsilon\dot{\xi}_{j}+2\sum_{k\neq j}d_{k}\nabla\theta_{k}\bigg),  \end{equation*}the second part of Lemma \ref{trigexpansions} implies this expression can be written in the form 
        \begin{equation*}
        \epsilon d_{j}\nabla\theta\cdot\bigg(-\dot{\xi}_{j}+2\sum_{k\neq j}d_{k}\frac{(\xi_{j}-\xi_{k})^{\perp}}{\abs{\xi_{j}-\xi_{k}}^{2}}\bigg)+\frac{2d_{j}}{r^{2}}\sum_{k\neq j}d_{k}\sum_{m=2}^{\infty}(-1)^{m-1}\frac{\epsilon^{m}r^{m}}{\abs{\xi_{j}-\xi_{k}}^{m}}\cos\big(m(\theta-q_{jk})\big).
        \end{equation*}Extracting the first term in the sum over \(m\) then implies that (\ref{R2expansion}) holds with
        \begin{equation}
        \label{Rj2def}
            R_{j2}:=-2\epsilon^{2}d_{j}\sum_{k\neq j}\frac{d_{k}}{\abs{\xi_{j}-\xi_{k}}^{2}}\cos\big(2(\theta-q_{jk})\big)+R_{j2}^{(0)}(r,t).
        \end{equation}
        
\end{proof}

An important consequence of the previous result and the global expressions (\ref{R1def}) and (\ref{R2def}) is the following.

\begin{corollary}
    For parameters \(\xi(t)=\big(\xi_{1}(t),\ldots,\xi_{n}(t)\big)\) satisfying \(\xi(t)=\xi^{0}(t)+O(\epsilon)\) in \(C^{1}\) sense, we have \(S(U_{\xi})=O(\epsilon^{2})\) in \(L^{\infty}\big(\mathbb{R}^{2}\times[0,T]\big)\).
\end{corollary}

\begin{proof}
    Under the assumption \(\xi(t)=\xi^{0}(t)+O(\epsilon)\), the leading order terms in (\ref{R1expansion})-(\ref{R2expansion}) have size \(O(\epsilon^{2})\) rather than \(O(\epsilon)\) thanks to the equation (\ref{KirchhoffODE}) satisfied by \(\xi^{0}(t)\). Thus \(S(U_{\xi})=O(\epsilon^{2})\) if \(\abs{y-\tilde{\xi}_{j}}\leq\delta\epsilon^{-1}\) for some \(j\). In the outer region \(\abs{y-\tilde{\xi}_{j}}\geq\delta\epsilon^{-1}\) for all \(j\), the estimates \(\abs{\nabla_{y}w_{j}}=O(\epsilon^{3})\) and \(\abs{\nabla_{y}\theta_{j}}=O(\epsilon)\) directly imply \(R_{1}=O(\epsilon^{4})\) and \(R_{2}=O(\epsilon^{2})\). 
\end{proof}

\begin{remark}
    The result above gives \(\epsilon^{2}\) smallness of the first error in \(L^{\infty}\) sense. On the other hand, it is directly verified that we have the large \(\abs{y}\) asymptotics
    \begin{align*}
        R_{1}&=\frac{y}{\abs{y}^{4}}\cdot\bigg(-\epsilon\sum_{j=1}^{n}\dot{\xi}_{j}\bigg)+O\big(\abs{y}^{-4}\big),\quad\text{as}\quad\abs{y}\to\infty,\\
        R_{2}&=\frac{y^{\perp}}{\abs{y}^{2}}\cdot\bigg(-\epsilon\sum_{j=1}^{n}d_{j}\dot{\xi}_{j}\bigg)+O\big(\abs{y}^{-2}\big),\quad\text{as}\quad\abs{y}\to\infty,
    \end{align*} if \(\abs{y}\geq\epsilon^{-(1+\sigma)}\) for some \(\sigma>0\). Thus \(S(U_{\xi})\) is square integrable in \(\mathbb{R}^{2}\) if and only if 
    \mbox{\(
        \sum_{j}d_{j}\dot{\xi}_{j}=0\)}. This condition is satisfied by solutions of (\ref{KirchhoffODE}), but not for a general choice of parameters. In other words, the first error is not finite in \(H^{1}(\mathbb{R}^{2})\)-norm in general.
\end{remark}

\subsection{Error expression for an improvement of the approximation} The next step in the construction is to build a refined expansion \(u_{*}\) for an \(n\)-vortex solution which improves the initial error in powers of \(\epsilon\). As outlined in \(\S\)\ref{approxsolsubsection}, this will be achieved using an ansatz of the form
\begin{equation}
\label{newapprox}
u_{*}(y,t)=e^{i\psi^{out}}\prod_{j=1}^{n}\bigg(\eta_{j}(W_{j}+\tilde{\eta}_{j}\phi_{j})+(1-\eta_{j})W_{j}e^{\tilde{\eta}_{j}\phi_{j}/W_{j}}\bigg),
\end{equation}where \(\psi^{out}\) is expressed in the original variables \((x,t)\), and  \(\phi_{j}=\phi_{j}(y_{j},t)\) (where \(y_{j}=y-\tilde{\xi}_{j}\)). The smooth cut-offs \(\eta_{j}\) and \(\tilde{\eta}_{j}\) are defined as follows: we first fix a smooth, radial function \(\eta_{0}\in C^{\infty}_{c}(\mathbb{R}^{2})\) such that
\begin{equation}
\label{eta0def}
    \eta_{0}(y):=\begin{cases}
   1,\quad\text{for}\quad\abs{y}\leq1,\\
   0,\quad\text{for}\quad\abs{y}\geq2.
   \end{cases}
\end{equation}We then set
\begin{equation*}
    \eta_{j}(y,t):=\eta_{0}(y-\tilde{\xi}_{j}),\quad\quad\tilde{\eta}_{j}(y,t):=\eta_{0}\big(\epsilon\delta^{-1}(y-\tilde{\xi}_{j})\big).
\end{equation*}

\medskip

Before entering the details of the expansion, our first aim is to compute a general expression for the error \(S(u_{*})\) associated to (\ref{newapprox}). This is ultimately achieved in Proposition \ref{ustarglobalerror} below. We start with a preliminary lemma.

\begin{lemma}
\label{ustarouterror}
    Let \(\phi_{j}=iW_{j}\psi_{j}\) for \(j\in\{1,\ldots,n\}\), where \(\psi_{j}=\psi_{j}(y_{j},t)\). In the region \(\abs{y-\tilde{\xi}_{j}}\geq2\) for all \(j\), we have
    \begin{equation}
    \label{ustarout}
        u_{*}=U_{\xi}\exp\bigg(i\bigg(\psi^{out}+\sum_{j=1}^{n}\tilde{\eta}_{j}\psi_{j}\bigg)\bigg),
    \end{equation}
    and
    \begin{equation}
    \label{ustarouterrorexp}
        S(u_{*})=iu_{*}\Bigg(\frac{S(U_{\xi})}{iU_{\xi}}+\widetilde{S}'(U_{\xi})\bigg[\psi^{out}+\sum_{j}\tilde{\eta}_{j}\psi_{j}\bigg]+\widetilde{N}_{0}\bigg(\psi^{out}+\sum_{j}\tilde{\eta}_{j}\psi_{j}\bigg)\Bigg).
    \end{equation}Here
    \begin{equation}
    \label{stildeprime}
    \widetilde{S}'(U_{\xi})[\psi]:=\epsilon^{2}i\psi_{t}+\Delta_{y}\psi+2\frac{\nabla_{y}U_{\xi}}{U_{\xi}}\cdot\nabla_{y}\psi-2i\abs{U_{\xi}}^{2}\psi_{2}
    \end{equation}denotes the linearized operator around \(U_{\xi}\) with respect to the multiplicative ansatz \(U_{\xi}e^{i\psi}\) (where \(\psi=\psi_{1}+i\psi_{2}\)), and 
    \begin{equation}
    \label{multnonlinearterms}
        \widetilde{N}_{0}(\psi):=i(\nabla_{y}\psi)^{2}+i\abs{U_{\xi}}^{2}\big(e^{-2\psi_{2}}-(1-2\psi_{2})\big)
    \end{equation}denotes the corresponding nonlinear terms.
\end{lemma}

\begin{remark*}
    Written in real and imaginary parts, the first term in (\ref{multnonlinearterms}) takes the form
    \begin{equation*}
i(\nabla_{y}\psi)^{2}=-2\nabla_{y}\psi_{1}\cdot\nabla_{y}\psi_{2}+i\big(\abs{\nabla_{y}\psi_{1}}^{2}-\abs{\nabla_{y}\psi_{2}}^{2}\big).
    \end{equation*}
\end{remark*}

\begin{proof}[Proof of Lemma \ref{ustarouterror}]
In the region considered we have \(\eta_{j}=0\) for all \(j\), and the ansatz (\ref{newapprox}) then reduces to (\ref{ustarout}). Next, for any functions \(u\) and \(\psi=\psi_{1}+i\psi_{2}\) we have the validity of the expansion
    \begin{equation*}
        S\big(ue^{i\psi}\big)=iue^{i\psi}\bigg(\frac{S(u)}{iu}+\widetilde{S}'(u)[\psi]+\widetilde{N}_{0}(\psi)\bigg)
    \end{equation*}where \(\widetilde{S}'\) is defined by (\ref{stildeprime}) and \(\widetilde{N}_{0}\) is defined by (\ref{multnonlinearterms}) (with \(U_{\xi}=u\)). Setting \(u=U_{\xi}\) and \(\psi=\psi^{out}+\sum_{j}\tilde{\eta}_{j}\psi_{j}\) then gives (\ref{ustarouterrorexp}).
\end{proof}

In the next lemma, we compute an expression for \(S(u_{*})\) which is valid inside the vortex cores. We continue (both here and in the following sections) to use the decomposition suggested above, namely
\begin{equation}
    \phi_{j}=iW_{j}\psi_{j}\quad\text{where}\quad{\psi_{j}=\psi_{j}(y_{j},t)}.
\end{equation}

\begin{lemma}
\label{ustarinnererror}
    If \(\abs{y-\tilde{\xi}_{j}}\leq10\) for some \(j\), we can write
    \begin{equation}
    \label{ustarin}
        u_{*}=e^{i\psi^{out}}U_{\xi}\bigg(1+i\big(\psi_{j}+\gamma_{j}(\psi_{j})\big)\bigg)
    \end{equation}where
    \begin{equation}
    \label{gammajdef}
        \gamma_{j}(\psi_{j}):=-i(1-\eta_{j})\big(e^{i\psi_{j}}-(1+i\psi_{j})\big).
    \end{equation}Moreover, in this region we have
    \begin{equation}
     \label{ustarinerrorexp}
        S(u_{*})=iU_{\xi}e^{i\psi^{out}}\bigg(\frac{S(U_{\xi})}{iU_{\xi}}+\widetilde{S}'(U_{\xi})[\psi_{j}]+\frac{S(U_{\xi})}{iU_{\xi}}i\psi_{j}+\widetilde{S}'(U_{\xi})[\psi^{out}]+N_{0j}\big(\psi^{out},\psi_{j}\big)\bigg)
    \end{equation}where
    \begin{equation}
    \label{innernonlinearterms}
    \begin{split}
    N_{0j}\big(\psi^{out},\psi_{j}\big):=&\hspace{1.4em}\widetilde{N}_{0}(\psi^{out})+\widetilde{S}'\big(U_{\xi}e^{i\psi^{out}}\big)[\gamma_{j}(\psi_{j})]\\
    &+2i\nabla_{y}\psi^{out}\cdot\nabla_{y}\psi_{j}-2i\abs{U_{\xi}}^{2}\big(e^{-2\psi_{2}^{out}}-1\big)\operatorname{Im}(\psi_{j})\\
    &+\widetilde{S}'(U_{\xi})[\psi^{out}]i\psi_{j}+\widetilde{N}_{0}(\psi^{out})i\psi_{j}+\frac{S(U_{\xi}e^{i\psi^{out}})}{iU_{\xi}e^{i\psi^{out}}}i\gamma_{j}(\psi_{j})\\
    &+2\abs{U_{\xi}}^{2}e^{-2\psi_{2}^{out}}\big(\operatorname{Im}(\psi_{j}+\gamma_{j}(\psi_{j}))\big)\big(\psi_{j}+\gamma_{j}(\psi_{j})\big)\\
    &+i\abs{U_{\xi}}^{2}e^{-2\psi_{2}^{out}}\abs{\psi_{j}+\gamma_{j}(\psi_{j})}^{2}-\abs{U_{\xi}}^{2}e^{-2\psi_{2}^{out}}\abs{\psi_{j}+\gamma_{j}(\psi_{j})}^{2}\big(\psi_{j}+\gamma_{j}(\psi_{j})\big).\\
    \end{split}
    \end{equation}
\end{lemma}

\begin{remark*}
Note that the function (\ref{gammajdef}) is supported in the region \(\abs{y-\tilde{\xi}_{j}}\geq1\), and has quadratic size in \(\psi_{j}\) (using the series expansion for the exponential). The expression in parentheses in (\ref{ustarinerrorexp}) should be understood as ``first error'' \(+\) ``linear terms'' \(+\) ``nonlinear terms'' where the nonlinear terms are given by (\ref{innernonlinearterms}). We note that the precise definition of \(N_{0j}(\psi^{out},\psi_{j})\) is not so important for our later analysis: the key features are that it has quadratic size in \((\psi^{out},\psi_{j})\), and can be written in such a way that the ``outer'' contributions depend only on \(\psi_{2}^{out}\) and the derivatives of \(\psi^{out}\) (i.e. not on \(\psi_{1}^{out}=\operatorname{Re}(\psi^{out})\)). A similar decomposition to (\ref{ustarin}) was used in the analysis of stationary Ginzburg-Landau vortices in \cite{delpinokowalczykmusso2006}.
\end{remark*}

\begin{proof}[Proof of Lemma \ref{ustarinnererror}] In the specified region we have 
\begin{equation*}
   u_{*}=e^{i\psi^{out}}U_{\xi}\big(\eta_{j}(1+i\psi_{j})+(1-\eta_{j})e^{i\psi_{j}}\big). 
\end{equation*}The representation (\ref{ustarin}) then follows after expanding the second exponential. To derive the expression for the error, we first note that for any functions \(u\) and \(\psi=\psi_{1}+i\psi_{2}\) we have
\begin{equation*}
    S(u+iu\psi)=iu\bigg(\frac{S(u)}{iu}+\widetilde{S}'(u)[\psi]+\frac{S(u)}{iu}i\psi+2\abs{u}^{2}\psi_{2}\psi+i\abs{u}^{2}\abs{\psi}^{2}-\abs{u}^{2}\abs{\psi}^{2}\psi\bigg)
\end{equation*}where \(\widetilde{S}'\) is the operator (\ref{stildeprime}). Setting \(u=U_{\xi}e^{i\psi^{out}}\), \(\psi=\psi_{j}+\gamma_{j}(\psi_{j})\) and separating linear and nonlinear terms, we arrive at (\ref{ustarinerrorexp}).
    \end{proof}

We now state the main result of this section, namely an expression for the error \(S(u_{*})\) which is valid globally in \(\mathbb{R}^{2}\times[0,T]\). It is convenient to introduce the following cut-offs (\textit{cf.} (\ref{eta0def})):

\begin{equation*}
    \eta_{j}^{(2)}(y,t):=\eta_{0}\bigg(\frac{y-\tilde{\xi}_{j}}{2}\bigg),\quad\quad\eta^{(2)}:=\sum_{j=1}^{n}\eta_{j}^{(2)}.
\end{equation*}Moreover, we define
\begin{equation*}
u_{*}^{\eta}:=(1-\eta^{(2)})u_{*}+\eta^{(2)}U_{\xi}e^{i\psi^{out}}.
\end{equation*}

\begin{proposition}
\label{ustarglobalerror}
   Using the notation above, we have
   \begin{equation}
   \label{ustarglobalerrorexp}
       \begin{split}
           S(u_{*})=iu_{*}^{\eta}\Bigg(&\:\sum_{j=1}^{n}\tilde{\eta}_{j}\bigg(\frac{S(U_{\xi})}{iU_{\xi}}+\widetilde{S}'(U_{\xi})[\psi_{j}]+\eta_{j}^{(2)}\frac{S(U_{\xi})}{iU_{\xi}}i\psi_{j}\bigg)\\
           &+\sum_{j=1}^{n}\bigg(\epsilon^{2}i\partial_{t}\tilde{\eta}_{j}+\Delta_{y}\tilde{\eta}_{j}+2\frac{\nabla_{y}U_{\xi}}{U_{\xi}}\cdot\nabla_{y}\tilde{\eta}_{j}\bigg)\psi_{j}+2\sum_{j=1}^{n}\nabla_{y}\tilde{\eta}_{j}\cdot\nabla_{y}\psi_{j}\\
           &+\bigg(1-\sum_{j=1}^{n}\tilde{\eta}_{j}\bigg)\frac{S(U_{\xi})}{iU_{\xi}}+\widetilde{S}'(U_{\xi})[\psi^{out}]+N(\psi^{out},\psi^{in})\Bigg).\\
         \end{split}
   \end{equation}Here we denote \(\psi^{in}=(\psi_{1},\ldots,\psi_{n})\), and the nonlinear terms are given by
   \begin{equation}
   \label{nonlineartermsoutin}
   \begin{split}
       N(\psi^{out},\psi^{in}):=&\sum_{j=1}^{n}\eta_{j}^{(2)}\bigg(\frac{1}{\eta_{j}^{(2)}+(1-\eta_{j}^{(2)})e^{i\psi_{j}}}-1\bigg)\frac{S(U_{\xi})}{iU_{\xi}}i\psi_{j}\\
       &+\sum_{j=1}^{n}\bigg(\frac{\eta_{j}^{(2)}}{\eta_{j}^{(2)}+(1-\eta_{j}^{(2)})e^{i\psi_{j}}}\bigg)N_{0j}\big(\psi^{out},\psi_{j}\big)\\
       &+\bigg(\frac{(1-\eta^{(2)})e^{i\sum_{j}\tilde{\eta}_{j}\psi_{j}}}{\eta^{(2)}+(1-\eta^{(2)})e^{i\sum_{j}\tilde{\eta}_{j}\psi_{j}}}\bigg)\widetilde{N}_{0}\bigg(\psi^{out}+\sum_{j=1}^{n}\tilde{\eta}_{j}\psi_{j}\bigg).
   \end{split}    
   \end{equation}

\end{proposition}

\begin{proof}
    We decompose \(S(u_{*})=\eta^{(2)}S(u_{*})+(1-\eta^{(2)})S(u_{*})\) where the first term admits an expansion of the form (\ref{ustarinerrorexp}), and the second term admits an expansion of the form (\ref{ustarouterrorexp}). Separating into linear and nonlinear terms and factoring out \(iu_{*}^{\eta}\) then gives (\ref{ustarglobalerrorexp}).
\end{proof}

In the sequel we will consider functions \(\psi_{j}\) whose spatial derivatives become singular as \(\abs{y_{j}}\to0\), however \(\phi_{j}=iW_{j}\psi_{j}\) will always be smooth and bounded in a neighbourhood of \(\tilde{\xi}_{j}\). Our final result of this section records an equivalent expression for (\ref{ustarglobalerrorexp}) close to \(\tilde{\xi}_{j}\) which avoids the presence of singular terms. It is stated in terms of \(\phi_{j}\) and 
\begin{equation}
\label{alphajdef}
    \alpha_{j}:=\prod_{k\neq j}W_{k}.
\end{equation}

\begin{lemma}
    In the region \(\abs{y-\tilde{\xi}_{j}}\leq1\) we have
    \begin{equation}\label{nonsingerrorexp}
    \begin{split}
        S(u_{*})=e^{i\psi^{out}}\alpha_{j}\Bigg(&\:iW_{j}\bigg(\frac{S(U_{\xi})}{iU_{\xi}}+\widetilde{S}'(U_{\xi})[\psi^{out}]+\widetilde{N}_{0}\big(\psi^{out}\big)\bigg)\\
        &+\epsilon^{2}i\partial_{t}\phi_{j}+\Delta_{y}\phi_{j}+e^{-2\psi_{2}^{out}}\abs{\alpha_{j}}^{2}(1-\abs{W_{j}}^{2})\phi_{j}-2e^{-2\psi_{2}^{out}}\abs{\alpha_{j}}^{2}\operatorname{Re}(\overline{W}_{j}\phi_{j})W_{j}\\
        &+\bigg(2\frac{\nabla_{y}(e^{i\psi^{out}}\alpha_{j})}{e^{i\psi^{out}}\alpha_{j}}-\epsilon i\dot{\xi}_{j}\bigg)\cdot\nabla_{y}\phi_{j}+\frac{S(e^{i\psi^{out}}\alpha_{j})}{e^{i\psi^{out}}\alpha_{j}}\phi_{j}\\
        &+\abs{\alpha_{j}}^{2}e^{-2\psi_{2}^{out}}\bigg(-2\operatorname{Re}(\overline{W}_{j}\phi_{j})\phi_{j}-\abs{\phi_{j}}^{2}W_{j}-\abs{\phi_{j}}^{2}\phi_{j}\bigg)\Bigg).
    \end{split}    
    \end{equation}
\end{lemma}

\begin{proof}
    In the region considered we have \(u_{*}=e^{i\psi^{out}}\big(U_{\xi}+\alpha_{j}\phi_{j}\big)\). The result then follows by direct calculation, starting with an expansion of the form
    \begin{equation*}
        S(u+\phi)=S(u)+S'(u)[\phi]+\big(-2\operatorname{Re}(\overline{u}\phi)\phi-\abs{\phi}^{2}u-\abs{\phi}^{2}\phi\big)
    \end{equation*}where \(u=e^{i\psi^{out}}U_{\xi}\) and \(\phi=e^{i\psi^{out}}\alpha_{j}\phi_{j}\). (We recall that the operator \(S'\) is defined in (\ref{Sprimedef})).
\end{proof}

\subsection{First improvement of the approximation}
\label{firstimprovementsubsect}
So far we have recorded a number of preliminary results with a view towards the construction of an accurate \(n\)-vortex expansion. In this section we begin the expansion proper: namely, we find functions \(\phi_{j}=\phi_{j}^{(1)}\) and \(\psi^{out}=\psi^{out,1}\) which yield a first improvement of the initial approximation \(U_{\xi}\). We consider parameters of the form
\begin{equation}
\label{xifirstimprov}
    \xi(t)=\xi^{0}(t)+O(\epsilon^{2}\abs{\log\epsilon}^{2})\quad\text{in }C^{3}\text{ sense},
\end{equation}and the function \(u_{*}\) defined via (\ref{newapprox}). Our main goal is to prove the following.

\begin{proposition}
\label{firstimprovprop}
    For parameters \(\xi(t)\) as in (\ref{xifirstimprov}), there exist functions \(\phi_{j}=\phi_{j}^{(1)}(y_{j},t;\xi)\) and \(\psi^{out}=\psi^{out,1}(x,t;\xi)\) satisfying
    \begin{align*}
       \abs{\phi_{j}^{(1)}(y_{j},t)}+(1+\abs{y_{j}})\abs{\nabla_{y}\phi_{j}^{(1)}(y_{j},t)}&\leq C\epsilon^{2}\abs{\log\epsilon},\quad\text{ in}\quad B_{2\delta\epsilon^{-1}}(0)\times[0,T],\\
        \abs{\psi^{out,1}(x,t)}+\abs{\nabla_{x}\psi^{out,1}(x,t)}&\leq C\epsilon^{2}\abs{\log\epsilon}^{2},\quad\text{in}\quad\mathbb{R}^{2}\times[0,T],
    \end{align*}such that the following error bounds hold:
    \begin{gather*}
    \abs{S(u_{*})}(y,t)+\abs{\nabla_{y}\big(S(u_{*})\big)}(y,t)\leq C\epsilon^{3}\abs{\log\epsilon}^{2},\quad\text{in}\quad \mathbb{R}^{2}\times[0,T],\\
        \int_{\mathbb{R}^{2}}\abs{S(u_{*})}^{2}(y,t)+\abs{\nabla_{y}\big(S(u_{*})\big)}^{2}(y,t)\:dy\leq C\epsilon^{6}\abs{\log\epsilon}^{5},\quad\text{for all}\quad t\in[0,T].
    \end{gather*}
\end{proposition}

In short, the result above says that the initial approximation \(U_{\xi}\) can be corrected by functions \(\phi_{j}^{(1)}\), \(j=1,\ldots,n\), and \(\psi^{out,1}\) to improve the size of the error from \(O(\epsilon^{2})\) to \(O(\epsilon^{3}\abs{\log\epsilon}^{2})\) in \(L^{\infty}\). Moreover, these corrections ensure the new error is finite in \(H^{1}(\mathbb{R}^{2})\)-norm (unlike the first error).

\medskip

The proof of Proposition \ref{firstimprovprop} proceeds along the lines suggested in \(\S\)\ref{approxsolsubsection}, that is, we eliminate terms in the first error \(S(U_{\xi})\) by solving elliptic equations of the form (\ref{innereq}) for \(\phi_{j}^{(1)}\), and a wave equation of the form (\ref{linearwaveeq}) for \(\psi_{1}^{out,1}=\operatorname{Re}(\psi_{1}^{out,1})\). We start with the construction of \(\phi_{j}^{(1)}\).

\smallskip

\textbf{Construction of \(\phi_{j}^{(1)}\).} Let us recall the global expression (\ref{ustarglobalerrorexp}) for the error \(S(u_{*})\) derived in the previous section, and the terms
\begin{equation}
\label{firstlineerror}
\sum_{j=1}^{n}\tilde{\eta}_{j}\bigg(\frac{S(U_{\xi})}{iU_{\xi}}+\widetilde{S}'(U_{\xi})[\psi_{j}]+\eta_{j}^{(2)}\frac{S(U_{\xi})}{iU_{\xi}}i\psi_{j}\bigg)
\end{equation}appearing on the first line (where \(\psi_{j}=(iW_{j})^{-1}\phi_{j}\)). We observe (using (\ref{stildeprime})) that for \(\psi_{j}=\psi_{j1}+i\psi_{j2}\) we have
\begin{equation*}
    \widetilde{S}'(U_{\xi})[\psi_{j}]+\eta_{j}^{(2)}\frac{S(U_{\xi})}{iU_{\xi}}i\psi_{j}=\epsilon^{2}i\partial_{t}\psi_{j}+\widetilde{L}_{j}[\psi_{j}]+\widetilde{L}_{j}^{\#}[\psi_{j}]
\end{equation*}where
\begin{gather}
    \widetilde{L}_{j}[\psi_{j}]:=\Delta_{y}\psi_{j}+2\frac{\nabla_{y} W_{j}}{W_{j}}\cdot\nabla_{y}\psi_{j}-2i\abs{W_{j}}^{2}\psi_{j2},\label{Ljtildedef}\\
    \widetilde{L}_{j}^{\#}[\psi_{j}]:=\bigg(2\frac{\nabla_{y}\alpha_{j}}{\alpha_{j}}-\epsilon i\dot{\xi}_{j}\bigg)\cdot\nabla_{y}\psi_{j}-2i\abs{W_{j}}^{2}\big(\abs{\alpha_{j}}^{2}-1\big)\psi_{j2}+\eta_{j}^{(2)}\frac{S(U_{\xi})}{iU_{\xi}}i\psi_{j}.\label{tildeLjdef}
\end{gather}(Here \(\alpha_{j}:=\prod_{k\neq j}W_{k}\) as defined in (\ref{alphajdef})). The operator \(\widetilde{L}_{j}[\psi_{j}]\) corresponds to the elliptic linearized operator around \(W_{j}\) in the \(\psi_{j}\) variable; in other words, we have \(L_{j}[iW_{j}\psi_{j}]=iW_{j}\widetilde{L}_{j}[\psi_{j}]\) where \(L_{j}\) is the operator (\ref{Ljdef}). On the other hand, \(\widetilde{L}_{j}^{\#}[\psi_{j}]\) is a ``lower order'' operator consisting of linear terms which are small compared to \(\widetilde{L}_{j}[\psi_{j}]\) in the region where \(\tilde{\eta}_{j}\) is supported. 

\medskip

Turning now to the expression for the first error \(S(U_{\xi})\), let us fix \(j\in\{1,\ldots,n\}\) and use polar coordinates \((r,\theta)\) for \(y_{j}=y-\tilde{\xi}_{j}\). It follows from Lemma \ref{refinederrorexp} that
\begin{equation*}
    \frac{S(u_{\xi})}{iU_{\xi}}=\epsilon\frac{\nabla_{y} W_{j}}{W_{j}}\cdot\bigg(-\dot{\xi}_{j}+d_{j}\nabla_{\xi_{j}}^{\perp}K(\xi)\bigg)+\big(R_{j1}+iR_{j2}\big)+\big(\widetilde{R}_{j1}+i\widetilde{R}_{j2}\big)
\end{equation*}in the region \(\abs{y_{j}}\leq2\delta\epsilon^{-1}\), where \(K(\xi)\) is the Helmholtz-Kirchhoff functional (\ref{kirchhofffunct}). The functions \(R_{j1}=O(\epsilon^{2}(1+r^{2})^{-1})\) and \(R_{j2}=O(\epsilon^{2})\) are defined in (\ref{Rj1def}) and (\ref{Rj2def}) respectively, and the remaining terms (defined so the above expansion holds) have size \(\widetilde{R}_{j1}=O(\epsilon^{3}r^{-1})\) and \(\widetilde{R}_{j2}=O\big(\epsilon^{3}(r^{-1}+r)\big)\).  

\smallskip

Our first step towards an improvement of the approximation is to eliminate \((R_{j1}+iR_{j2})\) and \((\widetilde{R}_{j1}+i\widetilde{R}_{j2})\) using the operator (\ref{Ljtildedef}). It is convenient to introduce a new smooth cut-off \begin{equation}\label{etaj9quarter}
    \tilde{\eta}_{j}^{(9/4)}(y_{j}):=1\quad\text{if}\quad\abs{y_{j}}\leq2\delta\epsilon^{-1},\quad\quad\tilde{\eta}_{j}^{(9/4)}(y_{j}):=0\quad\text{if}\quad\abs{y_{j}}\geq\tfrac{9}{4}\delta\epsilon^{-1},
    \end{equation}so the equation for \(\widetilde{R}_{j1}+i\widetilde{R}_{j2}\) can be considered in the entire space \(\mathbb{R}^{2}\). 

\begin{lemma}
\label{psij11and12construction}
    There exist functions 
    \begin{equation*}
    \psi_{j}^{(1,1)}=\psi_{j}^{(1,1)}(y_{j},t;\xi)\quad\text{and}\quad\psi_{j}^{(1,2)}=\psi_{j}^{(1,2)}(y_{j},t;\xi)    \end{equation*} satisfying
    \begin{align}
    \widetilde{L}_{j}[\psi_{j}^{(1,1)}]+(R_{j1}+iR_{j2})&=0,\quad\text{in}\quad\mathbb{R}^{2},\quad\text{for all}\quad t\in[0,T],\label{firstellipimprov1}\\
    \widetilde{L}_{j}[\psi_{j}^{(1,2)}]+\tilde{\eta}_{j}^{(9/4)}(\widetilde{R}_{j1}+i\widetilde{R}_{j2})&=0,\quad\text{in}\quad\mathbb{R}^{2},\quad\text{for all}\quad t\in[0,T],\label{firstellipimprov2}
    \end{align}such that
\begin{gather}
   \sup_{\mathbb{R}^{2}}\:\abs{\psi_{j}^{(1,1)}}\leq C\epsilon^{2},\notag\\
   \sup_{r\leq2}\:\abs{\psi_{j}^{(1,2)}}+\sup_{r\geq2}\:\abs{r^{-1}(\log r)^{-1}\psi_{j1}^{(1,2)}}+\sup_{r\geq2}\:\abs{r^{-1}\psi_{j2}^{(1,2)}}\leq C\epsilon^{3}.\label{firstellipimprov2est}
\end{gather}
        (Here \(r=\abs{y_{j}}\) and \(\psi_{j}^{(1,2)}=\psi_{j1}^{(1,2)}+i\psi_{j2}^{(1,2)}\)).
\end{lemma}

\begin{proof}
    First suppose \(j\in I_{+}\). We recall from Lemma \ref{refinederrorexp} that \(R_{j1}\) contains only mode 2 terms in its Fourier expansion, and \(R_{j2}\) contains only mode 2 and mode 0 terms in its Fourier expansion. Thus, using the notation introduced in (\ref{fourierprojop}) we can write \(R_{j}:=R_{j1}+iR_{j2}\) in the form
    \begin{equation*}
    \begin{split}
        R_{j}=&\quad i(P_{2}^{0}R_{j})(r)\\
        &+\big((P^{1}_{21}R_{j})(r)\cos2\theta+(P^{2}_{21}R_{j})(r)\sin2\theta\big)+i\big((P^{1}_{22}R_{j})(r)\sin2\theta+(P^{2}_{22}R_{j})(r)\cos2\theta\big),
    \end{split}    
    \end{equation*}where the terms in the real part have size \(O(\epsilon^{2}(1+r^{2})^{-1})\), and the terms in the imaginary part have size \(O(\epsilon^{2})\). We intend to build a solution \(\psi_{j}^{(1,1)}\) of (\ref{firstellipimprov1}) with an analogous expansion.

    Let us observe that (\ref{firstellipimprov1}) corresponds to equation (\ref{ellipticlinearpsi}) with \(\psi=\psi_{j}^{(1,1)}\) and \(\elliprhs=-R_{j}\); this problem then reduces to a collection of ODEs for the Fourier modes as described in \(\S\)\ref{linprobsect}. By \mbox{Proposition \ref{mode0formulaeprop}}, we can obtain a solution for the imaginary component \(P_{2}^{0}\psi_{j}^{(1,1)}\) in mode zero via the formula
    \begin{equation}
    \label{firstimprovmode0}
    \big(P_{2}^{0}\psi_{j}^{(1,1)}\big)(r)=-z_{1,0}(r)\int_{r}^{\infty}w(s)^{2}s(P_{2}^{0}R_{j})(s)z_{2,0}(s)\:ds-z_{2,0}(r)\int_{0}^{r}w(s)^{2}s(P_{2}^{0}R_{j})(s)z_{1,0}(s)\:ds,
    \end{equation}
    where \(z_{1,0}\) and \(z_{2,0}\) are the functions described in (\ref{mode0kernelprop}). Moreover, a solution can be built in mode 2 (\textit{cf.} Proposition \ref{modegeq1formulaeprop}) using the formula 
    \begin{equation}
    \label{firstimprovmodegeq1}
\begin{split}
    \Psi(r)=&\quad\left(\int_{0}^{r} w(s)^{2}s\:\tilde{\elliprhs}(s)\cdot z_{2,2}(s)\:ds\right)z_{1,2}(r)\\
    &+\left(\int_{r}^{\infty} w(s)^{2}s\:\tilde{\elliprhs}(s)\cdot z_{1,2}(s)\:ds\right)z_{2,2}(r)\\
    &-\left(\int_{r}^{\infty} w(s)^{2}s\:\tilde{\elliprhs}(s)\cdot z_{4,2}(s)\:ds\right)z_{3,2}(r)\\
    &-\left(\int_{0}^{r} w(s)^{2}s\:\tilde{\elliprhs}(s)\cdot z_{3,2}(s)\:ds\right)z_{4,2}(r),
\end{split}
\end{equation} where \begin{equation*}
    \Psi=\begin{cases}
    \big(P_{21}^{\nu}\psi_{j}^{(1,1)},\:P_{22}^{\nu}\psi_{j}^{(1,1)}\big),\quad\hspace{0.8em}\text{for }\nu=1,\\
    \big(P_{21}^{\nu}\psi_{j}^{(1,1)},\:-P_{22}^{\nu}\psi_{j}^{(1,1)}\big),\quad\text{for }\nu=2,
   \end{cases}\quad
   \tilde{\elliprhs}=\begin{cases}
    \big(-P_{21}^{\nu}R_{j},-P_{22}^{\nu}R_{j}\big),\quad\text{for }\nu=1,\\
    \big(-P_{21}^{\nu}R_{j},\:P_{22}^{\nu}R_{j}\big),\quad\hspace{0.55em}\text{for }\nu=2,
   \end{cases}
\end{equation*}and \(z_{1,2},\ldots,z_{4,2}\) are the functions introduced in (\ref{modegeq2kernelasymprop}) (with \(k=2\)). Direct verification using (\ref{firstimprovmode0}), (\ref{firstimprovmodegeq1}) gives \(O(\epsilon^{2})\) bounds for the Fourier modes: we then get a solution \(\psi_{j}^{(1,1)}\) of (\ref{firstellipimprov1}) satisfying the stated estimate.

Turning now to (\ref{firstellipimprov2}), this problem is of the form (\ref{ellipticlinearpsi}) where \(\elliprhs=-\tilde{\eta}_{j}^{(9/4)}(\widetilde{R}_{j1}+i\widetilde{R}_{j2})\) has nonzero modes at all values of \(k\geq0\) in its Fourier expansion. We decompose \(\psi_{j}^{(1,2)}\) and \(\widetilde{R}_{j}:=\tilde{\eta}_{j}^{(9/4)}(\widetilde{R}_{j1}+i\widetilde{R}_{j2})\) according to (\ref{fmodesum})-(\ref{fourierprojop}). Then the corresponding ODE for \(P_{1}^{0}\psi_{j}^{(1,2)}\) can be solved (\textit{cf.} Proposition \ref{mode0formulaeprop}) via the formula
\begin{equation*}
    P_{1}^{0}\psi_{j}^{(1,2)}(r)=-\int_{0}^{r}\frac{ds}{w(s)^{2}s}\int_{0}^{s} w(t)^{2}t\big(P_{1}^{0}\widetilde{R}_{j}\big)(t)\:dt,
\end{equation*}and the ODE for \(P_{2}^{0}\psi_{j}^{(1,2)}\) can be solved using (\ref{firstimprovmode0}) with \(P_{2}^{0}R_{j}\) replaced by \(P_{2}^{0}\widetilde{R}_{j}\). For Fourier modes \(k\geq1\), we build a solution using formula (\ref{firstimprovmodegeq1}) with \(z_{1,2},\ldots,z_{4,2}\) replaced by \(z_{1,k},\ldots,z_{4,k}\), and with the integral on the second line replaced by an integral from \(0\) to \(r\) if \(k=1\). It can be verified that each Fourier mode satisfies estimate (\ref{firstellipimprov2est}). We then get the existence of \(\psi_{j}^{(1,2)}\) satisfying (\ref{firstellipimprov2}) and (\ref{firstellipimprov2est}) by summing all Fourier modes and using maximum principle type estimates as described in Remark \ref{summingfouriermodes}. 

In the case \(j\in I_{-}\), we can use the conjugate symmetry of \(L_{j}\) as described in Remark \ref{conjugatesymmetry} to rewrite the given problems in the form (\ref{ellipticlinearpsi}). The existence of solutions \(\psi_{j}^{(1,1)}\) and \(\psi_{j}^{(1,2)}\) with the stated estimates then follows as above by decomposing in Fourier modes. This completes the proof of the lemma.  
\end{proof}

\begin{remark}[Derivative estimates]\label{derivestremark}
   For each integer \(\ell\geq1\), the modulus \(w(r)\) of the degree-one vortex satisfies 
   \begin{equation}
   \label{derivativesofw}
    \frac{d^{\ell}w}{dr^{\ell}}=\frac{(-1)^{\ell+1}(k+1)!}{2}r^{-\ell-2}+O(r^{-\ell-4}),\quad\text{as}\quad{r\to\infty}.
   \end{equation}(This can be deduced, for example, using the methods in \cites{chenelliottqi1994,herveherve1994}). Property (\ref{derivativesofw}) implies that
   \begin{align*}
       r^{2}\abs{D_{y}^{\ell}R_{j1}}+\abs{D_{y}^{\ell}R_{j2}}&\leq C_{\ell}\epsilon^{2}r^{-\ell},\quad\hspace{0.4em}\text{for}\quad r\geq2,\quad\ell\geq1,\\
       r^{2}\abs{D_{y}^{\ell}\widetilde{R}_{j1}}+\abs{D_{y}^{\ell}\widetilde{R}_{j2}}&\leq C_{\ell}\epsilon^{3}r^{1-\ell},\quad\text{for}\quad r\geq2,\quad\ell\geq1,
       \end{align*}and one can then use rescaled Schauder estimates and the representation formula for \((-\Delta+2)^{-1}\) as in Remark \ref{derivativeestimates} to get the derivative estimates
       \begin{gather*}
           \abs{D^{\ell}_{y}\psi_{j}^{(1,1)}}\leq C_{\ell}\epsilon^{2}r^{-\ell},\quad\text{for}\quad r\geq2,\quad\ell\geq1,\\
           (\log r)^{-1}\abs{\nabla_{y}\psi_{j1}^{(1,2)}}+\abs{\nabla_{y}\psi_{j2}^{(1,2)}}\leq C\epsilon^{3},\quad\text{for}\quad r\geq2,\\
           \abs{D_{y}^{\ell}\psi_{j}^{(1,2)}}\leq C_{\ell}\epsilon^{3}r^{1-\ell}(\log r),\quad\text{for}\quad r\geq2, \quad \ell\geq2,
       \end{gather*}for the solutions in the previous lemma. Moreover, standard elliptic estimates imply that \(\phi_{j}^{(1,1)}:=iW_{j}\psi_{j}^{(1,1)}\) and \(\phi_{j}^{(1,2)}:=iW_{j}\psi_{j}^{(1,2)}\) have bounded spatial derivatives of size \(O(\epsilon^{2})\) and \(O(\epsilon^{3})\) respectively for \(r\leq2\). Differentiating (\ref{firstellipimprov1}) and (\ref{firstellipimprov2}) in \(t\), we get
       \begin{gather*}
   \sup_{\mathbb{R}^{2}}\:\abs{\partial_{t}\psi_{j}^{(1,1)}}\leq C\epsilon^{2}\norm{\xi}_{C^{2}[0,T]}\leq C\epsilon^{2},\\
   \sup_{r\leq2}\:\abs{\partial_{t}\psi_{j}^{(1,2)}}+\sup_{r\geq2}\:\abs{r^{-1}(\log r)^{-1}\partial_{t}\psi_{j1}^{(1,2)}}+\sup_{r\geq2}\:\abs{r^{-1}\partial_{t}\psi_{j2}^{(1,2)}}\leq C\epsilon^{3}\norm{\xi}_{C^{2}[0,T]}\leq C\epsilon^{3},
\end{gather*}where
\begin{equation*}
    \norm{\xi}_{C^{2}[0,T]}:=\max_{\substack{j=1,\ldots,n\\ t\in[0,T]}}\bigg(\abs{\xi_{j}(t)}+\abs{\dot{\xi}_{j}(t)}+\abs{\ddot{\xi}_{j}(t)}\bigg).
\end{equation*}
\end{remark}

Let us return to (\ref{firstlineerror}). After eliminating \((R_{j1}+iR_{j2})\) and \((\widetilde{R}_{j1}+i\widetilde{R}_{j2})\), we have the expression
\begin{equation*}
\mathcal{E}_{j}^{(1,2)}:=\epsilon^{2}i\partial_{t}\big(\psi_{j}^{(1,1)}+\psi_{j}^{(1,2)}\big)+\widetilde{L}_{j}^{\#}[\psi_{j}^{(1,1)}+\psi_{j}^{(1,2)}]
\end{equation*}for the largest terms in the new error in the region \(\abs{y_{j}}\leq2\delta\epsilon^{-1}\) where \(\tilde{\eta}_{j}\) is supported. (Recall that \(\widetilde{L}_{j}^{\#}\) was defined in (\ref{tildeLjdef})). The bounds
\begin{equation*}
    2\frac{\nabla_{y}\alpha_{j}}{\alpha_{j}}-\epsilon i\dot{\xi}_{j}=O(\epsilon^{3})+iO(\epsilon^{2}r),\quad\abs{\alpha_{j}}^{2}-1=O(\epsilon^{2}),\quad\text{for}\quad\abs{y_{j}}\leq2\delta\epsilon^{-1},
\end{equation*}and the estimates for \(\psi_{j}^{(1,1)}\) and \(\psi_{j}^{(1,2)}\) just established imply that \(\mathcal{E}_{j}^{(1,2)}\) has size \(O\big(\epsilon^{4}(r^{-1}+1)\big)\) in the real part, and \(O\big(\epsilon^{4}(r^{-1}+\log(2+r))\big)\) in the imaginary part. We have thus achieved an \(O(\epsilon^{2})\) improvement of the error size close to the vortices, but in the real part the error size is the same as the initial error for \(r\sim\epsilon^{-1}\). 

To get a large \(r\) improvement we invert the elliptic operator \(\widetilde{L}_{j}\) again. We recall that \(\tilde{\eta}_{j}^{(9/4)}\) is a smooth cut-off equal to \(1\) if \(\abs{y_{j}}\leq2\delta\epsilon^{-1}\) and equal to \(0\) if \(\abs{y_{j}}\geq\tfrac{9}{4}\delta\epsilon^{-1}\).  

\begin{lemma}
\label{psij13construction}
    There exists \(\psi_{j}^{(1,3)}=\psi_{j}^{(1,3)}(y_{j},t;\xi)\) satisfying
    \begin{equation*}
    \widetilde{L}_{j}[\psi_{j}^{(1,3)}]+\tilde{\eta}_{j}^{(9/4)}\mathcal{E}_{j}^{(1,2)}=0, \quad\text{in}\quad\mathbb{R}^{2},\quad\text{for all}\quad t\in[0,T],    
    \end{equation*}
    such that
\begin{equation*}
    \sup_{r\leq2}\:\abs{\psi_{j}^{(1,3)}}+\sup_{r\geq2}\:\abs{r^{-2}(\log r)^{-1}\psi_{j1}^{(1,3)}}+\sup_{r\geq2}\:\abs{(\log r)^{-1}\psi_{j2}^{(1,3)}}\leq C\epsilon^{4}.
\end{equation*}
\end{lemma}

\begin{proof}
The proof follows a similar argument to Lemma 4.11 by decomposing in Fourier modes, and using the bounds on \(\mathcal{E}_{j}^{(1,2)}\) stated above. We omit the details.
\end{proof}

\begin{remark}\label{derivestremark2}
For the function \(\psi_{j}^{(1,3)}\) above we get the derivative estimates
\begin{align*}
    \abs{D_{y}^{\ell}\psi_{j1}^{(1,3)}}&\leq C\epsilon^{4}\abs{\log r}r^{2-\ell},\quad\text{for}\quad r\geq2,\quad\ell\geq1,\\
    \abs{D_{y}^{\ell}\psi_{j2}^{(1,3)}}&\leq C\epsilon^{4}\abs{\log r}r^{-\ell},\quad\hspace{0.37em}\text{for}\quad r\geq2,\quad\ell\geq1.
\end{align*}Moreover, \(\phi_{j}^{(1,3)}:=iW_{j}\psi_{j}^{(1,3)}\) has bounded spatial derivatives of size \(O(\epsilon^{4})\) for \(r\leq2\), and we have
\begin{equation*}
    \sup_{r\leq2}\:\abs{\partial_{t}\psi_{j}^{(1,3)}}+\sup_{r\geq2}\:\abs{r^{-2}(\log r)^{-1}\partial_{t}\psi_{j1}^{(1,3)}}+\sup_{r\geq2}\:\abs{(\log r)^{-1}\partial_{t}\psi_{j2}^{(1,3)}}\leq C\epsilon^{4}.
\end{equation*}    
\end{remark}

We now complete the construction of \(\phi_{j}^{(1)}:=iW_{j}\psi_{j}^{(1)}\) by setting 
\begin{equation*}
    \psi_{j}^{(1)}:=\psi_{j}^{(1,1)}+\psi_{j}^{(1,2)}+\psi_{j}^{(1,3)},
\end{equation*}where the superscript \((1)\) denotes first inner improvement, and the superscripts \((1,1)\), \((1,2)\), \((1,3)\) denote sub-improvements in the first step. With this choice of \(\psi_{j}=\psi_{j}^{(1)}\) we get the expression \(\sum_{j}\tilde{\eta}_{j}\mathcal{E}_{j}^{(1,3)}\) for the inner error (\ref{firstlineerror}), where
\begin{equation}
\label{innererrorafter3improv}
\mathcal{E}_{j}^{(1,3)}:=\epsilon\frac{\nabla_{y} W_{j}}{W_{j}}\cdot\bigg(-\dot{\xi}_{j}+d_{j}\nabla_{\xi_{j}}^{\perp}K(\xi)\bigg)+\epsilon^{2}i\partial_{t}\psi_{j}^{(1,3)}+\widetilde{L}_{j}^{\#}[\psi_{j}^{(1,3)}].
\end{equation}It is then readily verified that \(\mathcal{E}_{j}^{(1,3)}=\mathcal{E}_{j1}^{(1,3)}+i\mathcal{E}_{j2}^{(1,3)}\) satisfies 
\begin{align*}
\mathcal{E}_{j1}^{(1,3)}&=\epsilon\frac{\nabla_{y}w_{j}}{w_{j}}\cdot\bigg(-\dot{\xi}_{j}+d_{j}\nabla_{\xi_{j}}^{\perp}K(\xi)\bigg)+O\big(\epsilon^{6}\abs{\log\epsilon}(r^{-1}+1)\big),\\
\mathcal{E}_{j2}^{(1,3)}&=\epsilon d_{j}\nabla_{y}\theta_{j}\cdot\bigg(-\dot{\xi}_{j}+d_{j}\nabla_{\xi_{j}}^{\perp}K(\xi)\bigg)+O\big(\epsilon^{6}\abs{\log\epsilon}(r^{-1}+r^{2})\big),
\end{align*}for \(r\leq2\delta\epsilon^{-1}\). We have thus succeeded in improving the initial inner error of \mbox{Lemma \ref{refinederrorexp}}.

\medskip

\textbf{Construction of \(\psi^{out,1}\).} The next step towards a global improvement of the approximation is to improve the size of the outer error. Let us return to the expression (\ref{ustarglobalerrorexp}) for \(S(u_{*})\), with a focus now on the terms
\begin{equation}
\begin{split}
&\sum_{j=1}^{n}\bigg(\epsilon^{2}i\partial_{t}\tilde{\eta}_{j}+\Delta_{y}\tilde{\eta}_{j}+2\frac{\nabla_{y}U_{\xi}}{U_{\xi}}\cdot\nabla_{y}\tilde{\eta}_{j}\bigg)\psi_{j}+2\sum_{j=1}^{n}\nabla_{y}\tilde{\eta}_{j}\cdot\nabla_{y}\psi_{j}\\
           &+\bigg(1-\sum_{j=1}^{n}\tilde{\eta}_{j}\bigg)\frac{S(U_{\xi})}{iU_{\xi}}+\widetilde{S}'(U_{\xi})[\psi^{out}]
\end{split}           
\end{equation}appearing on the second and third lines. We observe that \((1-\sum_{j}\tilde{\eta}_{j})\frac{S(U_{\xi})}{iU_{\xi}}\) is supported in the region \(\abs{x-\xi_{j}}\geq\delta\) for all \(j=1,\ldots,n\), where \(x=\epsilon y\) is the original space variable. Moreover, the expression
\begin{equation}
\label{Jdef}
    \mathcal{J}[\psi^{in}]:=\sum_{j=1}^{n}\bigg(\epsilon^{2}i\partial_{t}\tilde{\eta}_{j}+\Delta_{y}\tilde{\eta}_{j}+2\frac{\nabla_{y}U_{\xi}}{U_{\xi}}\cdot\nabla_{y}\tilde{\eta}_{j}\bigg)\psi_{j}+2\sum_{j=1}^{n}\nabla_{y}\tilde{\eta}_{j}\cdot\nabla_{y}\psi_{j},
\end{equation}involves the inner corrections \(\psi^{in}=(\psi_{1},\ldots,\psi_{n})\) and the derivatives of \(\tilde{\eta}_{j}\), and is thus supported in the annular regions \(\delta\leq\abs{x-\xi_{j}}\leq2\delta\) for \mbox{\(j=1,\ldots,n\)}.

\medskip

Setting \(\psi_{j}=\psi_{j}^{(1)}\), it is directly verified that \(\mathcal{J}[\psi_{1}^{(1)},\ldots,\psi_{n}^{(1)}]\) has size \(O(\epsilon^{4}\abs{\log\epsilon})\) in both components in the region where it is nonzero. We further note that
\begin{equation*}
    \big(1-\sum_{j=1}^{n}\tilde{\eta}_{j}\big)\frac{S(U_{\xi})}{iU_{\xi}}=\chi R_{1}+i\chi R_{2}
\end{equation*}where \(\chi(x,t):=1-\sum_{j}\tilde{\eta}_{j}\) and \(R_{1}\) and \(R_{2}\) are the functions defined in (\ref{R1def}) and (\ref{R2def}) respectively, and it is then easy to check that we have the estimates
\begin{gather*}
    \chi R_{1}=O(\epsilon^{4})\sum_{j=1}^{n}\frac{1}{1+\abs{x-\xi_{j}}^{3}},\\
    \chi R_{2}=\epsilon^{2}\chi\sum_{j=1}^{n}d_{j}\frac{(x-\xi_{j})^{\perp}}{\abs{x-\xi_{j}}^{2}}\cdot(-\dot{\xi}_{j})+O(\epsilon^{2})\sum_{j=1}^{n}\frac{1}{1+\abs{x-\xi_{j}}^{2}}.
\end{gather*}

We intend to improve the error sizes mentioned above by eliminating \((1-\sum_{j}\tilde{\eta}_{j})\frac{S(U_{\xi})}{iU_{\xi}}\) and \(\mathcal{J}[\psi^{(1)}_{1},\ldots,\psi_{n}^{(1)}]\) with a suitable choice of \(\psi^{out}=\psi^{out}_{1}+i\psi^{out}_{2}\). To this end we note that the operator \(\widetilde{S}'(U_{\xi})\) takes the form
\begin{equation*}
\widetilde{S}'(U_{\xi})[\psi^{out}]=\epsilon^{2}i\psi^{out}_{t}+\epsilon^{2}\Delta_{x}\psi^{out}+2\epsilon^{2}\frac{\nabla_{x}U_{\xi}}{U_{\xi}}\cdot\nabla_{x}\psi^{out}-2i\abs{U_{\xi}}^{2}\psi^{out}_{2}
\end{equation*}in the \(x\) variable, and for smooth functions \(E_{1}(x,t)\) and \(E_{2}(x,t)\) the linear problem 
\begin{equation}
\label{outerlinearprob}
\widetilde{S}'(U_{\xi})[\psi^{out}]+\big(E_{1}+iE_{2}\big)=0
\end{equation}
reads
\begin{align*}
   -\epsilon^{2}\partial_{t}\psi_{2}^{out}+\epsilon^{2}\Delta_{x}\psi_{1}^{out}+2\epsilon^{2}\frac{\nabla_{x}\abs{U_{\xi}}}{\abs{U_{\xi}}}\cdot\nabla_{x}\psi_{1}^{out} -2\epsilon^{2}\nabla_{x}\varphi_{\xi}\cdot\nabla_{x}\psi_{2}^{out}+E_{1}&=0,\\
    \epsilon^{2}\partial_{t}\psi_{1}^{out}+\epsilon^{2}\Delta_{x}\psi_{2}^{out}+2\epsilon^{2}\frac{\nabla_{x}\abs{U_{\xi}}}{\abs{U_{\xi}}}\cdot\nabla_{x}\psi_{2}^{out}+2\epsilon^{2}\nabla_{x}\varphi_{\xi}\cdot\nabla_{x}\psi_{1}^{out}-2\abs{U_{\xi}}^{2}\psi_{2}^{out}+E_{2}&=0,
\end{align*}where
\begin{equation}
    \nabla_{x}\varphi_{\xi}(x,t):=\sum_{j=1}^{n}d_{j}\frac{(x-\xi_{j})^{\perp}}{\abs{x-\xi_{j}}^{2}}.
\end{equation}Since \(\abs{U_{\xi}}^{2}=1+O(\epsilon^{2})\) and \(\nabla_{x}\abs{U_{\xi}}=O(\epsilon^{2})\) in the region \(\abs{x-\xi_{j}}\geq\delta\) for all \(j\), it is reasonable to approximate the terms
\begin{equation*}
    \epsilon^{2}\Delta_{x}\psi_{2}^{out}+2\epsilon^{2}\frac{\nabla_{x}\abs{U_{\xi}}}{\abs{U_{\xi}}}\cdot\nabla_{x}\psi_{2}^{out}-2\abs{U_{\xi}}^{2}\psi_{2}^{out}
    \end{equation*}
with \(-2\psi_{2}^{out}\) far from the vortices. Neglecting also the term \(2\epsilon^{2}\nabla_{x}\varphi_{\xi}\cdot\nabla_{x}\psi_{1}^{out}\), we anticipate that an approximate solution to (\ref{outerlinearprob}) can be obtained in the imaginary part far from the vortices by setting
\begin{equation}
\label{psi2intermspsi1}
    \psi_{2}^{out}=\frac{1}{2}\big(E_{2}+\epsilon^{2}\partial_{t}\psi_{1}^{out}\big).
\end{equation}

Let us now substitute (\ref{psi2intermspsi1}) into the real part of (\ref{outerlinearprob}): we get
\begin{equation}
\label{realwavexp}
\begin{split}
    \operatorname{Re}\widetilde{S}'(U_{\xi})[\psi^{out}]+E_{1}&=\epsilon^{2}\bigg(-\frac{\epsilon^{2}}{2}\partial_{tt}\psi_{1}^{out}+\Delta_{x}\psi_{1}^{out}+\epsilon^{-2}E_{1}-\frac{1}{2}\partial_{t}E_{2}-\nabla_{x}\varphi_{\xi}\cdot\nabla_{x}E_{2}\bigg)\\
    &\hspace{1em}-\epsilon^{4}\nabla_{x}\varphi_{\xi}\cdot\nabla_{x}\partial_{t}\psi_{1}^{out}+2\epsilon^{2}\frac{\nabla_{x}\abs{U_{\xi}}}{\abs{U_{\xi}}}\cdot\nabla_{x}\psi_{1}^{out}.
\end{split}    
\end{equation}The terms on the second line of (\ref{realwavexp}) are formally of smaller order, as can be seen by rewriting all functions in the rescaled time variable \(\tau:=\sqrt{2}\epsilon^{-1}t\). Then we have 
\begin{equation*}
    \begin{split}
\operatorname{Re}\widetilde{S}'(U_{\xi})[\psi^{out}]+E_{1}&=\epsilon^{2}\bigg(-\partial_{\tau\tau}\psi_{1}^{out}+\Delta_{x}\psi_{1}^{out}+\epsilon^{-2}E_{1}-\frac{\epsilon^{-1}}{\sqrt{2}}\partial_{\tau}E_{2}-\nabla_{x}\varphi_{\xi}\cdot\nabla_{x}E_{2}\bigg)\\
    &\hspace{1em}-\epsilon^{3}\sqrt{2}\nabla_{x}\varphi_{\xi}\cdot\nabla_{x}\partial_{\tau}\psi_{1}^{out}+2\epsilon^{2}\frac{\nabla_{x}\abs{U_{\xi}}}{\abs{U_{\xi}}}\cdot\nabla_{x}\psi_{1}^{out},
    \end{split}
\end{equation*}and we expect to eliminate the terms involving \(E_{1}\), \(E_{2}\) up to a small remainder by solving the linear wave equation
\begin{equation*}
    -\partial_{\tau\tau}\psi_{1}^{out}+\Delta_{x}\psi_{1}^{out}+\epsilon^{-2}E_{1}-\frac{\epsilon^{-1}}{\sqrt{2}}\partial_{\tau}E_{2}-\nabla_{x}\varphi_{\xi}\cdot\nabla_{x}E_{2}=0.
\end{equation*}

The construction of \(\psi^{out,1}=\psi_{1}^{out,1}+i\psi_{2}^{out,1}\) now proceeds by implementing the reasoning above with 
\begin{align}
    E_{1}&=E_{1}^{out,1}:=\operatorname{Re}\big(\mathcal{J}[\psi_{1}^{(1)},\ldots,\psi_{n}^{(1)}]\big)+\chi R_{1},\label{E1outdef}\\
    E_{2}&=E_{2}^{out,1}:=\operatorname{Im}\big(\mathcal{J}[\psi_{1}^{(1)},\ldots,\psi_{n}^{(1)}]\big)+\chi R_{2}.\label{E2outdef}
\end{align}We have the following result.

\begin{lemma}
\label{psiout1construction}
Let
\begin{equation}
\label{Fout1def}
    F^{out,1}(x,\tau):=\frac{\epsilon^{-1}}{\sqrt{2}}\partial_{\tau}E_{2}^{out,1}+\nabla_{x}\varphi_{\xi}\cdot\nabla_{x}E_{2}^{out,1}-\epsilon^{-2}E_{1}^{out,1}
\end{equation}where \(\tau=\sqrt{2}\epsilon^{-1}t\) and \(E_{1}^{out,1}\), \(E_{2}^{out,1}\) are defined as in (\ref{E1outdef}), (\ref{E2outdef}). Then there exists a solution \(\psi_{1}^{out,1}(x,\tau)\) of the linear wave equation
\begin{equation}\label{psiout1eq}
 \begin{cases}
-\partial_{\tau\tau}\psi_{1}^{out,1}+\Delta_{x}\psi_{1}^{out,1}=F^{out,1},&\quad\text{in}\quad\mathbb{R}^{2}\times(0,\sqrt{2}\epsilon^{-1}T),\\[4pt]
    \psi_{1}^{out,1}(x,0)=\partial_{\tau}\psi_{1}^{out,1}(x,0)=0,&\quad\text{in}\quad\mathbb{R}^{2},   
    \end{cases} 
\end{equation}such that
\begin{equation*}
    \abs{\psi_{1}^{out,1}(x,\tau)}+\abs{\partial_{\tau}\psi_{1}^{out,1}(x,\tau)}+\abs{\nabla_{x}\psi_{1}^{out,1}(x,\tau)}\leq C\epsilon^{2}\abs{\log\epsilon}^{2}, \quad\text{in}\quad\mathbb{R}^{2}\times(0,\sqrt{2}\epsilon^{-1}T).
\end{equation*}Furthermore, if we set
\begin{equation*}
\psi_{2}^{out,1}(x,\tau):=\frac{1}{2}\big(E_{2}^{out,1}+\epsilon\sqrt{2}\partial_{\tau}\psi_{1}^{out,1}\big)
\end{equation*}then
\begin{equation*}
    \abs{\psi_{2}^{out,1}(x,\tau)}+\abs{\nabla_{x}\psi_{2}^{out,1}(x,\tau)}\leq C\epsilon^{2}\abs{\log\epsilon},\quad\text{in}\quad\mathbb{R}^{2}\times(0,\sqrt{2}\epsilon^{-1}T).\end{equation*}
\end{lemma}

\begin{proof}
Using the expressions (\ref{R1def}) and (\ref{R2def}) for \(R_{1}\) and \(R_{2}\) and the expression (\ref{Jdef}) for \(\mathcal{J}[\psi^{in}]\), it is directly verified that we can decompose (\ref{Fout1def}) as
\begin{equation*}
F^{out,1}=F_{a}^{out,1}+F_{b}^{out,1}+F_{c}^{out,1}
\end{equation*}where
\begin{equation*}
    F_{a}^{out,1}:=\frac{1}{2}\epsilon^{2}\chi\sum_{j=1}^{n}d_{j}\frac{(x-\xi_{j})^{\perp}}{\abs{x-\xi_{j}}^{2}}\cdot(-\ddot{\xi}_{j}),
\end{equation*}and
\begin{equation*}
    F_{b}^{out,1}=O_{c}\big(\epsilon^{2}\abs{\log\epsilon}\big),\quad\quad F_{c}^{out,1}=O(\epsilon^{2})\sum_{j=1}^{n}\frac{1}{1+\abs{x-\xi_{j}}^{2}}.
\end{equation*}

Here \(\ddot{\xi}_{j}\) denotes the second derivative of \(\xi_{j}\) with respect to the original time variable \(t\), and the notation \(O_{c}(\epsilon^{2}\abs{\log\epsilon})\) denotes terms of size \(O(\epsilon^{2}\abs{\log\epsilon})\) that are compactly supported in the regions \(\delta\leq\abs{x-\xi_{j}}\leq2\delta\) for \(j=1,\ldots,n\).

\smallskip

Since
\begin{equation*}
    \frac{(x-\xi_{j})^{\perp}}{\abs{x-\xi_{j}}^{2}}-\frac{x^{\perp}}{1+\abs{x}^{2}}=O(\abs{x}^{-2}),\quad\text{as}\quad\abs{x}\to\infty,
\end{equation*}we can write
\begin{equation*}
    F_{a}^{out,1}=-\frac{1}{2}\epsilon^{2}\chi\frac{x^{\perp}}{(1+\abs{x}^{2})}\cdot\Bigg(\sum_{j=1}^{n}d_{j}\ddot{\xi}_{j}\Bigg)+O(\epsilon^{2})\sum_{j=1}^{n}\frac{1}{1+\abs{x-\xi_{j}}^{2}}.
\end{equation*}We then note that \(\sum_{j}d_{j}\xi_{j}^{0}(t)\) is a conserved quantity for the system (\ref{KirchhoffODE}) satisfied by \(\xi^{0}(t)\), hence
\begin{equation*}
    \sum_{j=1}^{n}d_{j}\ddot{\xi_{j}^{0}}=\frac{d}{dt}\Bigg(\sum_{j=1}^{n}d_{j}\dot{\xi}_{j}^{0}\Bigg)=0
\end{equation*}and 
\begin{equation*}
    \abs{F_{a}^{out,1}}\leq C\epsilon^{2}\Bigg(\sum_{j=1}^{n}\frac{1}{1+\abs{x-\xi_{j}}^{2}}+\norm{\xi-\xi^{0}}_{C^{2}[0,T]}\sum_{j=1}^{n}\frac{1}{1+\abs{x-\xi_{j}}}\Bigg).
\end{equation*}

In Section \ref{waveestimatesect}, we show that the linear wave equation
\begin{equation}\label{linearwavefirstimprov}
 \begin{cases}
-\partial_{\tau\tau}\psi+\Delta_{x}\psi=F(x,\tau),&\quad\text{in}\quad\mathbb{R}^{2}\times(0,\sqrt{2}\epsilon^{-1}T),\\[4pt]
    \psi(x,0)=\partial_{\tau}\psi(x,0)=0,&\quad\text{in}\quad\mathbb{R}^{2},  
    \end{cases} 
\end{equation}has a unique solution \(\psi(x,\tau)\) such that
\begin{equation*}
    \psi(x,\tau)=\left\{\begin{array}{ll}
        O\big((\log(1+\tau))^{2}\big),&\text{ if}\quad F(x,\tau)=O\big(\frac{1}{1+\abs{x}^{2}}\big),\\[8pt]
        O\big(\tau\log(1+\tau)\big),&\text{ if}\quad F(x,\tau)=O\big(\frac{1}{1+\abs{x}}\big),\\[8pt]   
        O\big(\log(2+\tau)\big),&\text{ if}\quad F(x,\tau)\text{ is compactly supported in }\mathbb{R}^{2}\text{ for each }\tau.
        \end{array}\right.
\end{equation*}(See Lemmas \ref{Fquaddecaylemma}, \ref{Flineardecaylemma}, \ref{Fcsupportlemma} for the precise statements). By applying these results separately to the various terms in \(F^{out,1}\) and summing the resulting estimates, we obtain a solution \(\psi_{1}^{out,1}(x,\tau)\) of (\ref{psiout1eq}) such that
\begin{equation*}
    \abs{\psi_{1}^{out,1}(x,\tau)}\leq C\epsilon^{2}\bigg(\abs{\log\epsilon}^{2}+\epsilon^{-1}\abs{\log\epsilon}\norm{\xi-\xi^{0}}_{C^{2}[0,T]}\bigg).
\end{equation*}

The assumption (\ref{xifirstimprov}) on the parameters \(\xi(t)\) implies that the right-hand side of the inequality above is bounded by \(C\epsilon^{2}\abs{\log\epsilon}^{2}\). Moreover, differentiating (\ref{psiout1eq}) with respect to the spatial variables and using the \(\tau\) derivative estimates stated in Remarks \ref{derivativesFquaddecay}, \ref{derivativesFlineardecay} and \ref{derivativesFcompsupp}, we get
\begin{equation*}
    \sum_{\substack{\ell_{1}\leq 2\\ \ell_{2}\leq m}}\abs{\partial_{\tau}^{\ell_{1}}D_{x}^{\ell_{2}}\psi_{1}^{out,1}(x,\tau)}\leq C_{m}\epsilon^{2}\abs{\log\epsilon}^{2}
\end{equation*}for any integer \(m\geq 2\). 

To conclude we note that the expression (\ref{E2outdef}) for \(E_{2}^{out,1}\) and the estimates for \(\psi_{1}^{out,1}\) just established provide an estimate for \(
\psi_{2}^{out,1}:=\frac{1}{2}\big(E_{2}^{out,1}+\epsilon\sqrt{2}\partial_{\tau}\psi_{1}^{out,1}\big)
\) of the form
\begin{equation*}
    \sum_{\substack{\ell_{1}\leq 1\\ \ell_{2}\leq m}}\abs{\partial_{\tau}^{\ell_{1}}D_{x}^{\ell_{2}}\psi_{2}^{out,1}(x,\tau)}\leq C_{m}\epsilon^{2}\abs{\log\epsilon}
\end{equation*}for \(m\geq2\). The proof of the lemma is thus complete.
\end{proof}

\begin{remark}
\label{refinedestimatespsi1out}
    The refined estimates for (\ref{linearwavefirstimprov}) stated in the proofs of Lemmas \ref{Fquaddecaylemma}, \ref{Flineardecaylemma}, \ref{Fcsupportlemma} allow one to determine precise asymptotics of \(\psi_{1}^{out,1}\) in the region \(\abs{x}\geq2\tau\). We get
    \begin{equation*}
        \abs{\partial_{\tau}^{\ell_{1}}D_{x}^{\ell_{2}}\psi_{1}^{out,1}(x,\tau)}\leq C\epsilon^{2}\abs{\log\epsilon}\left(\frac{\tau^{2-\ell_{1}}}{(1+\abs{x}^{2})^{\frac{\ell_{2}+2}{2}}}+\norm{\xi-\xi^{0}}_{C^{2}[0,T]}\frac{\tau^{2-\ell_{1}}}{(1+\abs{x})^{\ell_{2}+1}}\right)
    \end{equation*}for \(\abs{x}\geq2\tau\) and any integers \(0\leq\ell_{1}\leq2\) and \(\ell_{2}\geq0\), where \(C=C_{\ell_{2}}\).
\end{remark}    

\begin{remark}\label{angularderivpsi1out} 
In Section \ref{arbitraryapproxsect} we will need estimates for the angular derivatives of \(\psi_{1}^{out,1}\). Let us write \(x=(\abs{x}\cos\vartheta,\abs{x}\sin\vartheta)\) for the polar coordinates of \(x=(x_{1},x_{2})\in\mathbb{R}^{2}\), and define
\begin{equation*}
   \partial_{\vartheta}:=-x_{2}\partial_{x_{1}}+x_{1}\partial_{x_{2}}.
\end{equation*}We observe that \(\partial_{\vartheta}\) commutes with the wave operator \(-\partial_{\tau\tau}+\Delta_{x}\). Then differentiating (\ref{psiout1eq}) with respect to \(\vartheta\) and using the estimates stated in \(\S\)\ref{wavepointwiseestimates}, we find
\begin{equation*}
   \abs{\partial_{\vartheta}^{\ell}\psi_{1}^{out,1}(x,\tau)}\leq C_{\ell}\epsilon^{2}\bigg(\abs{\log\epsilon}^{2}+\epsilon^{-1}\abs{\log\epsilon}\norm{\xi-\xi^{0}}_{C^{2}[0,T]}\bigg)
\end{equation*}for any integer \(\ell\geq1\), and
\begin{equation*}
    \abs{\partial_{\vartheta}^{\ell}\psi_{1}^{out,1}(x,\tau)}\leq C_{\ell}\epsilon^{2}\abs{\log\epsilon}\bigg(\frac{\tau^{2}}{1+\abs{x}^{2}}+\norm{\xi-\xi^{0}}_{C^{2}[0,T]}\frac{\tau^{2}}{(1+\abs{x})}\bigg)
\end{equation*}in the region \(\abs{x}\geq2\tau\).
\end{remark}

We can now complete the proof of the main result of this section.

\begin{proof}[Proof of Proposition \ref{firstimprovprop}]

Collecting the results of Lemmas \ref{psij11and12construction}, \ref{psij13construction} and \ref{psiout1construction}, we have constructed functions 
\begin{equation*}
    \phi_{j}^{(1)}(y_{j},t):=iW_{j}\big(\psi_{j}^{(1,1)}+\psi_{j}^{(1,2)}+\psi_{j}^{(1,3)}\big)
\end{equation*}and
\begin{equation*}
\psi^{out,1}(x,\tau):=\psi_{1}^{out,1}+i\psi_{2}^{out,1}
\end{equation*}such that
\begin{align*}
       \abs{\phi_{j}^{(1)}(y_{j},t)}+(1+\abs{y_{j}})\abs{\nabla_{y}\phi_{j}^{(1)}(y_{j},t)}&\leq C\epsilon^{2}\abs{\log\epsilon},\quad\text{ in}\quad B_{2\delta\epsilon^{-1}}(0)\times[0,T],\\
        \abs{\psi^{out,1}(x,\tau)}+\abs{\nabla_{x}\psi^{out,1}(x,\tau)}&\leq C\epsilon^{2}\abs{\log\epsilon}^{2},\quad\text{in}\quad\mathbb{R}^{2}\times[0,\sqrt{2}\epsilon^{-1}T],
    \end{align*}where \(\tau=\sqrt{2}\epsilon^{-1}t\). 
    
    Using the equations satisfied by the constituent functions of \(\psi_{j}=\psi_{j}^{(1)}=(iW_{j})^{-1}\phi_{j}^{(1)}\) and \(\psi^{out}=\psi^{out,1}\), we check that the expression (\ref{ustarglobalerrorexp}) for the new error \(S(u_{*})\) takes the form
    \begin{equation}
    \label{firstnewerrorglobal}
    \begin{split}
        S(u_{*})=iu_{*}^{\eta}\Bigg(&\sum_{j=1}^{n}\tilde{\eta}_{j}\mathcal{E}_{j}^{(1,3)}+\big(-\epsilon^{3}\sqrt{2}\nabla_{x}\varphi_{\xi}\cdot\nabla_{x}\partial_{\tau}\psi_{1}^{out,1}+2\epsilon^{2}\frac{\nabla_{x}\abs{U_{\xi}}}{\abs{U_{\xi}}}\cdot\nabla_{x}\psi_{1}^{out,1}\big)\\
        &+i\big(\epsilon^{2}\Delta_{x}\psi_{2}^{out,1}+2\epsilon^{2}\frac{\nabla_{x}\abs{U_{\xi}}}{\abs{U_{\xi}}}\cdot\nabla_{x}\psi_{2}^{out,1}+2\epsilon^{2}\nabla_{x}\varphi_{\xi}\cdot\nabla_{x}\psi_{1}^{out,1}\big) \\
        &+2i(1-\abs{U_{\xi}}^{2})\psi_{2}^{out,1}+N\big(\psi^{out,1},\psi^{in,1}\big)\bigg)
        \end{split}
    \end{equation}where \(\psi^{in,1}:=\big(\psi_{1}^{(1)},\ldots,\psi_{n}^{(1)}\big)\)
    and \(N(\psi^{out},\psi^{in})\) is the nonlinear operator (\ref{nonlineartermsoutin}). A term-by-term analysis of (\ref{firstnewerrorglobal}) combined with the expression (\ref{innererrorafter3improv}) for \(\mathcal{E}_{j}^{(1,3)}\) gives
    \begin{equation*}
    \begin{split}
        S(u_{*})=iu_{*}^{\eta}\Bigg(&\tilde{\eta}_{j}\:\epsilon\frac{\nabla_{y}W_{j}}{W_{j}}\cdot\bigg(-\dot{\xi}_{j}+d_{j}\nabla_{\xi_{j}}^{\perp}K(\xi)\bigg)+2\epsilon\frac{\nabla_{y}W_{j}}{W_{j}}\cdot\nabla_{x}\psi_{1}^{out,1}\\
        &+i(1-\abs{W_{j}}^{2})\epsilon\sqrt{2}\partial_{\tau}\psi_{1}^{out,1}+O\big(\epsilon^{4}\abs{\log\epsilon}^{2}r^{-1}\big)+iO\big(\epsilon^{4}\abs{\log\epsilon}^{2}(r^{-1}+1)\big)\Bigg)\\
    \end{split}    
    \end{equation*}in the region \(\abs{y_{j}}\leq2\delta\epsilon^{-1}\), where \(r=\abs{y_{j}}\). Since \(\nabla_{x}\psi_{1}^{out,1}\) and \(\partial_{\tau}\psi_{1}^{out,1}\) have size \(O(\epsilon^{2}\abs{\log\epsilon}^{2})\) and
    \begin{equation*}
    \xi(t)=\xi^{0}(t)+O(\epsilon^{2}\abs{\log\epsilon}^{2})\quad\text{in }C^{3}\text{ sense},
    \end{equation*}we can then estimate
    \begin{equation}\label{firstnewerrorclosetoxi}
        \abs{S(u_{*})}(y,t)\leq\frac{C\epsilon^{3}\abs{\log\epsilon}^{2}}{1+\abs{y_{j}}},\quad\text{for}\quad\abs{y_{j}}\leq2\delta\epsilon^{-1}.
    \end{equation}

    In the outer region \(\abs{y_{j}}\geq2\delta\epsilon^{-1}\) for all \(j=1,\ldots,n\), it is more convenient to measure the size of \(S(u_{*})\) in the original space variable \(x=\epsilon y\). We find 
    \begin{equation}\label{firstnewerrorfarfromxi}
    \begin{split}
      \operatorname{Re}\bigg(\frac{S(u_{*})}{iu_{*}^{\eta}}\bigg)=O(\epsilon^{5}\abs{\log\epsilon}^{2})\sum_{j=1}^{n}\frac{1}{1+\abs{x-\xi_{j}}},\\
      \operatorname{Im}\bigg(\frac{S(u_{*})}{iu_{*}^{\eta}}\bigg)=O(\epsilon^{4}\abs{\log\epsilon}^{2})\sum_{j=1}^{n}\frac{1}{1+\abs{x-\xi_{j}}},
      \end{split}
    \end{equation}if \(\abs{x-\xi_{j}}\geq2\delta\) for all \(j\) and \(\abs{x}\leq2\tau\). If \(\abs{x-\xi_{j}}\geq2\delta\) for all \(j\) and \(\abs{x}\geq2\tau\), we can use the refined asymptotics for \(\psi_{1}^{out,1}\) in Remark \ref{refinedestimatespsi1out} to deduce
    \begin{equation}\label{firstnewerroroutsidelightcone}
        \abs{S(u_{*})}\leq\frac{C\epsilon^{4}\abs{\log\epsilon}^{2}\tau}{1+\abs{x}^{2}}.
    \end{equation}

    Now combining the estimates (\ref{firstnewerrorclosetoxi}), (\ref{firstnewerrorfarfromxi}) and (\ref{firstnewerroroutsidelightcone}) for \(S(u_{*})\) in the stated regions, we conclude that     
    \begin{equation*}
        \abs{S(u_{*})}(y,t)\leq C\epsilon^{3}\abs{\log\epsilon}^{2}\quad\text{globally in}\quad \mathbb{R}^{2}\times[0,T]
    \end{equation*}and
    \begin{equation*}
        \int_{\mathbb{R}^{2}}\abs{S(u_{*})}^{2}(y,t)\:dy\leq C\epsilon^{6}\abs{\log\epsilon}^{5},\quad\text{for all}\quad t\in[0,T].
    \end{equation*}
    
    The same estimates hold for \(\nabla_{y}(S(u_{*}))\), as can be seen using (\ref{firstnewerrorglobal}) and the expression (\ref{nonsingerrorexp}) for \(S(u_{*})\) (with \(\phi_{j}=\phi_{j}^{(1)}\) and \(\psi^{out}=\psi^{out,1}\)) close to the points \(\tilde{\xi}_{j}(t)\). The proof of the proposition is complete.
\end{proof}

\begin{remark}\label{zerosetfirstimprov}
   It is worth noting that the inner corrections \(\phi_{j}^{(1)}=iW_{j}\psi_{j}^{(1)}\) constructed in this section have the property that \(\psi_{j}^{(1)}(y_{j},t)\) is bounded at the origin for \(j=1,\ldots,n\). Consequently these corrections do not change the zero set of \(U_{\xi}\), i.e. the zeroes of \(u_{*}(y,t)\) are located precisely at the points \(\tilde{\xi}_{j}(t)=\epsilon^{-1}\xi_{j}(t)\) for \(j=1,\ldots,n\). 
\end{remark}

\section{Approximation to arbitrary order}
\label{arbitraryapproxsect}

Motivated by the energy estimates we intend to employ later, our aim in this section is to generalize the result of Proposition \ref{firstimprovprop}. We show that an \(n\)-vortex approximation can be constructed up to an error of arbitrary algebraic order in \(\epsilon\):

\begin{proposition}
\label{arbitraryapproxprop}
    For any integer \(m\geq4\), there exist parameters
    \begin{equation*}
        \xi(t)=\xi^{0}(t)+\sum_{k=1}^{2m-5}\xi^{k}(t),\quad\quad\xi(t):[0,T]\to\mathbb{R}^{2n},
    \end{equation*}and functions
    \begin{equation*}
        \phi_{j}(y_{j},t)=\sum_{k=1}^{2m-5}\phi_{j}^{(k)}(y_{j},t)\quad\text{for}\quad j=1,\ldots,n,\quad\quad\psi^{out}(x,\tau)=\sum_{k=1}^{2m-5}\psi^{out,k}(x,\tau),
        \end{equation*}
    such that, for \(u_{*}(y,t)\) defined by (\ref{newapprox}), we have the error bound
    \begin{equation*}
        \sum_{\ell=0}^{m}\bigg(\int_{\mathbb{R}^{2}}\abs{D_{y}^{\ell}(S(u_{*}))}^{2}(y,t)\:dy\bigg)^{1/2}\leq C_{m}\epsilon^{m},\quad\text{for all}\quad t\in[0,T].    \end{equation*}  
        
    For certain numbers \(c_{k}=c_{k,m}>0\) we can estimate
\begin{equation*}
    \begin{split}
    &\sum_{j=1}^{n}\sup_{B_{2\delta\epsilon^{-1}}(0)\times[0,T]}\bigg(\abs{\phi_{j}^{(k)}(y_{j},t)}+\big(1+\abs{y_{j}}\big)\abs{\nabla_{y}\phi_{j}^{(k)}(y_{j},t)}\bigg)\\
    &+\sup_{\mathbb{R}^{2}\times[0,\sqrt{2}\epsilon^{-1}T]}\bigg(\abs{\psi^{out,k}(x,\tau)}+\abs{\nabla_{x}\psi^{out,k}(x,\tau)}\bigg)+\sup_{t\in[0,T]}\abs{\xi^{k}(t)}\\
    &\hspace{6em}\leq\left\{\begin{array}{cl}
    C\epsilon^{2}\abs{\log\epsilon}^{2},&\text{ if}\quad k=1,\\[4pt]
    C_{k}\epsilon^{2+\frac{k}{2}}\abs{\log\epsilon}^{c_{k}},&\text{ if}\quad k\geq2.   
    \end{array}\right.
    \end{split}
\end{equation*}
    
\end{proposition}
        
The proof of Proposition \ref{arbitraryapproxprop} proceeds via an iteration argument. In the spirit of \(\S\)\ref{firstimprovementsubsect} we will solve elliptic equations for \(\phi_{j}^{(k)}\) and a wave equation for \(\psi_{1}^{out,k}=\operatorname{Re}(\psi^{out,k})\) to eliminate error terms created by preceding corrections \(\phi_{j}^{(k-1)}\) and \(\psi^{out,k-1}\). One difference compared to \(\S\)\ref{firstimprovementsubsect} is that the parameters \(\xi_{j}(t)\) have to be adjusted before solving the elliptic problems. Moreover, we will use energy estimates to control the lower-order wave corrections.

\subsection{First adjustment of the parameters}\label{firstparadjustsubsect}

Let \(\phi_{j}^{(1)}(y_{j},t)\) and \(\psi^{out,1}(x,\tau)\) be the functions constructed in \(\S\)\ref{firstimprovementsubsect}. Setting \(\phi_{j}=\phi_{j}^{(1)}\) and \(\psi^{out}=\psi^{out,1}\) in the ansatz (\ref{newapprox}), it follows from the proof of Proposition \ref{firstimprovprop} that
\begin{equation}\label{firstnewerrclose}
\begin{split}
    S(u_{*})=iu_{*}^{\eta}\Bigg(&\epsilon\frac{\nabla_{y}W_{j}}{W_{j}}\cdot\bigg(-\dot{\xi}_{j}+d_{j}\nabla_{\xi_{j}}^{\perp}K(\xi)+2\nabla_{x}\psi_{1}^{out,1}\big(\xi_{j}(t)+\epsilon y_{j},\tau\big)\bigg)\\[4pt]
    &+i\epsilon\sqrt{2}\big(1-w(r)^{2}\big)\:\partial_{\tau}\psi_{1}^{out,1}\big(\xi_{j}(t)+\epsilon y_{j},\tau\big)\\[4pt]
    &+O\big(\epsilon^{4}\abs{\log\epsilon}^{2}r^{-1}\big)+iO\big(\epsilon^{4}\abs{\log\epsilon}^{2}(r^{-1}+1)\big)\Bigg)
\end{split}    
\end{equation}in the region \(\abs{y_{j}}\leq\delta\epsilon^{-1}\), where \(r=\abs{y_{j}}\). We intend to eliminate the largest \mbox{Fourier mode 1} terms in (\ref{firstnewerrclose}) with a suitable adjustment of the points \(\xi_{j}(t)\). In precise terms, we have the following result.

\begin{lemma}
\label{firstnewerrexpcloselemma}
   For \(\abs{y_{j}}\leq\delta\epsilon^{-1}\) we can express (\ref{firstnewerrclose}) as
   \begin{equation}\label{firstnewerrtaylorexp}
\begin{split}
    S(u_{*})=iu_{*}^{\eta}\Bigg(&\epsilon\frac{\nabla_{y}W_{j}}{W_{j}}\cdot\bigg(-\dot{\xi}_{j}+d_{j}\nabla_{\xi_{j}}^{\perp}K(\xi)+2\nabla_{x}\psi_{1}^{out,1}\big(\xi_{j}(t),\tau\big)\bigg)\\
    &+i\mathcal{E}_{j}^{(2,0)}(r,t)+O\big(\epsilon^{4}\abs{\log\epsilon}^{2}r^{-1}\big)+iO\big(\epsilon^{4}\abs{\log\epsilon}^{2}(r^{-1}+1)\big)\Bigg)
\end{split}    
\end{equation}where the first line of (\ref{firstnewerrtaylorexp}) consists of terms at mode 1 in their Fourier expansion in \(y_{j}\), and \(\mathcal{E}_{j}^{(2,0)}(r,t)\) is a real-valued, radial function of size \(O(\epsilon^{3}\abs{\log\epsilon}^{2}(1+r^{2})^{-1})\). 

The first line of (\ref{firstnewerrtaylorexp}) can be eliminated at main order by a correction of the parameters
\begin{equation*}
    \xi^{1}(t)=\big(\xi^{1}_{1}(t),\ldots,\xi^{1}_{n}(t)\big)=O(\epsilon^{2}\abs{\log\epsilon}^{2})
\end{equation*}with the property that, setting \(\xi(t)=\xi^{0}(t)+\xi^{1}(t)\), we have
\begin{equation}\label{errfirstparameteradjust}
    -\dot{\xi}_{j}+d_{j}\nabla_{\xi_{j}}^{\perp}K(\xi)+2\nabla_{x}\psi_{1}^{out,1}\big(\xi_{j}(t),\tau\big)=O(\epsilon^{4}\abs{\log\epsilon}^{4}),\quad\text{for}\quad j=1,\ldots,n.
\end{equation}
\end{lemma}

\begin{proof}
    Expression (\ref{firstnewerrtaylorexp}) follows directly from (\ref{firstnewerrclose}) by Taylor expanding the terms involving \(\psi_{1}^{out,1}(\xi_{j}(t)+\epsilon y_{j},\tau)\) at \(y_{j}=0\). Since
    \begin{equation*}
        \frac{\nabla_{y}W_{j}}{W_{j}}=\frac{w'(r)}{w(r)}\begin{pmatrix}[1.3]
        \cos\theta\\
        \sin\theta
    \end{pmatrix}+i\frac{d_{j}}{r}\begin{pmatrix}[1.3]
        -\sin\theta\\
        \cos\theta
    \end{pmatrix}
    \end{equation*}where \(\theta\) denotes the polar angle of \(y_{j}\), it is straightforward to see that the first line of (\ref{firstnewerrtaylorexp}) contains only mode 1 terms in its Fourier expansion.

\smallskip

    We now define \(\xi^{1}(t)=\big(\xi_{1}^{1}(t),\ldots,\xi_{n}^{1}(t)\big)\) to be the unique solution of the system

    \begin{equation}
    \label{xi1ode}
        \begin{array}{ll}
\dot{\xi}_{j}^{1}(t)=d_{j}\nabla_{\xi_{j}}^{\perp}K\big(\xi^{1}(t)+\xi^{0}(t)\big)-d_{j}\nabla_{\xi_{j}}^{\perp}K\big(\xi^{0}(t)\big)\\[6pt]
        \hspace{6em}+2\nabla_{x}\psi_{1}^{out,1}(\xi_{j}^{0}(t),\tau;\xi^{0}), 
        \end{array}\quad\quad j=1,\ldots,n,
    \end{equation}with zero initial data
    \begin{equation}
    \label{xi1initial}
        \xi_{j}^{1}(0)=0,\quad\text{for}\quad j=1,\ldots,n.
    \end{equation}
    
    Here \(\psi_{1}^{out,1}(x,\tau;\xi^{0})\) denotes the unique solution of the wave equation (\ref{psiout1eq}) with parameters fixed at \(\xi(t)=\xi^{0}(t)\) in the definition of \(F^{out,1}(x,\tau)\). By the proof of Lemma \ref{psiout1construction} we have    
    \begin{equation*}
    \nabla_{x}\psi_{1}^{out,1}(x,\tau;\xi^{0})=O(\epsilon^{2}\abs{\log\epsilon}^{2}),
    \end{equation*}and a bootstrap argument for the ODE (\ref{xi1ode})-(\ref{xi1initial}) then implies
    \begin{equation*}
        \sup_{t\in[0,T]}\left(\abs{\xi^{1}(t)}+\abs{\dot{\xi}^{1}(t)}\right)\leq C\epsilon^{2}\abs{\log\epsilon}^{2}.
    \end{equation*}

    We now observe that the equation (\ref{xi1ode}) satisfied by \(\xi^{1}(t)\) and the equation (\ref{KirchhoffODE}) satisfied by \(\xi^{0}(t)\) imply that, for \(\xi(t)=\xi^{0}(t)+\xi^{1}(t)\), the expression on the left-hand side of (\ref{errfirstparameteradjust}) vanishes up to an error
    \begin{equation}
    \label{errfirstparameteradjustdifference}
        2\nabla_{x}\psi_{1}^{out,1}\big(\xi_{j}(t),\tau;\xi\big)-2\nabla_{x}\psi_{1}^{out,1}\big(\xi_{j}^{0}(t),\tau;\xi^{0}\big).
    \end{equation}
    Differentiating (\ref{psiout1eq}) with respect to \(\xi\) we find 
\begin{equation*}
\partial_{\xi}\nabla_{x}\psi_{1}^{out,1}(x,\tau;\xi)=O(\epsilon^{2}\abs{\log{\epsilon}^{2}}),
\end{equation*}
and it follows that (\ref{errfirstparameteradjustdifference}) has size \(O(\epsilon^{4}\abs{\log{\epsilon}}^{4})\). The proof of the lemma is complete.
\end{proof}

\begin{remark}
Going back to the proof of Lemma \ref{psiout1construction}, we can differentiate the equation satisfied by \(\psi_{1}^{out,1}(x,\tau;\xi^{0})\) in \(\tau\) to find a wave equation for \(\partial_{\tau}\psi_{1}^{out,1}(x,\tau;\xi^{0})\) with right-hand side 
\begin{equation}
\label{partialtauFxi0}
    \partial_{\tau}F^{out,1}(x,\tau;\xi^{0})
\end{equation}and initial data
\begin{equation}
    \partial_{\tau}\psi_{1}^{out,1}(x,0;\xi^{0})=0,\quad\partial_{\tau}(\partial_{\tau}\psi_{1}^{out,1})(x,0;\xi^{0})=-F^{out,1}(x,0;\xi^{0}).
\end{equation}Here \(F^{out,1}(x,\tau;\xi^{0})\) denotes the function (\ref{Fout1def}) with parameters fixed at \(\xi=\xi^{0}\).

\smallskip

Since the parameters \(\xi^{0}(t)\) are expressed in the original time variable \(t=\tfrac{1}{\sqrt{2}}\epsilon\tau\), each \(\tau\) derivative of \(F^{out,1}(x,\tau;\xi^{0})\) gains a power of \(\epsilon\). We then check that the function (\ref{partialtauFxi0}) has size
\begin{equation*}
    O\bigg(\frac{\epsilon^{3}}{1+\abs{x}^{2}}\bigg),
\end{equation*}and combining this fact with the estimates for the homogeneous wave equation stated in Remark \ref{homwaveeq} we get
\begin{equation}\label{1logpartialtaupsi1out}
\partial_{\tau}\nabla_{x}\psi_{1}^{out,1}(\xi_{j}^{0}(t),\tau;\xi^{0})=O(\epsilon^{2}\abs{\log\epsilon}),\quad\text{for}\quad j=1,\ldots,n.
\end{equation}

Now differentiating (\ref{xi1ode}) in \(t\) and using (\ref{1logpartialtaupsi1out}) and similar for \(\partial_{\tau\tau}\nabla_{x}\psi_{1}^{out,1}\), we find
\begin{equation}
\frac{d^{2}}{dt^{2}}\xi^{1}(t)=O\big(\epsilon\abs{\log\epsilon}\big),\quad\quad\frac{d^{3}}{dt^{3}}\xi^{1}(t)=O\big(\abs{\log\epsilon}\big).
\end{equation}It follows that the choice of parameters \(\xi(t)=\xi^{0}(t)+\xi^{1}(t)\) does not satisfy condition (\ref{xifirstimprov}) exactly; instead we have the slightly weaker control
\begin{equation}\label{xiweakcontrol}
\begin{split}
\sup_{t\in[0,T]}\bigg(&\abs{\xi(t)-\xi^{0}(t)}+\big|\frac{d}{dt}(\xi(t)-\xi^{0}(t))\big|+\epsilon\abs{\log\epsilon}\big|\frac{d^{2}}{dt^{2}}((\xi(t)-\xi^{0}(t))\big|\\
&\hspace{4em}+\epsilon^{2}\abs{\log\epsilon}\big|\frac{d^{3}}{dt^{3}}(\xi(t)-\xi^{0}(t))\big|\bigg)\leq C\epsilon^{2}\abs{\log\epsilon}^{2}.
\end{split}
\end{equation}

Returning to \(\S\)\ref{firstimprovementsubsect} one can check that replacing assumption (\ref{xifirstimprov}) with (\ref{xiweakcontrol}) does not affect the main results established there. In particular, the estimates for \(\phi_{j}^{(1)}\), \(\psi^{out,1}\) and \(S(u_{*})\) stated in Proposition \ref{firstimprovprop} remain true for \(\xi(t)\) satisfying (\ref{xiweakcontrol}).

\end{remark}

\subsection{Second improvement of the approximation}\label{secondimprovsubsect}

Fix an arbitrary integer \(m\geq4\). The next step towards the proof of Proposition \ref{arbitraryapproxprop} is to find functions \(\phi_{j}^{(2)}\) and \(\psi^{out,2}\) which yield a second improvement of the approximation. This is achieved in Proposition \ref{secondimprovementprop} below. Later \mbox{in \(\S\)\ref{inductionapproxsubsect}}, an iteration scheme based on the second improvement will suffice to approximate up \linebreak to \(O(\epsilon^{m})\) accuracy. 

\smallskip

Let \(M_{2}:=2^{2m}m\), and for \(\xi^{1}(t)\) defined via (\ref{xi1ode})-(\ref{xi1initial}), consider parameters of the form
\begin{equation}
\label{xiRem}
    \xi(t)=\xi^{0}(t)+\xi^{1}(t)+O(\epsilon^{3}\abs{\log\epsilon}^{c_{2}}),\quad\quad c_{2}:=M_{2}+5,\end{equation}with the \(\tau\) derivative control
\begin{equation}
\label{xiRemderiv}
    \sup_{t\in[0,T]}\Bigg(\sum_{\ell=1}^{M_{2}}\big|\frac{d^{\ell}}{d\tau^{\ell}}(\xi(t)-\xi^{0}(t)-\xi^{1}(t))\big|\Bigg)\leq C\epsilon^{4}\abs{\log\epsilon}^{c_{2}}.
\end{equation}The following result holds.

\begin{proposition}
\label{secondimprovementprop}
There exist functions \(\phi_{j}^{(2)}=\phi_{j}^{(2)}(y_{j},t)\) and \(\psi^{out,2}=\psi^{out,2}(x,\tau)\) satisfying
\begin{align*}
       \abs{\phi_{j}^{(2)}(y_{j},t)}+(1+\abs{y_{j}})\abs{\nabla_{y}\phi_{j}^{(2)}(y_{j},t)}&\leq C\epsilon^{3}\abs{\log\epsilon}^{c_{2}},\quad\text{ in}\quad B_{2\delta\epsilon^{-1}}(0)\times[0,T],\\
        \abs{\psi^{out,2}(x,\tau)}+\abs{\nabla_{x}\psi^{out,2}(x,\tau)}&\leq C\epsilon^{3}\abs{\log\epsilon}^{c_{2}},\quad\text{in}\quad\mathbb{R}^{2}\times[0,\sqrt{2}\epsilon^{-1}T],
    \end{align*}such that, for \(\xi(t)\) of the form (\ref{xiRem})-(\ref{xiRemderiv}) and \(\phi_{j}=\phi_{j}^{(1)}+\phi_{j}^{(2)}\), \(\psi^{out}=\psi^{out,1}+\psi^{out,2}\), we have the following error bound for the approximation (\ref{newapprox}):
   \begin{equation*}
        \sum_{\ell=0}^{M_{2}-3}\bigg(\int_{\mathbb{R}^{2}}\abs{D_{y}^{\ell}(S(u_{*}))}^{2}(y,t)\:dy\bigg)^{1/2}\leq C\epsilon^{4}\abs{\log\epsilon}^{c_{2}+\tfrac{1}{2}},\quad\text{for all}\quad t\in[0,T].    \end{equation*}

\end{proposition}

\begin{proof}
Returning to the expression (\ref{ustarglobalerrorexp}) for the error \(S(u_{*})\), let us write
\begin{equation}
\label{2ndimprovexp}    \psi_{j}=\psi_{j}^{(1)}+\psi_{j}^{(2)}\quad\text{and}\quad\psi^{out}=\psi^{out,1}+\psi^{out,2}
\end{equation}where \(\psi_{j}^{(1)}=(iW_{j})^{-1}\phi_{j}^{(1)}\) and \(\psi^{out,1}\) are the functions constructed in \(\S\)\ref{firstimprovementsubsect}, and \(\psi_{j}^{(2)}\) and \(\psi^{out,2}\) are functions to be determined. A comparison with (\ref{firstnewerrorglobal}) in \(\S\)\ref{firstimprovementsubsect} shows that, with the addition of the \(\psi_{j}^{(2)}\), \(\psi^{out,2}\) corrections, the new error can be expressed in the form
\begin{equation}\label{erorwithpsij2}
\begin{split}
   S(u_{*})=iu_{*}^{\eta}\Bigg(&\sum_{j=1}^{n}\tilde{\eta}_{j}\bigg(\mathcal{E}_{j}^{(2)}+\sqrt{2}\epsilon i\partial_{\tau}\psi_{j}^{(2)}+\widetilde{L}_{j}[\psi_{j}^{(2)}]+\widetilde{L}_{j}^{\#}[\psi_{j}^{(2)}]\bigg)\\
        &+\big(E_{1}^{out,2}+iE_{2}^{out,2}\big)+\widetilde{S}'(U_{\xi})[\psi^{out,2}]\\
        &+N\big(\psi^{out,1}+\psi^{out,2},\psi^{in,1}+\psi^{in,2}\big)-N\big(\psi^{out,1},\psi^{in,1}\big)\Bigg)
\end{split}    
\end{equation}where
\begin{equation}\label{mathEj2def}
\begin{split}
\mathcal{E}_{j}^{(2)}:=\:&\mathcal{E}_{j}^{(1,3)}+\bigg(-\epsilon^{3}\sqrt{2}\nabla_{x}\varphi_{\xi}\cdot\nabla_{x}\partial_{\tau}\psi_{1}^{out,1}+2\epsilon^{2}\frac{\nabla_{x}\abs{U_{\xi}}}{\abs{U_{\xi}}}\cdot\nabla_{x}\psi_{1}^{out,1}\bigg)\\
&+i\bigg(\frac{\epsilon^{3}}{\sqrt{2}}\Delta_{x}\partial_{\tau}\psi_{1}^{out,1}+\sqrt{2}\epsilon^{3}\frac{\nabla_{x}\abs{U_{\xi}}}{\abs{U_{\xi}}}\cdot\nabla_{x}\partial_{\tau}\psi_{1}^{out,1}+2\epsilon^{2}\nabla_{x}\varphi_{\xi}\cdot\nabla_{x}\psi_{1}^{out,1}\bigg)\\[4pt]
&+i(1-\abs{U_{\xi}}^{2})\sqrt{2}\epsilon\partial_{\tau}\psi_{1}^{out,1}+N(\psi^{out,1},\psi^{in,1})
\end{split}
\end{equation}
and
\begin{align}
\begin{split}
    E_{1}^{out,2}:=\:&\operatorname{Re}\mathcal{J}[\psi^{in,2}]+\bigg(1-\sum_{j=1}^{n}\tilde{\eta}_{j}\bigg)\operatorname{Re}\big(N(\psi^{out,1},\psi^{in,1})\big)\\
    &+\bigg(1-\sum_{j=1}^{n}\tilde{\eta}_{j}\bigg)\bigg(-\epsilon^{3}\sqrt{2}\nabla_{x}\varphi_{\xi}\cdot\nabla_{x}\partial_{\tau}\psi_{1}^{out,1}+2\epsilon^{2}\frac{\nabla_{x}\abs{U_{\xi}}}{\abs{U_{\xi}}}\cdot\nabla_{x}\psi_{1}^{out,1}\bigg), 
\end{split}\label{E1out2}\\    
\begin{split}
    E_{2}^{out,2}:=\:&\operatorname{Im}\mathcal{J}[\psi^{in,2}]+\bigg(1-\sum_{j=1}^{n}\tilde{\eta}_{j}\bigg)\operatorname{Im}\big(N(\psi^{out,1},\psi^{in,1})\big)\\
    &+\bigg(1-\sum_{j=1}^{n}\tilde{\eta}_{j}\bigg)\bigg(\frac{\epsilon^{3}}{\sqrt{2}}\Delta_{x}\partial_{\tau}\psi_{1}^{out,1}+\sqrt{2}\epsilon^{3}\frac{\nabla_{x}\abs{U_{\xi}}}{\abs{U_{\xi}}}\cdot\nabla_{x}\partial_{\tau}\psi_{1}^{out,1}\\
    &+2\epsilon^{2}\nabla_{x}\varphi_{\xi}\cdot\nabla_{x}\psi_{1}^{out,1}+(1-\abs{U_{\xi}}^{2})\sqrt{2}\epsilon\partial_{\tau}\psi_{1}^{out,1}\bigg)\\
    &+\frac{\epsilon^{2}}{2}\Delta_{x}E_{2}^{out,1}+\epsilon^{2}\frac{\nabla_{x}\abs{U_{\xi}}}{\abs{U_{\xi}}}\cdot\nabla_{x}E_{2}^{out,1}+(1-\abs{U_{\xi}}^{2})E_{2}^{out,1}.
\end{split}\label{E2out2}
\end{align}

Here \(\psi^{in,1}=(\psi_{1}^{(1)},\ldots,\psi_{n}^{(1)})\) and \(\psi^{in,2}=(\psi_{1}^{(2)},\ldots,\psi_{n}^{(2)})\), and the operator \(\mathcal{J}[\psi^{in}]\) was defined in (\ref{Jdef}). For the definition of \(N(\psi^{out},\psi^{in})\) we recall (\ref{nonlineartermsoutin}): note, in particular, that if \(\abs{x-\xi_{j}}\geq2\delta\) for all \(j\) we have the simple expression
\begin{equation*}
    N(\psi^{out},\psi^{in})=\widetilde{N}_{0}(\psi^{out})=i\epsilon^{2}(\nabla_{x}\psi^{out})^{2}+i\abs{U_{\xi}}^{2}\big(e^{-2\psi^{out}_{2}}-(1-2\psi^{out}_{2})\big).\end{equation*}  

\textbf{Step 1: Construction of \(\psi_{j}^{(2)}\)}. 

\smallskip

Fix \(j\in\{1,\ldots,n\}\) and write, as usual, \(r:=\abs{y_{j}}\). By Lemma \ref{firstnewerrexpcloselemma} the current inner error (\ref{mathEj2def}) can be expressed in the form
\begin{equation*}
    \mathcal{E}_{j}^{(2)}=\epsilon\frac{\nabla_{y}W_{j}}{W_{j}}\cdot\bigg(-\dot{\xi}_{j}+d_{j}\nabla_{\xi_{j}}^{\perp}K(\xi)+2\nabla_{x}\psi_{1}^{out,1}\big(\xi_{j},\tau\big)\bigg)+\widetilde{\mathcal{E}}_{j}^{(2)}(y_{j},t;\xi)\end{equation*}where
    \begin{equation*}
        \widetilde{\mathcal{E}}_{j}^{(2)}(y_{j},t;\xi)=i\mathcal{E}_{j}^{(2,0)}(r,t;\xi)+O\big(\epsilon^{4}\abs{\log\epsilon}^{2}r^{-1}\big)+iO\big(\epsilon^{4}\abs{\log\epsilon}^{2}(r^{-1}+1)\big)    \end{equation*}and \begin{equation*}
        \mathcal{E}_{j}^{(2,0)}(r,t;\xi)=O(\epsilon^{3}\abs{\log\epsilon}^{2}(1+r^{2})^{-1}).
        \end{equation*}

        The difference \(\mathcal{E}_{j}^{(2)}-\widetilde{\mathcal{E}}_{j}^{(2)}\) has been made small in \(\S\)\ref{firstparadjustsubsect} by adjusting the parameters. Now to eliminate \(\widetilde{\mathcal{E}}_{j}^{(2)}\) at main order, we make several elliptic improvements by solving recursively
        \begin{align}
        \widetilde{L}_{j}[\psi_{j}^{(2,1)}]&=-\widetilde{\mathcal{E}}_{j}^{(2)}(y_{j},t;\xi^{0}+\xi^{1}),\label{psi21eq}\\
            \widetilde{L}_{j}[\psi_{j}^{(2,k)}]&=-\sqrt{2}\epsilon i\partial_{\tau}\psi_{j}^{(2,k-1)}-\widetilde{L}_{j}^{\#}[\psi_{j}^{(2,k-1)}],\label{psi2keq}
        \end{align}for \(2\leq k\leq M_{2}+2\), where \(\widetilde{L}_{j}\) is the operator (\ref{Ljtildedef}) and \(\widetilde{L}_{j}^{\#}\) is the lower-order operator (\ref{tildeLjdef}). 
        
        In precise terms, the problems (\ref{psi21eq})-(\ref{psi2keq}) should be considered in \(\mathbb{R}^{2}\) with a smooth cut-off \(\tilde{\eta}_{j}^{(9/4)}\) of the form (\ref{etaj9quarter}) multiplying the right-hand sides. We note that the parameters have been fixed at \(\xi(t)=\xi^{0}(t)+\xi^{1}(t)\) on the right-hand side of (\ref{psi21eq}).

        \smallskip

        Decomposing (\ref{psi21eq})-(\ref{psi2keq}) in Fourier modes and using the linear theory in \(\S\)\ref{linprobsect}, we can find solutions \(\psi_{j}^{(2,k)}=\psi_{j1}^{(2,k)}+i\psi_{j2}^{(2,k)}\) of these equations such that, for \(1\leq k\leq M_{2}+2\),
        \begin{equation}\label{psij2kest}
            \sup_{r\leq2}\:\abs{\psi_{j}^{(2,k)}}+\sup_{r\geq2}\:\abs{r^{1-k}\psi_{j1}^{(2,k)}}+\sup_{r\geq2}\:\abs{r^{2-k}\psi_{j2}^{(2,k)}}\leq C_{k}\epsilon^{2+k}\abs{\log\epsilon}^{2+k}.
        \end{equation}The same estimate holds for the \(\tau\) derivatives of \(\psi_{j}^{(2,k)}\), and for the spatial \(y_{j}\) derivatives with extra decay in \(r\) (see Remark \ref{derivativeestimates}). Now setting
        \begin{equation}\label{psij2def}
            \psi_{j}^{(2)}(y_{j},t):=\sum_{k=1}^{M_{2}+2}\psi_{j}^{(2,k)}(y_{j},t)
        \end{equation}and
        \begin{equation*}
            \mathcal{E}_{j}^{(2,M_{2}+2)}(y_{j},t;\xi):=\widetilde{\mathcal{E}}_{j}^{(2)}(\cdot;\xi^{0}+\xi^{1})+\sqrt{2}\epsilon i\partial_{\tau}\psi_{j}^{(2)}+\widetilde{L}_{j}[\psi_{j}^{(2)}]+\widetilde{L}_{j}^{\#}[\psi_{j}^{(2)}],
        \end{equation*}we check that
        \begin{equation*}
        \mathcal{E}_{j}^{(2,M_{2}+2)}=\sqrt{2}\epsilon i\partial_{\tau}\psi_{j}^{(2,M_{2}+2)}+\widetilde{L}_{j}^{\#}[\psi_{j}^{(2,M_{2}+2)}]
        \end{equation*}and 
        \begin{align}
            \operatorname{Re}\big(\mathcal{E}_{j}^{(2,M_{2}+2)}\big)&=O\big(\epsilon^{M_{2}+5}\abs{\log\epsilon}^{M_{2}+4}(r^{M_{2}}+r^{-1})\big),\label{ej2M2plus2error1}\\
            \operatorname{Im}\big(\mathcal{E}_{j}^{(2,M_{2}+2)}\big)&=O\big(\epsilon^{M_{2}+5}\abs{\log\epsilon}^{M_{2}+4}(r^{M_{2}+1}+r^{-1})\big),\label{ej2M2plus2error2}
        \end{align}for \(r\leq 2\delta\epsilon^{-1}\). 

        The error sizes (\ref{ej2M2plus2error1})-(\ref{ej2M2plus2error2}) represent a substantial improvement on the size of \(\widetilde{\mathcal{E}}_{j}^{(2)}\) close to \(\xi_{j}(t)\) (i.e. at distances \(r=O(1)\)). On the other hand, if \(r=O(\epsilon^{-1})\) then (\ref{ej2M2plus2error1}) and (\ref{ej2M2plus2error2}) have the same size as \(\operatorname{Re}\big(\widetilde{\mathcal{E}}_{j}^{(2)}\big)\) and \(\operatorname{Im}\big(\widetilde{\mathcal{E}}_{j}^{(2)}\big)\) up to additional factors of \(\abs{\log\epsilon}\). This reveals an important difference compared to the first improvement, namely that we cannot improve in the whole region \(r\leq2\delta\epsilon^{-1}\) by simply solving elliptic equations. 

        To improve the new inner error for large \(r\) we will include part of \(\mathcal{E}_{j}^{(2,M_{2}+2)}\) on the right-hand side of a second wave equation, see (\ref{psi1out2eq}) below. Getting derivative estimates for that problem involves differentiating in \(x\) and hence losing powers of \(\epsilon\) close to the vortices: this is ultimately why we want so many powers of \(\epsilon\) in (\ref{ej2M2plus2error1})-(\ref{ej2M2plus2error2}).

\smallskip      

\textbf{Step 2: Construction of \(\psi^{out,2}\).}

\smallskip

Let \(\eta_{j}(y,t)\) be the smooth cut-off defined after (\ref{eta0def}), so that \(\eta_{j}=1\) if \(\abs{y_{j}}\leq1\) and \(\eta_{j}=0\) if \(\abs{y_{j}}\geq2\). Equivalently,
\begin{equation*}
    \eta_{j}=\begin{cases}
        1,\quad\text{if}\quad\abs{x-\xi_{j}}\leq\epsilon,\\
        0,\quad\text{if}\quad\abs{x-\xi_{j}}\geq2\epsilon.
    \end{cases}
\end{equation*}

For \(\mathcal{E}_{j}^{(2,M_{2}+2)}\) as above and the functions \(E_{1}^{out,2}\), \(E_{2}^{out,2}\) defined by (\ref{E1out2}), (\ref{E2out2}), we intend to get a global second improvement of the approximation by eliminating the expression
\begin{equation*}
    \sum_{j=1}^{n}(1-\eta_{j})\tilde{\eta}_{j}\mathcal{E}_{j}^{(2,M_{2}+2)}+\big(E_{1}^{out,2}+iE_{2}^{out,2}\big).
\end{equation*}

This will be achieved, in the spirit of \(\S\)\ref{firstimprovementsubsect}, by solving a linear wave equation for \(\psi_{1}^{out,2}=\operatorname{Re}(\psi^{out,2})\) and an algebraic equation for \(\psi_{2}^{out,2}=\operatorname{Im}(\psi^{out,2})\). Let us write 
\begin{equation*}
   \mathcal{E}_{j}^{(2,M_{2}+2)}=\mathcal{E}_{j1}^{(2,M_{2}+2)}+i\mathcal{E}_{j2}^{(2,M_{2}+2)}
\end{equation*}and define
\begin{equation}
\begin{split}
    F^{out,2}(x,\tau;\xi):=\frac{\epsilon^{-1}}{\sqrt{2}}\partial_{\tau}\bigg(&\sum_{j=1}^{n}(1-\eta_{j})\tilde{\eta}_{j}\mathcal{E}_{j2}^{(2,M_{2}+2)}+E_{2}^{out,2}\bigg)\\
    +\nabla_{x}\varphi_{\xi}\cdot\nabla_{x}\bigg(&\sum_{j=1}^{n}(1-\eta_{j})\tilde{\eta}_{j}\mathcal{E}_{j2}^{(2,M_{2}+2)}+E_{2}^{out,2}\bigg)\\
    -\epsilon^{-2}\bigg(&\sum_{j=1}^{n}(1-\eta_{j})\tilde{\eta}_{j}\mathcal{E}_{j1}^{(2,M_{2}+2)}+E_{1}^{out,2}\bigg).
\end{split}    
\end{equation}

We then let \(\psi_{1}^{out,2}(x,\tau)\) be the unique solution of the problem
\begin{equation}\label{psi1out2eq}
\begin{cases}
-\partial_{\tau\tau}\psi_{1}^{out,2}+\Delta_{x}\psi_{1}^{out,2}=F^{out,2}(x,\tau;\xi^{0}+\xi^{1}),&\quad\text{in}\quad\mathbb{R}^{2}\times(0,\sqrt{2}\epsilon^{-1}T),\\[4pt]
    \quad\quad\psi_{1}^{out,2}(x,0)=\partial_{\tau}\psi_{1}^{out,2}(x,0)=0,&\quad\text{in}\quad\mathbb{R}^{2},   
    \end{cases} 
\end{equation}and define
\begin{equation}\label{psi2out2def}
    \psi_{2}^{out,2}(x,\tau):=\frac{1}{2}\bigg(\sum_{j=1}^{n}(1-\eta_{j})\tilde{\eta}_{j}\mathcal{E}_{j2}^{(2,M_{2}+2)}+E_{2}^{out,2}+\epsilon\sqrt{2}\partial_{\tau}\psi_{1}^{out,2}\bigg).\end{equation}(In (\ref{psi2out2def}), all functions on the right-hand side should be considered with parameters fixed at \(\xi(t)=\xi^{0}(t)+\xi^{1}(t)\)).

    \smallskip

    Using (\ref{E1out2}), (\ref{E2out2}), (\ref{ej2M2plus2error1}), (\ref{ej2M2plus2error2}) and the estimates for \(\psi_{1}^{out,1}\) stated in Lemma \ref{psiout1construction} and Remark \ref{refinedestimatespsi1out}, we check that the function \(F^{out,2}\) can be decomposed as
    \begin{equation*}
        F^{out,2}=F_{a}^{out,2}+F_{b}^{out,2}+F_{c}^{out,2}+F_{d}^{out,2}
    \end{equation*}where
    \begin{gather*}
        F_{a}^{out,2}:=\chi\big(2\sqrt{2}\epsilon\nabla_{x}\varphi_{\xi}\cdot\nabla_{x}\partial_{\tau}\psi_{1}^{out,1}\big),\\[2pt]
        F_{b}^{out,2}:=\chi\frac{\epsilon^{2}}{2}\Delta_{x}\partial_{\tau\tau}\psi_{1}^{out,1},
    \end{gather*}and
    \begin{equation*}
        F_{c}^{out,2}=O_{c}(\epsilon^{3}\abs{\log\epsilon}^{M_{2}+4}),\quad\quad F_{d}^{out,2}=O(\epsilon^{3}\abs{\log\epsilon}^{4})\sum_{j=1}^{n}\frac{1}{1+\abs{x-\xi_{j}}^{2}}.
    \end{equation*}Here \(\chi(x,t):=1-\sum_{j}\tilde{\eta}_{j}\) as defined in \(\S\)\ref{firstimprovementsubsect}, and \(F_{c}^{out,2}\) denotes terms of size \(O(\epsilon^{3}\abs{\log\epsilon}^{M_{2}+4})\) that are compactly supported in the regions \(\abs{x-\xi_{j}}\leq 2\delta\) for \(j=1,\ldots,n\).  

    We observe as in the proof of Lemma \ref{psiout1construction} that 
    \begin{equation*}
     \nabla_{x}\varphi_{\xi}=\sum_{j=1}^{n}d_{j}\frac{(x-\xi_{j})^{\perp}}{\abs{x-\xi_{j}}^{2}}=\sum_{j=1}^{n}d_{j}\frac{x^{\perp}}{1+\abs{x}^{2}}+O(\abs{x}^{-2})\quad\text{as}\quad\abs{x}\to\infty,
     \end{equation*}hence, for large \(\abs{x}\), 
     \begin{equation*}
         \nabla_{x}\varphi_{\xi}\cdot\nabla_{x}\partial_{\tau}\psi_{1}^{out,1}\approx\bigg(\sum_{j=1}^{n}d_{j}\bigg)\frac{\:\partial_{\vartheta}\partial_{\tau}\psi_{1}^{out,1}}{1+{\abs{x}^{2}}}
     \end{equation*}where \(\partial_{\vartheta}=x^{\perp}\cdot\nabla_{x}\) denotes the angular derivative with respect to \(x=(\abs{x}\cos\vartheta,\abs{x}\sin\vartheta)\). By Remark \ref{angularderivpsi1out} we have \(\partial_{\vartheta}\psi_{1}^{out,1}=O(\epsilon^{2}\abs{\log\epsilon}^{2})\), and the same bound holds for \(\partial_{\vartheta}\partial_{\tau}\psi_{1}^{out,1}\). We can thus estimate
     \begin{equation*}
         F_{a}^{out,2}(x,\tau)=O(\epsilon^{3}\abs{\log\epsilon}^{2})\sum_{j=1}^{n}\frac{1}{1+\abs{x-\xi_{j}}^{2}}.
     \end{equation*}

     Now let \(\psi_{1,a}^{out,2}(x,\tau)\) be the unique solution (in \(\mathbb{R}^{2}\times(0,\sqrt{2}\epsilon^{-1}T)\)) of 
     \begin{equation}\label{psi1aout2waveeq}
         -\partial_{\tau\tau}\psi_{1,a}^{out,2}+\Delta_{x}\psi_{1,a}^{out,2}=(F_{a}^{out,2}+F_{c}^{out,2}+F_{d}^{out,2})(\:\cdot\:;\xi^{0}+\xi^{1})
     \end{equation}with zero initial data. Using Lemma \ref{Fquaddecaylemma}, Lemma \ref{Fcsupportlemma} and the bounds on \(F_{a}^{out,2}\), \(F_{c}^{out,2}\), \(F_{d}^{out,2}\) stated above, we deduce
     \begin{equation}\label{psi1aout2est}
         \abs{\psi_{1,a}^{out,2}(x,\tau)}\leq C\epsilon^{3}\abs{\log\epsilon}^{M_{2}+5}.
     \end{equation}

     To obtain an estimate for (\ref{psi1out2eq}), it remains to consider the solution \(\psi_{1,b}^{out,2}(x,\tau)\) of 
     \begin{equation}\label{psi1bout2ivp}
         \begin{cases}
-\partial_{\tau\tau}\psi_{1,b}^{out,2}+\Delta_{x}\psi_{1,b}^{out,2}=F_{b}^{out,2}(\:\cdot\:;\xi^{0}+\xi^{1}),&\quad\text{in}\quad\mathbb{R}^{2}\times(0,\sqrt{2}\epsilon^{-1}T),\\[4pt]
    \quad\quad\psi_{1,b}^{out,2}(x,0)=\partial_{\tau}\psi_{1,b}^{out,2}(x,0)=0,&\quad\text{in}\quad\mathbb{R}^{2}.   
    \end{cases} 
     \end{equation}
     The standard energy identity for (\ref{psi1bout2ivp}) combined with the Cauchy-Schwarz inequality (see (\ref{waveenergyidentity}) and (\ref{basicwaveenergyest}) in \(\S\)\ref{anenergyestsubsect}) gives
     \begin{equation}\label{psi1bout2energyest}
     \begin{split}
         \norm{\partial_{\tau}{\psi_{1,b}^{out,2}(x,\tau)}}_{L_{x}^{2}}+\norm{\nabla_{x}\psi_{1,b}^{out,2}(x,\tau)}_{L_{x}^{2}}\\
         \leq C\epsilon^{2}\int_{0}^{\tau}\norm{\Delta_{x}\partial_{\tau\tau}\psi_{1}^{out,1}(x,s)}_{L^{2}_{x}}\:ds
     \end{split}    
     \end{equation}for each \(\tau\geq0\), where
     \begin{equation*}
        \norm{\psi(x,\tau)}_{L^{2}_{x}}^{2}:=\int_{\mathbb{R}^{2}}\abs{\psi(x,\tau)}^{2}\:dx.
     \end{equation*}
    
    Now returning to the equation (\ref{psiout1eq}) satisfied by \(\psi_{1}^{out,1}\), we can differentiate in \(x\) and \(\tau\) to find a wave equation for \(\Delta_{x}\partial_{\tau}\psi_{1}^{out,1}\) with right-hand side
    \begin{equation*}
        \Delta_{x}\partial_{\tau}F^{out,1}=O\left(\frac{\epsilon^{3}\abs{\log\epsilon}^{2}}{1+\abs{x}^{3}}\right)
    \end{equation*}and initial data
    \begin{gather*}
        \Delta_{x}\partial_{\tau}\psi_{1}^{out,1}(x,0)=0,\\[3pt]
        \partial_{\tau}\big(\Delta_{x}\partial_{\tau}\psi_{1}^{out,1}\big)(x,0)=-\Delta_{x}F^{out,1}(x,0)=O\left(\frac{\epsilon^{2}\abs{\log\epsilon}}{1+\abs{x}^{3}}\right).
    \end{gather*}Standard energy estimates for this problem imply
    \begin{equation}
    \label{laplacepartialtautaupsiout1}
        \norm{\Delta_{x}\partial_{\tau\tau}\psi_{1}^{out,1}(x,s)}_{L^{2}_{x}}\leq C\epsilon^{2}\abs{\log\epsilon}^{2},\quad\text{for all}\quad s\in[0,\sqrt{2}\epsilon^{-1}T],
    \end{equation}and combining (\ref{psi1bout2energyest}) and (\ref{laplacepartialtautaupsiout1}), we get
    \begin{equation}\label{gradenergyestpsi1bout2}
        \norm{\partial_{\tau}{\psi_{1,b}^{out,2}(x,\tau)}}_{L_{x}^{2}}+\norm{\nabla_{x}\psi_{1,b}^{out,2}(x,\tau)}_{L_{x}^{2}}\leq C\epsilon^{3}\abs{\log\epsilon}^{2}.
    \end{equation}

    In summary, we have shown so far that the solution \(\psi_{1}^{out,2}\) of (\ref{psi1out2eq}) can be decomposed as \(\psi_{1}^{out,2}=\psi_{1,a}^{out,2}+\psi_{1,b}^{out,2}\) where \(\psi_{1,a}^{out,2}\) satisfies (\ref{psi1aout2est}) and \(\psi_{1,b}^{out,2}\) satisfies (\ref{gradenergyestpsi1bout2}). An \(L^{2}_{x}\)-bound for \(\psi_{1,b}^{out,2}\) itself can be obtained by employing energy estimates similar to (\ref{psi1bout2energyest}) to the wave equations with right-hand sides
    \begin{equation*}
        \chi\frac{\epsilon^{2}}{2}\partial_{x_{k}}\partial_{\tau\tau}\psi_{1}^{out,1},\quad\quad k=1,2,
    \end{equation*}and differentiating the solutions. We find  
    \begin{equation*}
    \norm{\psi_{1,b}^{out,2}(x,\tau)}_{L_{x}^{2}}\leq C\epsilon^{3}\abs{\log\epsilon}^{2}.
    \end{equation*}It then follows upon repeated differentiation that \(\psi_{1,b}^{out,2}\) and its derivatives up to any order satisfy \(O(\epsilon^{3}\abs{\log\epsilon}^{2})\) bounds in \(L^{2}_{x}\): by Sobolev embedding, we conclude
    \begin{equation}\label{psi1bout2derivest}
        \abs{\partial_{\tau}^{\ell_{1}}D_{x}^{\ell_{2}}\psi_{1,b}^{out,2}(x,\tau)}\leq C_{\ell_{1},\ell_{2}}\epsilon^{3}\abs{\log\epsilon}^{2}\quad\text{for all}\quad\ell_{1},\ell_{2}\geq0.
    \end{equation}

    Finally, we can estimate the derivatives of \(\psi_{1,a}^{out,2}\) by differentiating equation (\ref{psi1aout2waveeq}). The difference here is that the terms involving \(\mathcal{E}_{j1}^{(2,M_{2}+2)}\) and \(\mathcal{E}_{j2}^{(2,M_{2}+2)}\) lose a factor of \((\epsilon\abs{y_{j}})^{-1}\) with each derivative in \(x\), hence \(O(\epsilon^{3}\abs{\log\epsilon}^{M_{2}+5})\) bounds can only be obtained for derivatives up to order \(M_{2}\). In precise terms, we find
    \begin{equation}\label{psi1aout2derivest}
       \sum_{\substack{\ell_{1},\ell_{2}\geq 0\\ \ell_{1}+\ell_{2}\leq M_{2}}}\abs{\partial_{\tau}^{\ell_{1}}D_{x}^{\ell_{2}}\psi_{1,a}^{out,2}(x,\tau)}\leq C_{m}\epsilon^{3}\abs{\log\epsilon}^{M_{2}+5}.
    \end{equation}

\smallskip      

\textbf{Step 3: Completion of the proof.}

\smallskip

Let us set \(\phi_{j}^{(2)}:=iW_{j}\psi_{j}^{(2)}\) and \(\psi^{out,2}:=\psi_{1}^{out,2}+i\psi_{2}^{out,2}\) where \(\psi_{j}^{(2)}\) is defined by (\ref{psij2def}) and \(\psi_{1}^{out,2}\), \(\psi_{2}^{out,2}\) are defined by (\ref{psi1out2eq}), (\ref{psi2out2def}). It follows from (\ref{psij2kest}) and the corresponding gradient estimates that
\begin{equation*}
    \abs{\phi_{j}^{(2)}(y_{j},t)}+(1+\abs{y_{j}})\abs{\nabla_{y}\phi_{j}^{(2)}(y_{j},t)}\leq C\epsilon^{3}\abs{\log\epsilon}^{M_{2}+4},\quad\text{ in}\quad B_{2\delta\epsilon^{-1}}(0)\times[0,T].
\end{equation*}Moreover, by (\ref{psi1bout2derivest}), (\ref{psi1aout2derivest}) and (\ref{psi2out2def}) we have
\begin{align}
    \sum_{\ell_{1}+\ell_{2}\leq M_{2}}\abs{\partial_{\tau}^{\ell_{1}}D_{x}^{\ell_{2}}\psi_{1}^{out,2}(x,\tau)}&\leq C\epsilon^{3}\abs{\log\epsilon}^{M_{2}+5},\label{psi1out2fullderivest}\\
    \sum_{\ell_{1}+\ell_{2}\leq M_{2}-1}\abs{\partial_{\tau}^{\ell_{1}}D_{x}^{\ell_{2}}\psi_{2}^{out,2}(x,\tau)}&\leq C\epsilon^{4}\abs{\log\epsilon}^{M_{2}+5},\label{psi2out2est}
    \end{align}and thus
    \begin{equation*}
    \abs{\psi^{out,2}(x,\tau)}+\abs{\nabla_{x}\psi^{out,2}(x,\tau)}\leq C\epsilon^{3}\abs{\log\epsilon}^{M_{2}+5},\quad\text{in}\quad\mathbb{R}^{2}\times[0,\sqrt{2}\epsilon^{-1}T].
    \end{equation*}

    It remains to measure the size of the new error. Returning to (\ref{erorwithpsij2}) and using the equations satisfied by \(\psi_{j}^{(2)}\), \(\psi_{1}^{out,2}\) and \(\psi_{2}^{out,2}\), we find that   
    \begin{equation*}
    \begin{split}
        S(u_{*})=iu_{*}^{\eta}\Bigg(\tilde{\eta}_{j}\:\epsilon\frac{\nabla_{y}W_{j}}{W_{j}}\cdot\bigg(-\dot{\xi}_{j}+d_{j}\nabla_{\xi_{j}}^{\perp}K(\xi)+2\nabla_{x}\psi_{1}^{out,1}\big(\xi_{j},\tau\big)\bigg)
        \:+O\big(\epsilon^{4}\abs{\log\epsilon}^{M_{2}+5}r^{-1}\big)\Bigg)
    \end{split}
    \end{equation*}in the region \(r=\abs{y_{j}}\leq2\delta\epsilon^{-1}\), and 
    \begin{equation*}
        S(u_{*})=O(\epsilon^{5}\abs{\log\epsilon}^{M_{2}+5})\sum_{j=1}^{n}\frac{1}{1+\abs{x-\xi_{j}}},\quad\text{if}\quad\abs{x-\xi_{j}}\geq2\delta\quad\text{for all }j.
    \end{equation*}
    
    In the region \(\abs{x}\geq3\tau\), refined estimates for the solution of (\ref{psi1out2eq}) similar to those described in Remark \ref{refinedestimatespsi1out} show that \(\psi_{1}^{out,2}\) and its derivatives up to order \(M_{2}\) have size at most
    \begin{equation*}
        O\left(\frac{\epsilon^{3}\abs{\log\epsilon}^{M_{2}+5}\:\tau^{2}}{1+\abs{x}^{2}}\right)\quad\text{as}\quad\abs{x}\to\infty.
    \end{equation*}Using this fact we have
    \begin{equation*}
        \abs{S(u_{*})}\leq C\bigg(\frac{\epsilon^{5}\abs{\log\epsilon}^{M_{2}+5}\:\tau}{1+\abs{x}^{2}}\bigg),\quad\text{if}\quad\abs{x-\xi_{j}}\geq2\delta\quad\text{for all }j\quad\text{and}\quad\abs{x}\geq3\tau,   \end{equation*}then combining these estimates for \(S(u_{*})\) in the stated regions and recalling the assumptions (\ref{xiRem})-(\ref{xiRemderiv}) on the parameters we conclude
        \begin{equation*}
            \bigg(\int_{\mathbb{R}^{2}}\abs{S(u_{*})}^{2}(y,t)\:dy\bigg)^{1/2}\leq C\epsilon^{4}\abs{\log\epsilon}^{c_{2}+\tfrac{1}{2}},\quad\text{for all}\quad t\in[0,T],
            \end{equation*}where \(c_{2}=M_{2}+5\). The same estimate holds for the \(D_{y}\) derivatives of \(S(u_{*})\) up to order \(M_{2}-3\), since the current expression for \(S(u_{*})\) involves the derivatives of \(\psi_{1}^{out,2}\) up to order 3. The proof of the proposition is complete. 
\end{proof}

\begin{remark}
    Differentiating equation (\ref{psi1out2eq}) with respect to \(\vartheta\) (where \(x=(\abs{x}\cos\vartheta,\abs{x}\sin\vartheta)\)), it is straightforward to check (using Remark \ref{angularderivpsi1out}) that estimate (\ref{psi1out2fullderivest}) remains true with \(D_{x}\) replaced by \(\partial_{\vartheta}\).
    
    To improve the approximation further we also need an \(L^{2}_{x}\)-estimate for the derivatives \linebreak of \(\psi_{1}^{out,2}\). Returning to equation (\ref{psi1aout2waveeq}), one can use the \(O(\abs{x}^{-2})\) space decay on the right-hand side and an energy estimate proved in \(\S\)\ref{anenergyestsubsect} (see Lemma \ref{anenergyestlemma}) to deduce
    \begin{equation}\label{psi1aout2energyest}
        \sum_{1\leq\ell_{1}+\ell_{2}\leq M_{2}}\norm{\partial_{\tau}^{\ell_{1}}D_{x}^{\ell_{2}}\psi_{1,a}^{out,2}}_{L_{x}^{2}}\leq C\epsilon^{5/2}\abs{\log\epsilon}^{c_{2}}.
    \end{equation}Combining (\ref{psi1aout2energyest}) with the \(L^{2}_{x}\)-bounds for \(\psi_{1,b}^{out,2}\) stated before (\ref{psi1bout2derivest}), we have 
    \begin{equation}\label{psi1out2l2est}
        \sum_{1\leq\ell_{1}+\ell_{2}\leq M_{2}}\norm{\partial_{\tau}^{\ell_{1}}D_{x}^{\ell_{2}}\psi_{1}^{out,2}}_{L_{x}^{2}}\leq C\epsilon^{5/2}\abs{\log\epsilon}^{c_{2}}.
    \end{equation}
    \end{remark}

\subsection{An induction argument}\label{inductionapproxsubsect}
We now prove the main result of this section.

\begin{proof}[Proof of Proposition \ref{arbitraryapproxprop}] 

It remains to construct the parameter corrections \(\xi^{k}(t)\) for \(k\geq2\) and the inner-outer corrections \(\phi_{j}^{(k)}(y_{j},t)\) and \(\psi^{out,k}(x,\tau)\) for \(k\geq3\). We proceed by induction, in which the essential ideas at each step are a natural generalization of those already described in the second improvement of the approximation. Let us outline the main details.

\smallskip

In the \((k+1)\)th improvement of the approximation, we assume that \(\xi^{k-1}\), \(\phi_{j}^{(k)}\), \(\psi^{out,k}\) and their predecessors have been constructed and set
\begin{align*}
    \xi(t)&=\xi^{0}(t)+\ldots+\xi^{k-1}(t)+\xi^{k}(t),\\[2pt]
    \phi_{j}(y_{j},t)&=\phi_{j}^{(1)}(y_{j},t)+\ldots+\phi_{j}^{(k)}(y_{j},t)+\phi_{j}^{(k+1)}(y_{j},t),\\[2pt]
    \psi^{out}(x,\tau)&=\psi^{out,1}(x,\tau)+\ldots+\psi^{out,k}(x,\tau)+\psi^{out,k+1}(x,\tau),
\end{align*}in the ansatz (\ref{newapprox}) for \(u_{*}\). The functions \(\xi^{k}(t)\) and \(\phi_{j}^{(k+1)}=iW_{j}\psi_{j}^{(k+1)}\) are chosen, \linebreak for \(j=1,\ldots,n\), to eliminate inner error terms of the form  
\begin{equation*}
    \mathcal{E}_{j}^{(k+1)}(y_{j},t):=\epsilon\frac{\nabla_{y}W_{j}}{W_{j}}\cdot\mathcal{P}_{j}^{(k)}(\xi,t)+\widetilde{\mathcal{E}}_{j}^{(k+1)}(y_{j},t),\quad\text{for}\quad\abs{y_{j}}\leq2\delta\epsilon^{-1},
\end{equation*}where
\begin{equation}
\label{mathPjkdef}
    \mathcal{P}_{j}^{(k)}(\xi,t):=-\dot{\xi}_{j}+d_{j}\nabla_{\xi_{j}}^{\perp}K(\xi)+2\nabla_{x}\psi_{1}^{out,1}(\xi_{j},\tau)+\ldots+2\nabla_{x}\psi_{1}^{out,k}(\xi_{j},\tau)
\end{equation}and
\begin{gather*}
    \big|\operatorname{Re}\widetilde{\mathcal{E}}_{j}^{(k+1)}\big|\leq C\epsilon r^{-1}(1+r)^{-1}\sum_{1\leq\ell_{1}+\ell_{2}\leq3}\norm{\partial_{\tau}^{\ell_{1}}D_{x}^{\ell_{2}}\psi_{1}^{out,k}}_{\infty},\\
    \big|\operatorname{Im}\widetilde{\mathcal{E}}_{j}^{(k+1)}\big|\leq C\epsilon r^{-1}\sum_{1\leq\ell_{1}+\ell_{2}\leq3}\norm{\partial_{\tau}^{\ell_{1}}D_{x}^{\ell_{2}}\psi_{1}^{out,k}}_{\infty}.\end{gather*}(Here \(r=\abs{y_{j}}\), \(\psi_{1}^{out,k}=\operatorname{Re}(\psi^{out,k})\) and \(\norm{\cdot}_{\infty}\) denotes the supremum norm on \(\mathbb{R}^{2}\times[0,\sqrt{2}\epsilon^{-1}T]\)). 
    
    In more precise terms, the functions 
    \begin{equation*}
        \xi^{k}(t)=\big(\xi_{1}^k(t),\ldots,\xi_{n}^{k}(t)\big)
    \end{equation*}are defined as the solution of an ODE involving \(\psi_{1}^{out,1},\ldots,\psi_{1}^{out,k}\) (with zero initial data) which eliminates expression (\ref{mathPjkdef}) at main order for \(j=1,\ldots,n\). We obtain the estimate
    \begin{equation}\label{xikparameterest}
        \sup_{t\in[0,T]}\bigg(\abs{\xi^{k}(t)}+\abs{\dot{\xi}^{k}(t)}\bigg)\leq C\norm{\nabla_{x}\psi_{1}^{out,k}}_{\infty}.
    \end{equation}

    Concerning the elimination of \(\widetilde{\mathcal{E}}_{j}^{(k+1)}\), we fix a large number \(M_{k+1}\) to be chosen later and make \(M_{k+1}+2\) elliptic improvements (inverting the operator \(\widetilde{L}_{j}\) using the linear theory in \(\S\)\ref{linprobsect}) to create a new error
    \begin{equation*}
    \mathcal{E}_{j}^{(k+1,M_{k+1}+2)}:=\widetilde{\mathcal{E}}_{j}^{(k+1)}+\sqrt{2}\epsilon i\partial_{\tau}\psi_{j}^{(k+1)}+\widetilde{L}_{j}[\psi_{j}^{(k+1)}]+\widetilde{L}_{j}^{\#}[\psi_{j}^{(k+1)}]
    \end{equation*}of size
    \begin{align*}
            \big|\operatorname{Re}\mathcal{E}_{j}^{(k+1,M_{k+1}+2)}\big|&\leq C\epsilon^{M_{k+1}+3}\abs{\log\epsilon}^{M_{k+1}+2}r^{M_{k+1}}\sum_{1\leq\ell_{1}+\ell_{2}\leq M_{k+1}+5}\norm{\partial_{\tau}^{\ell_{1}}D_{x}^{\ell_{2}}\psi_{1}^{out,k}}_{\infty},\\
            \big|\operatorname{Im}\mathcal{E}_{j}^{(k+1,M_{k+1}+2)}\big|&\leq C\epsilon^{M_{k+1}+3}\abs{\log\epsilon}^{M_{k+1}+2}r^{M_{k+1}+1}\sum_{1\leq\ell_{1}+\ell_{2}\leq M_{k+1}+5}\norm{\partial_{\tau}^{\ell_{1}}D_{x}^{\ell_{2}}\psi_{1}^{out,k}}_{\infty},
        \end{align*}for \(1\leq r\leq 2\delta\epsilon^{-1}\) (if \(r\leq1\) the same estimates hold with the powers of \(r\) replaced by \(r^{-1}\)). The function \(\phi_{j}^{(k+1)}=iW_{j}\psi_{j}^{(k+1)}\) built by solving the elliptic problems in this process satisfies
        \begin{equation}\label{phijkplus1innerest}
        \begin{split}
            \sup_{B_{2\delta\epsilon^{-1}}(0)\times[0,T]}\bigg(\abs{\phi_{j}^{(k+1)}(y_{j},t)}+\big(1+\abs{y_{j}}\big)\abs{\nabla_{y}\phi_{j}^{(k+1)}(y_{j},t)}\bigg)\\
            \leq C\epsilon\abs{\log\epsilon}^{M_{k+1}+2}\sum_{1\leq\ell_{1}+\ell_{2}\leq M_{k+1}+5}\norm{\partial_{\tau}^{\ell_{1}}D_{x}^{\ell_{2}}\psi_{1}^{out,k}}_{\infty}.
            \end{split}    
        \end{equation}

        Now let us consider the improvement of the outer error. The function \(\psi_{1}^{out,k+1}=\operatorname{Re}(\psi^{out,k+1})\) is constructed by solving a wave equation of the form 
        \begin{equation}\label{psi1outkplusoneeq}
        \begin{cases}
        -\partial_{\tau\tau}\psi_{1}^{out,k+1}+\Delta_{x}\psi_{1}^{out,k+1}=F^{out,k+1}(x,\tau),&\quad\text{in}\quad\mathbb{R}^{2}\times(0,\sqrt{2}\epsilon^{-1}T),\\[4pt]
        \quad\quad\psi_{1}^{out,k+1}(x,0)=\partial_{\tau}\psi_{1}^{out,k+1}(x,0)=0,&\quad\text{in}\quad\mathbb{R}^{2},   
        \end{cases}    
    \end{equation}where
    \begin{equation*}
        F^{out,k+1}=F_{a}^{out,k+1}+F_{b}^{out,k+1}+F_{c}^{out,k+1}+\text{lower order terms}
    \end{equation*}with
    \begin{gather*}
        F_{a}^{out,k+1}:=\chi\big(2\sqrt{2}\epsilon\nabla_{x}\varphi_{\xi}\cdot\nabla_{x}\partial_{\tau}\psi_{1}^{out,k}\big),\\[2pt]
        F_{b}^{out,k+1}:=\chi\frac{\epsilon^{2}}{2}\Delta_{x}\partial_{\tau\tau}\psi_{1}^{out,k},
    \end{gather*}and
\begin{equation*}
\begin{split}
    F_{c}^{out,k+1}:=\frac{\epsilon^{-1}}{\sqrt{2}}\partial_{\tau}\bigg(&\sum_{j=1}^{n}(1-\eta_{j})\tilde{\eta}_{j}\mathcal{E}_{j2}^{(k+1,M_{k+1}+2)}\bigg)\\
    +\nabla_{x}\varphi_{\xi}\cdot\nabla_{x}\bigg(&\sum_{j=1}^{n}(1-\eta_{j})\tilde{\eta}_{j}\mathcal{E}_{j2}^{(k+1,M_{k+1}+2)}\bigg)\\
    -\epsilon^{-2}\bigg(&\sum_{j=1}^{n}(1-\eta_{j})\tilde{\eta}_{j}\mathcal{E}_{j1}^{(k+1,M_{k+1}+2)}\bigg).
\end{split}    
\end{equation*}
Here \(\chi(x,t)=1-\sum_{j}\tilde{\eta}_{j}\) and 
\begin{equation*}
\mathcal{E}_{j1}^{(k+1,M_{k+1}+2)}:=\operatorname{Re}(\mathcal{E}_{j}^{(k+1,M_{k+1}+2)}),\quad\quad\mathcal{E}_{j2}^{(k+1,M_{k+1}+2)}:=\operatorname{Im}(\mathcal{E}_{j}^{(k+1,M_{k+1}+2)}).\end{equation*}

In the imaginary part, \(\psi_{2}^{out,k+1}=\operatorname{Im}(\psi^{out,k+1})\) is determined by an expression of the form
\begin{equation}\label{psi2outkplus1def}
    \psi_{2}^{out,k+1}(x,\tau):=\frac{1}{2}\bigg(\sum_{j=1}^{n}(1-\eta_{j})\tilde{\eta}_{j}\mathcal{E}_{j2}^{(k+1,M_{k+1}+2)}+E_{2}^{out,k}+\epsilon\sqrt{2}\partial_{\tau}\psi_{1}^{out,k+1}\bigg)\end{equation}where
\begin{equation*}
    \abs{E_{2}^{out,k}(x,\tau)}\leq C\epsilon^{2}\sum_{1\leq\ell_{1}+\ell_{2}\leq3}\norm{\partial_{\tau}^{\ell_{1}}D_{x}^{\ell_{2}}\psi_{1}^{out,k}}_{\infty}.
\end{equation*}

In \(\S\)\ref{anenergyestsubsect} we prove an elementary lemma which provides an \(O(\tau^{1/2}\log\tau)\) growth rate for the \(L^{2}_{x}\)-norm of the derivatives of the solution of an inhomogeneous wave equation with sufficient space decay on the right-hand side (see Lemma \ref{anenergyestlemma}). Using this fact, estimate (\ref{basicwaveenergyest}) and Sobolev embedding, one can establish that the solution \(\psi_{1}^{out,k+1}\) of (\ref{psi1outkplusoneeq}) satisfies
\begin{equation}\label{psioutkplus1intermsofprevious}\begin{split}
    &\sum_{1\leq\ell_{1}+\ell_{2}\leq M_{k+1}}\norm{\partial_{\tau}^{\ell_{1}}D_{x}^{\ell_{2}}\psi_{1}^{out,k+1}}_{L^{2}_{x}}\\
    &\hspace{8em}\leq C\epsilon^{1/2}\abs{\log\epsilon}^{M_{k+1}+3}\sum_{1\leq\ell_{1}+\ell_{2}\leq \:2M_{k+1}+8}\norm{\partial_{\tau}^{\ell_{1}}D_{x}^{\ell_{2}}\psi_{1}^{out,k}}_{L^{2}_{x}}.
\end{split}    
\end{equation}

At this point we note that estimate (\ref{psioutkplus1intermsofprevious}) is only of use if the number \(2M_{k+1}+8\) is no greater than the number \(M_{k}\) of derivatives of \(\psi_{1}^{out,k}\) controlled in the previous step. We thus define recursively
\begin{equation}\label{Mkplus1def}
    M_{k+1}:=\left \lfloor{\frac{M_{k}-8}{2}}\right \rfloor,\quad\text{for}\quad k\geq2,
\end{equation}where the right-hand side of (\ref{Mkplus1def}) denotes the largest integer less than or equal to \((M_{k}-8)/2\). For \(k=2\), we recall \begin{equation}\label{M2def}
    M_{2}:=2^{2m}m.
\end{equation}

Returning to the estimates (\ref{psi1out2fullderivest}) and (\ref{psi1out2l2est}) for \(\psi_{1}^{out,2}\) proved in \(\S\)\ref{secondimprovsubsect}, one can use Lemma \ref{anenergyestlemma} and (\ref{basicwaveenergyest}) to show (independently of (\ref{psioutkplus1intermsofprevious})) that
\begin{equation*}
    \sum_{1\leq\ell_{1}+\ell_{2}\leq M_{3}}\norm{\partial_{\tau}^{\ell_{1}}D_{x}^{\ell_{2}}\psi_{1}^{out,3}}_{L^{2}_{x}}\leq C\epsilon^{7/2}\abs{\log\epsilon}^{c_{3}}
\end{equation*}for some \(c_{3}>0\). By applying (\ref{psioutkplus1intermsofprevious}) repeatedly, we obtain by induction
\begin{equation}
\label{psi1outkfinalenergyest}
    \sum_{1\leq\ell_{1}+\ell_{2}\leq M_{k}}\norm{\partial_{\tau}^{\ell_{1}}D_{x}^{\ell_{2}}\psi_{1}^{out,k}}_{L^{2}_{x}}\leq C\epsilon^{2+\frac{k}{2}}\abs{\log\epsilon}^{c_{k}},\quad\text{for all}\quad k\geq3.
\end{equation}

Pointwise estimates for \(\psi_{1}^{out,k}\) itself can be obtained from (\ref{psi1outkplusoneeq}) (with \(k\) replaced by \(k-1\)) using Lemma \ref{Fquaddecaylemma}, energy estimates and Sobolev embedding in the same spirit as \(\S\)\ref{secondimprovsubsect}. We find
\begin{equation}
\label{psi1outkfinalpointest}
\norm{\psi_{1}^{out,k}}_{\infty}+\norm{\nabla_{x}\psi_{1}^{out,k}}_{\infty} \leq C\epsilon^{2+\frac{k}{2}}\abs{\log\epsilon}^{c_{k}},\quad\text{for all}\quad k\geq3. 
\end{equation}Upon combining (\ref{psi1outkfinalenergyest}), (\ref{psi1outkfinalpointest}) with (\ref{xikparameterest}), (\ref{phijkplus1innerest}) and (\ref{psi2outkplus1def}), we also have
\begin{gather*}
    \sup_{t\in[0,T]}\bigg(\abs{\xi^{k}(t)}+\abs{\dot{\xi}^{k}(t)}\bigg)\leq C\epsilon^{2+\frac{k}{2}}\abs{\log\epsilon}^{c_{k}},\\
    \sup_{B_{2\delta\epsilon^{-1}}(0)\times[0,T]}\bigg(\abs{\phi_{j}^{(k)}(y_{j},t)}+\big(1+\abs{y_{j}}\big)\abs{\nabla_{y}\phi_{j}^{(k)}(y_{j},t)}\bigg)\leq C\epsilon^{2+\frac{k+1}{2}}\abs{\log\epsilon}^{c_{k}},\\
\norm{\psi_{2}^{out,k}}_{\infty}+\norm{\nabla_{x}\psi_{2}^{out,k}}_{\infty}\leq C\epsilon^{3+\frac{k}{2}}\abs{\log\epsilon}^{c_{k}}.
            \end{gather*}

After the \((2m-5)\)th improvement of the approximation, we set \(\xi^{2m-5}=0\) and check that for \(\xi(t)\), \(\phi_{j}\) and \(\psi^{out}\) defined as in the statement of Proposition \ref{arbitraryapproxprop}, the largest term in the new error \(S(u_{*})\) is given by 
\begin{equation*}
    iu_{*}^{\eta}\big(2i\epsilon^{2}\nabla_{x}\varphi_{\xi}\cdot\nabla_{x}\psi_{1}^{out,2m-5}\big).
\end{equation*}Using (\ref{psi1outkfinalenergyest}) and Sobolev embedding, we can then verify that
\begin{equation}\label{l2epsilonmest}
\begin{split}
     \bigg(\int_{\mathbb{R}^{2}}\abs{S(u_{*})}^{2}(y,t)\:dy\bigg)^{1/2}&\leq C\epsilon\cdot\epsilon^{2+\frac{2m-5}{2}}\abs{\log\epsilon}^{c_{2m-5}}\\
     &= C\epsilon^{m+\frac{1}{2}}\abs{\log\epsilon}^{c_{2m-5}}\:\leq \:C\epsilon^{m}.
\end{split}     
\end{equation}

Finally, we note that the error \(S(u_{*})\) involves up to three derivatives of \(\psi_{1}^{out,2m-5}\) and two derivatives are lost to estimate in \(L^{\infty}\) with Sobolev embedding. Thus estimate (\ref{l2epsilonmest}) holds for the \(D_{y}\) derivatives of \(S(u_{*})\) up to order \(M_{2m-5}-5\). Using (\ref{Mkplus1def})-(\ref{M2def}), we check that
\begin{equation*}
        M_{k}\:\geq\:2^{2m-(k-2)}m-9\bigg(\frac{1}{2}+\ldots+\frac{1}{2^{k-2}}\bigg)\quad\text{for}\quad k\geq2,
\end{equation*}hence
\begin{equation*}
    M_{2m-5}-5\:\:\geq\:\:2^{7}m-9-5\:\:\geq\:\: m.
\end{equation*}

It follows that (\ref{l2epsilonmest}) is satisfied by the \(D_{y}\) derivatives of \(S(u_{*})\) up to order \(m\). The proof of the proposition is complete.
\end{proof}

\begin{remark}\label{zeroesofustar}
    Building upon Remark \ref{zerosetfirstimprov}, the inner corrections \(\phi_{j}^{(k)}=iW_{j}\psi_{j}^{(k)}\) constructed by induction have \(\psi_{j}^{(k)}(y_{j},t)\) bounded at the origin for all \(k\geq1\). It follows that for \(\xi(t)\), \(\phi_{j}\) and \(\psi^{out}\) as in Proposition \ref{arbitraryapproxprop}, the zeroes of \(u_{*}(y,t)\) are located precisely at the points \(\tilde{\xi}_{j}(t)=\epsilon^{-1}\xi_{j}(t)\) for \(j=1,\ldots,n\). 
\end{remark}

\section{Estimates for the linear wave equation}
\label{waveestimatesect}

In this section we prove some estimates which were used in Section \ref{firstimprovementsubsect} and Section \ref{arbitraryapproxsect}. We consider the linear wave equation 
\begin{equation}
\label{linearwaveeqdef}
 \begin{cases}
-\partial_{\tau\tau}\psi+\Delta_{x}\psi=F(x,\tau),&\quad\text{in}\quad\mathbb{R}^{2}\times(0,\infty),\\[4pt]
    \psi(x,0)=\partial_{\tau}\psi(x,0)=0,&\quad\text{in}\quad\mathbb{R}^{2}.   
    \end{cases} 
\end{equation}

\subsection{Pointwise estimates}\label{wavepointwiseestimates}
Given a smooth real-valued function \(F(x,\tau)\), problem (\ref{linearwaveeqdef}) has a unique smooth real-valued solution \(\psi(x,\tau)\) given by the Duhamel formula
\begin{equation}
\label{duhamelformula}
    \psi(x,\tau)=\frac{1}{2\pi}\int_{0}^{\tau}\int_{B(x,s)}\frac{F(x',\tau-s)}{\sqrt{s^2-\abs{x'-x}^{2}}}\:dx'\:ds.
\end{equation}

We are interested in the large \(\tau\) behaviour of (\ref{duhamelformula}) for right-hand sides \(F(x,\tau)\) with a decay rate in \(\abs{x}\). First assume
\begin{equation*}
    \abs{F(x,\tau)}\leq\frac{C}{1+\abs{x}^{2}}
\end{equation*}for some constant \(C>0\) independent of \(x\) and \(\tau\). Then we have the following result.

\begin{lemma}
\label{Fquaddecaylemma}
    There exists a constant \(C>0\) such that for any function \(F(x,\tau)\) satisfying 
    \begin{equation*}
    \norm{F}_{\infty,2}:=\sup_{\mathbb{R}^{2}\times(0,\infty)}\abs{F(x,\tau)}(1+\abs{x}^{2})<\infty,
    \end{equation*}
    problem (\ref{linearwaveeqdef}) has a unique solution \(\psi(x,\tau)\) such that
    \begin{equation}
    \label{estFquaddecay}
        \abs{\psi(x,\tau)}\leq C\big(\log(1+\tau)\big)^{2}\norm{F}_{\infty,2}.
    \end{equation}
\end{lemma}

\begin{proof}
    The existence and uniqueness of solutions to (\ref{linearwaveeqdef}) follows directly from the Duhamel formula (\ref{duhamelformula}). Thus it remains to prove estimate (\ref{estFquaddecay}). 

    Taking the absolute value of (\ref{duhamelformula}) we get the inequality
    \begin{equation*}
        \abs{\psi(x,\tau)}\leq I\big(\abs{x},\tau\big)\norm{F}_{\infty,2}
    \end{equation*}where
    \begin{equation*}
        I(\abs{x},\tau):=\frac{1}{2\pi}\int_{0}^{\tau}\int_{B(0,s)}\frac{1}{\sqrt{s^{2}-\abs{x'}^{2}}}\frac{1}{\big(1+\abs{x'+x}^{2}\big)}\:dx'\:ds.
    \end{equation*}
    
    The expression above depends only on \(\tau\) and the norm of \(x\) since solutions to (\ref{linearwaveeqdef}) are radially symmetric if \(F(x,\tau)\) is radially symmetric. We may then assume \(x=(\abs{x},0)\), and writing the inner integral for \(I(\abs{x},\tau)\) in polar coordinates
    \begin{equation*}
        x':=(r'\cos\vartheta',r'\sin\vartheta')
    \end{equation*}we get
    \begin{equation*}
        \begin{split}
            I(\abs{x},\tau)&=\frac{1}{2\pi}\int_{0}^{\tau}\int_{0}^{s}\int_{0}^{2\pi}\frac{1}{\sqrt{s^{2}-(r')^{2}}}\frac{r'}{(1+\abs{x}^{2}+2r'\abs{x}\cos\vartheta'+(r')^{2})}\:d\vartheta'\:dr'\:ds\\[4pt]
            &=\int_{0}^{\tau}\int_{0}^{s}\frac{r'}{\sqrt{s^{2}-(r')^{2}}\sqrt{1+(\abs{x}+r')^{2}}\sqrt{1+(\abs{x}-r')^{2}}}\:dr'\:ds.
        \end{split}
    \end{equation*}

    Let us write \(I^{hom}(\abs{x},s)\) for the integral from \(0\) to \(s\) on the second line of the above. By estimating this integral separately in the regions 
    \begin{equation*}
        \abs{x}\leq\tfrac{1}{2}s,\quad\tfrac{1}{2}s\leq\abs{x}\leq2s,\quad\text{and}\quad\abs{x}\geq2s,
    \end{equation*}it can be verified that
    \begin{equation}
    \label{Ihomest}
        I^{hom}(\abs{x},s)\leq\left\{\begin{array}{cc}
        \frac{C\log(1+s)}{\sqrt{1+s}\sqrt{1+\abs{s-\abs{x}}}},&\text{ if}\quad\abs{x}\leq2s,\\[10pt]
        \frac{Cs}{1+\abs{x}^{2}},&\text{ if}\quad\abs{x}\geq2s.        \end{array}\right.
    \end{equation}
    We then conclude the proof by noting
    \begin{equation*}
        I(\abs{x},\tau)=\int_{0}^{\tau}I^{hom}(\abs{x},s)\:ds\leq\left\{\begin{array}{cc}
        C\big(\log(1+\tau)\big)^{2},&\text{ if}\quad\abs{x}\leq2\tau,\\[5pt]
        \frac{C\tau^{2}}{1+\abs{x}^{2}},&\text{ if}\quad\abs{x}\geq2\tau.        \end{array}\right.
        \end{equation*}(This estimate for \(I(\abs{x},\tau)\) actually provides a stronger result than (\ref{estFquaddecay}), as it shows the solution decays like \(O(\tau^{2}\abs{x}^{-2})\) in the region \(\abs{x}\geq2\tau\) outside the light cone).    \end{proof}

    \begin{remark}\label{derivativesFquaddecay}
    After changing variables \(s\mapsto\tau-s\) and \(x'\mapsto\frac{x'-x}{\tau-s}\) in (\ref{duhamelformula}) and differentiating in \(\tau\), we get the explicit expression
\begin{equation}
\label{dtaupsi}
\begin{split}
    \partial_{\tau}\psi(x,\tau)&=\frac{1}{2\pi}\int_{0}^{\tau}\frac{1}{s}\int_{B(x,s)}\frac{F(x',\tau-s)}{\sqrt{s^{2}-\abs{x'-x}^{2}}}\:dx'\:ds\\
    &\phantom{=}+\frac{1}{2\pi}\int_{0}^{\tau}\frac{1}{s}\int_{B(x,s)}\frac{\nabla_{x'} F(x',\tau-s)\cdot(x'-x)}{\sqrt{s^{2}-\abs{x'-x}^{2}}}\:dx'\:ds
\end{split}    
\end{equation}for the first \(\tau\) derivative of \(\psi(x,\tau)\). The reasoning in the previous proof can then be extended to deduce the estimate
\begin{equation}
\label{estFquaddecaywithtime}    \abs{\psi(x,\tau)}+\abs{\partial_{\tau}\psi(x,\tau)}\leq C\big(\log(1+\tau)\big)^{2}\bigg(\norm{F}_{\infty,2}+\norm{\nabla_{x}F}_{\infty,2}\bigg)
\end{equation}for solutions of (\ref{linearwaveeqdef}), provided the norm on the right-hand side of (\ref{estFquaddecaywithtime}) is finite. 

By differentiating (\ref{linearwaveeqdef}) repeatedly with respect to the spatial and \(\tau\) variables, one can in fact show that
\begin{equation*}
    \sum_{\ell_{1}+\ell_{2}\leq m}\abs{\partial_{\tau}^{\ell_{1}}D_{x}^{\ell_{2}}\psi(x,\tau)}\leq C_{m}\big(\log(1+\tau)\big)^{2}\sum_{\substack{\ell_{1}+\ell_{2}\leq m\\ \ell_{1}\leq m-2}}\norm{\partial_{\tau}^{\ell_{1}}D_{x}^{\ell_{2}}F}_{\infty,2}
\end{equation*}for any integer \(m\geq2\). 
\end{remark}

\begin{remark}\label{homwaveeq}
    The proof of Lemma \ref{Fquaddecaylemma} also provides an estimate for solutions of the homogeneous problem
    \begin{equation}
    \label{linearwavehom}
    \left\{\begin{array}{cl}
        -\partial_{\tau\tau}\psi+\Delta_{x}\psi=0,&\quad\text{ in}\quad\mathbb{R}^{2}\times(0,\infty),\\[4pt]
        \psi(x,0)=0,\quad\partial_{\tau}\psi(x,0)=g(x),&\quad\text{ in}\quad\mathbb{R}^{2}.        \end{array}\right.
    \end{equation}
    Indeed, for smooth \(g(x)\) the unique solution of (\ref{linearwavehom}) is given by
    \begin{equation*}
    \psi(x,\tau)=\frac{1}{2\pi}\int_{B(x,\tau)}\frac{g(x')}{\sqrt{\tau^{2}-\abs{x'-x}^{2}}}\:dx',
\end{equation*}and estimate (\ref{Ihomest}) then implies
\begin{equation*}
    \abs{\psi(x,\tau)}\leq\left\{\begin{array}{cc}
        \bigg(\frac{C\log(1+\tau)}{\sqrt{1+\tau}\sqrt{1+\abs{\tau-\abs{x}}}}\bigg)\norm{g}_{\infty,2},&\text{ if}\quad\abs{x}\leq2\tau,\\[10pt]
        \bigg(\frac{C\tau}{1+\abs{x}^{2}}\bigg)\norm{g}_{\infty,2},&\text{ if}\quad\abs{x}\geq2\tau.        \end{array}\right.
\end{equation*}
\end{remark}

Let us now return to the inhomogeneous problem (\ref{linearwaveeqdef}). For right-hand sides \(F(x,\tau)\) satisfying the linear decay estimate
\begin{equation*}
    \abs{F(x,\tau)}\leq\frac{C}{1+\abs{x}},
\end{equation*}we have the following result.

\begin{lemma}
\label{Flineardecaylemma}
    There exists a constant \(C>0\) such that for any function \(F(x,\tau)\) with 
    \begin{equation*}
        \norm{F}_{\infty,1}:=\sup_{\mathbb{R}^{2}\times(0,\infty)}\abs{F(x,\tau)}\big(1+\abs{x}\big)<\infty,
    \end{equation*}
    the unique solution \(\psi(x,\tau)\) of problem (\ref{linearwaveeqdef}) satisfies
    \begin{equation}
    \label{estFlineardecay}
        \abs{\psi(x,\tau)}\leq C\tau\big(\log(1+\tau)\big)\norm{F}_{\infty,1}.
    \end{equation}
\end{lemma}

\begin{proof}
    Using formula (\ref{duhamelformula}) we can estimate
    \begin{equation*}
        \abs{\psi(x,\tau)}\leq C\widetilde{I}(\abs{x},\tau)\norm{F}_{\infty,1}
    \end{equation*}where
    \begin{equation*}
        \widetilde{I}(\abs{x},\tau):=\frac{1}{2\pi}\int_{0}^{\tau}\int_{B(0,s)}\frac{1}{\sqrt{s^{2}-\abs{x'}^{2}}}\frac{1}{\sqrt{1+\abs{x'+x}^{2}}}\:dx'\:ds.
    \end{equation*}
    
    To bound this expression we may assume (as in the proof of Lemma \ref{Fquaddecaylemma}) that \(x=(\abs{x},0)\). Then writing the inner integral for \(\widetilde{I}(\abs{x},\tau)\) in polar coordinates and using standard trigonometric identities we get 
    \begin{equation*}
        \widetilde{I}(\abs{x},\tau)=\int_{0}^{\tau}\int_{0}^{s}\frac{r'K(k)}{\sqrt{s^{2}-(r')^{2}}\sqrt{1+(\abs{x}+r')^{2}}}\:dr'\:ds
    \end{equation*}where
    \begin{equation*}
        k^{2}:=\frac{4\abs{x}r'}{1+(\abs{x}+r')^{2}}
    \end{equation*}and
    \begin{equation*}
        K(k):=\int_{0}^{\pi/2}\frac{1}{\sqrt{1-k^{2}\sin^{2}\vartheta'}}\:d\vartheta'.
    \end{equation*}
    
    We recognise \(K(k)\) as the complete elliptic integral of the first kind, with the asymptotic behaviour (see \cite{carlson})
    \begin{equation*}
        K(k)\sim\frac{\pi}{2}\quad\text{as}\quad k\to0,\quad\quad K(k)\sim\frac{1}{2}\log\bigg(\frac{1}{1-k^{2}}\bigg)\quad\text{as}\quad k\to1.
    \end{equation*}It can then be verified that
    \begin{equation*}
        \widetilde{I}(\abs{x},\tau)\leq\left\{\begin{array}{cc}
        C\tau\big(\log(1+\tau)\big),&\text{ if}\quad\abs{x}\leq2\tau,\\[5pt]
        \frac{C\tau^{2}}{1+\abs{x}},&\text{ if}\quad\abs{x}\geq2\tau.        \end{array}\right.
    \end{equation*}
    (Again this proof provides a slightly stronger estimate than (\ref{estFlineardecay}), as it shows the solution \(\psi(x,\tau)\) has size \(O(\tau^{2}\abs{x}^{-1})\) in the region \(\abs{x}\geq2\tau\) outside the light cone).
    \end{proof}

    \begin{remark}\label{derivativesFlineardecay}
    As in Remark \ref{derivativesFquaddecay} we can estimate the derivatives of \(\psi\) using formula (\ref{dtaupsi}) and by differentiating (\ref{linearwaveeqdef}). We get
        \begin{equation*}
    \sum_{\ell_{1}+\ell_{2}\leq m}\abs{\partial_{\tau}^{\ell_{1}}D_{x}^{\ell_{2}}\psi(x,\tau)}\leq C_{m}\tau\big(\log(1+\tau)\big)\sum_{\substack{\ell_{1}+\ell_{2}\leq m\\ \ell_{1}\leq m-2}}\norm{\partial_{\tau}^{\ell_{1}}D_{x}^{\ell_{2}}F}_{\infty,1}
\end{equation*}for any integer \(m\geq2\).
\end{remark}

The final result of this section provides an estimate for (\ref{linearwaveeqdef}) when \(F(x,\tau)\) is compactly supported in \(\mathbb{R}^{2}\) for each \(\tau\). We define
\begin{equation*}
    \norm{F}_{\infty}:=\sup_{\mathbb{R}^{2}\times(0,\infty)}\abs{F(x,\tau)}.
\end{equation*}

\begin{lemma}
\label{Fcsupportlemma}
    Suppose there exists a number \(R>0\) such that 
    \begin{equation*}
        F(x,\tau)=0\quad\text{for all}\quad\abs{x}\geq R,\quad\tau\geq0.
    \end{equation*}Then the unique solution of (\ref{linearwaveeqdef}) satisfies
    \begin{equation*}
        \psi(x,\tau)=0\quad\text{whenever}\quad\abs{x}\geq R+\tau,
    \end{equation*}and we have the estimate
    \begin{equation*}
        \abs{\psi(x,\tau)}\leq C_{R}\log\big(2+\tau\big)\norm{F}_{\infty},\quad\text{for all}\quad(x,\tau)\in\mathbb{R}^{2}\times(0,\infty).
    \end{equation*}
\end{lemma}

\begin{proof}
    The fact that \(\psi(x,\tau)=0\) if \(\abs{x}\geq R+\tau\) follows directly from the representation formula (\ref{duhamelformula}). By estimating (\ref{duhamelformula}) separately in the regions \(\abs{x}\leq R\) and \(R\leq\abs{x}\leq R+\tau\) and using the property \(F(x',\tau-s)=0\) if \(\abs{x'}\geq R\), we see that the stated bound for \(\abs{\psi(x,\tau)}\) holds.
\end{proof}

\begin{remark}\label{derivativesFcompsupp}
    Under the assumptions of the previous lemma we have in fact
    \begin{equation*}
    \sum_{\ell_{1}+\ell_{2}\leq m}\abs{\partial_{\tau}^{\ell_{1}}D_{x}^{\ell_{2}}\psi(x,\tau)}\leq C\log(2+\tau)\sum_{\substack{\ell_{1}+\ell_{2}\leq m\\ \ell_{1}\leq m-2}}\norm{\partial_{\tau}^{\ell_{1}}D_{x}^{\ell_{2}}F}_{\infty}
    \end{equation*}for any integer \(m\geq2\), where \(C=C(m,R)\).
\end{remark}

\subsection{An energy estimate}\label{anenergyestsubsect} The results in the previous section were based on estimating the size of (\ref{duhamelformula}) in a pointwise sense. We now turn our attention to energy estimates.

\smallskip

For solutions \(\psi(x,\tau)\) of (\ref{linearwaveeqdef}) with sufficiently fast decay as \(\abs{x}\to\infty\), we can multiply the equation by \(\partial_{\tau}\psi\) and integrate over \(\mathbb{R}^{2}\) to obtain the energy identity
\begin{equation}
\label{waveenergyidentity}
    \frac{d}{d\tau}\bigg(\norm{\partial_{\tau}\psi(x,\tau)}_{L^{2}_{x}}^{2}+\norm{\nabla_{x}\psi(x,\tau)}_{L^{2}_{x}}^{2}\bigg)=-2\int_{\mathbb{R}^{2}}F\big(\partial_{\tau}\psi\big)\:dx
\end{equation} where
\begin{equation*}
    \norm{\psi(x,\tau)}_{L^{2}_{x}}^{2}:=\int_{\mathbb{R}^{2}}\abs{\psi(x,\tau)}^{2}\:dx.
\end{equation*}
Applying the Cauchy-Schwarz inequality on the right-hand side of (\ref{waveenergyidentity}) and integrating in \(\tau\), we get the estimate
\begin{equation}
\label{basicwaveenergyest}
    \norm{\partial_{\tau}{\psi(x,\tau)}}_{L_{x}^{2}}+\norm{\nabla_{x}\psi(x,\tau)}_{L_{x}^{2}}\leq C\int_{0}^{\tau}\norm{F(x,s)}_{L^{2}_{x}}\:ds.
\end{equation}

We note that (\ref{basicwaveenergyest}) alone is not sufficient for the purpose of Section \ref{inductionapproxsubsect} since it allows the left hand side to grow like \(O(\tau)\) for large \(\tau\) if \(\norm{F(x,s)}_{L^{2}_{x}}\) is merely bounded in \(s\). To get an \(O(\tau^{1/2}\log\tau)\) growth rate instead we use the following lemma.

\begin{lemma}\label{anenergyestlemma}
    Let \(\norm{F}_{\infty,2}:=\sup_{\:\mathbb{R}^{2}\times(0,\infty)}\abs{F}(1+\abs{x}^{2})\) be the norm defined in Lemma \ref{Fquaddecaylemma}. Then the estimate
    \begin{equation*}
    \begin{split}
        &\norm{\partial_{\tau}{\psi(x,\tau)}}_{L_{x}^{2}}+\norm{\nabla_{x}\psi(x,\tau)}_{L_{x}^{2}}\\
        &\leq C\tau^{1/2}\big(\log(1+\tau)\big)\bigg(\norm{F}_{\infty,2}+\norm{\nabla_{x}F}_{\infty,2}\bigg)^{1/2}\sup_{s\in[0,\tau]}\bigg(\int_{\mathbb{R}^{2}}\abs{F(x,s)}\:dx\bigg)^{1/2}
    \end{split}
    \end{equation*}holds for the unique solution \(\psi(x,\tau)\) of (\ref{linearwaveeqdef}), provided the norms on the right-hand side of the above are finite.
\end{lemma}

\begin{proof}
    If \(F(x,\tau)\) satisfies \(\norm{F}_{\infty,2}+\norm{\nabla_{x}F}_{\infty,2}<\infty\), the corresponding solution of (\ref{linearwaveeqdef}) satisfies
    \begin{equation*}
     \abs{\partial_{\tau}\psi(x,\tau)}\leq C\big(\log(1+\tau)\big)^{2}\bigg(\norm{F}_{\infty,2}+\norm{\nabla_{x}F}_{\infty,2}\bigg)   
    \end{equation*}by Remark \ref{derivativesFquaddecay}. We can use this bound for \(\partial_{\tau}\psi\) in the energy identity (\ref{waveenergyidentity}) to deduce
    \begin{equation}
    \label{waveenergyidentitysqrtest}
    \begin{split}
    \frac{d}{d\tau}\bigg(\norm{\partial_{\tau}\psi(x,\tau)}_{L^{2}_{x}}^{2}+\norm{\nabla_{x}&\psi(x,\tau)}_{L^{2}_{x}}^{2}\bigg)\\
    \leq &C\big(\log(1+\tau)\big)^{2}\bigg(\norm{F}_{\infty,2}+\norm{\nabla_{x}F}_{\infty,2}\bigg)\int_{\mathbb{R}^{2}}\abs{F}\:dx,
    \end{split}
\end{equation}
and the stated result then follows after integrating (\ref{waveenergyidentitysqrtest}) in \(\tau\) and taking a square root on both sides.
\end{proof}

We conclude Section \ref{waveestimatesect} by noting that all results stated here remain valid if (\ref{linearwaveeqdef}) is considered on a finite time interval \(\tau\in[0,\widetilde{T}]\). In that case the norms \(\norm{F}_{\infty,2}\), \(\norm{F}_{\infty,1}\) and \(\norm{F}_{\infty}\) should be defined with the relevant suprema taken over \(\mathbb{R}^{2}\times[0,\widetilde{T}]\) rather than \(\mathbb{R}^{2}\times(0,\infty)\).

\section{Linear estimates for the full problem}\label{linearfullprobsect}

From now on we fix \(u_{*}\) to be the \(O(\epsilon^{m})\) approximation described in Proposition \ref{arbitraryapproxprop}, with \(m\geq9\) a large integer to be chosen later. The purpose of this section is to obtain estimates for the linear operator
\begin{equation*}
    S'(u_{*})[\phi]:=\epsilon^{2}i\phi_{t}+\Delta_{y}\phi+(1-\abs{u_{*}}^{2})\phi-2\operatorname{Re}(\overline{u}_{*}\phi)u_{*}.
\end{equation*}

As outlined in \(\S\)\ref{energyestimatesoutline}, we will be interested in the linear problem
\begin{equation}\label{fulllinprobpensect}
\begin{cases}
S'(u_{*})[\phi]=f(y,t),&\quad\text{in}\quad\mathbb{R}^{2}\times[0,T],\\
    \hspace{1.1em}\phi(y,0)=0,&\quad\text{in}\quad\mathbb{R}^{2},  
    \end{cases}
\end{equation}where here and in what follows all functions are expressed in the rescaled space variable \(y=x/\epsilon\) and the original time variable \(t\). Let us define the standard Sobolev norms
\begin{gather*}
\norm{\phi(\cdot,t)}_{H_{y}^{k}}^{2}:=\sum_{\ell=0}^{k}\int_{\mathbb{R}^{2}}\abs{D_{y}^{\ell}\phi(y,t)}^{2}\:dy,\\[5pt]
\norm{\phi}_{L^{\infty}_{t}H^{k}_{y}}:=\sup_{t\in[0,T]}\norm{\phi(\cdot,t)}_{H_{y}^{k}},
\end{gather*}for \(k\geq0\). The main result of this section is the following.

\begin{proposition}\label{mainlinearizedprop}
    There exist constants \(C,\kappa>0\) such that for any \(f(y,t)\) satisfying \linebreak \(\norm{f}_{L^{\infty}_{t}H^{1}_{y}}<\infty\), problem (\ref{fulllinprobpensect}) has a unique solution \(\phi(y,t)\) such that
    \begin{equation}\label{fulllinearizedH1est}
        \norm{\phi}_{L^{\infty}_{t}H^{1}_{y}}\leq C\epsilon^{-4-\kappa T}\norm{f}_{L^{\infty}_{t}H^{1}_{y}}.
    \end{equation}
\end{proposition}

The proof of estimate (\ref{fulllinearizedH1est}) proceeds, as described in \(\S\)\ref{energyestimatesoutline}, by considering the evolution of a suitable energy functional. In \(\S\)\ref{aquadformsubsect} we introduce a key quadratic form \(\mathcal{B}[\phi,\phi]\) for this analysis and compute its time derivative along solutions of (\ref{fulllinprobpensect}). The absence of large error terms in the evolution is ensured with a good choice of coefficients \(A(y,t)\) and \(B(y,t)\), whose precise definition is given in \(\S\)\ref{AandBdefsubsect}. In \(\S\)\ref{coercivitysubsect} we prove a lower bound for \(\mathcal{B}[\phi,\phi]\) on a finite-codimension subspace. Finally \(\S\)\ref{linearestproofsubsect} involves the construction of a global positive quantity
\begin{equation*}
    \mathcal{Q}(t)=\mathcal{B}\big[\phi(t),\phi(t)\big]+\text{additional terms}
\end{equation*}for solutions of (\ref{fulllinprobpensect}), from which differentiation in \(t\) and Gr\"{o}nwall's inequality leads to the desired conclusion.

\subsection{A quadratic form}\label{aquadformsubsect} For a smooth vector field \(A(y,t):\mathbb{R}^{2}\times[0,T]\to\mathbb{R}^{2}\) and smooth function \(B(y,t):\mathbb{R}^{2}\times[0,T]\to\mathbb{R}\) to be specified later, we define the symmetric bilinear form
\begin{equation*}
\begin{split}
    \mathcal{B}[\phi,\hat\phi]:=&\operatorname{Re}\int_{\mathbb{R}^{2}}\nabla\phi\cdot\overline{\nabla\hat{\phi}}-\operatorname{Re}\int_{\mathbb{R}^{2}}(1-\abs{u_{*}}^{2})\phi\overline{\hat{\phi}}+\int_{\mathbb{R}^{2}}2\operatorname{Re}(\overline{u}_{*}\phi)\operatorname{Re}(\overline{u}_{*}\hat\phi)\\
    &+\operatorname{Re}\int_{\mathbb{R}^{2}}iA\cdot\nabla\phi\overline{\hat{\phi}}+\frac{1}{2}\operatorname{Re}\int_{\mathbb{R}^{2}}i(\operatorname{div}A)\phi\overline{\hat{\phi}}+\operatorname{Re}\int_{\mathbb{R}^{2}}B\phi\overline{\hat\phi}.
\end{split}    
\end{equation*} All spatial derivatives from now on are taken with respect to \(y\), i.e. \(\nabla=\nabla_{y}\). Setting \(\hat\phi=\phi\) in \(\mathcal{B}[\phi,\hat\phi]\), we obtain the corresponding quadratic form
\begin{equation*}
  \mathcal{B}[\phi,\phi]=\int_{\mathbb{R}^{2}}\abs{\nabla\phi}^{2}-\int_{\mathbb{R}^{2}}(1-\abs{u_{*}}^{2})\abs{\phi}^{2}+\int_{\mathbb{R}^{2}}2\operatorname{Re}(\overline{u}_{*}\phi)^{2}+\operatorname{Re}\int_{\mathbb{R}^{2}}iA\cdot\nabla\phi\overline{\phi} +\int_{\mathbb{R}^{2}}B\abs{\phi}^{2}.  
\end{equation*}

Our first objective is to compute the evolution of \(\mathcal{B}[\phi,\phi]\) under the linear dynamics (\ref{fulllinprobpensect}). Let \(\rho(y,t)\) and \(\varphi(y,t)\) denote the modulus and phase of \(u_{*}(y,t)\) respectively, so that \(\rho:=\abs{u_{*}}\) and \(u_{*}=\rho e^{i\varphi}\). We have the following result.
\begin{lemma}\label{lemmtimederivqf}
    For any smooth solution \(\phi=\phi(y,t)\) of (\ref{fulllinprobpensect}) with sufficient decay as \(\abs{y}\to\infty\), we have
    \begin{equation}\label{timederivquadform}
    \frac{d}{dt}\mathcal{B}[\phi,\phi]=\frac{2}{\epsilon^{2}}\mathcal{B}[\phi,-if]+\mathcal{R}(\phi,\phi),
\end{equation}where
\begin{equation}\label{quadformremainder}
\begin{split}
    \mathcal{R}(\phi,\phi):=&-\frac{2}{\epsilon^{2}}\operatorname{Re}\int_{\mathbb{R}^{2}}\big(\nabla A\cdot\nabla\phi\big)\cdot\overline{\nabla\phi}\\[3pt]
    &+\frac{1}{\epsilon^{2}}\operatorname{Re}\int_{\mathbb{R}^{2}}i\big(\epsilon^{2}\partial_{t}A-2\nabla B\big)\cdot(\nabla\phi)\overline{\phi}\\[3pt]
    &+\frac{1}{\epsilon^{2}}\int_{\mathbb{R}^{2}}\bigg(\epsilon^{2}\partial_{t}(\rho^{2})+\nabla(\rho^{2})\cdot A+\epsilon^{2}\partial_{t}B+\tfrac{1}{2}\Delta(\operatorname{div}A)\bigg)\abs{\phi}^{2}\\[3pt]
    &+\frac{2}{\epsilon^{2}}\int_{\mathbb{R}^{2}}\big(\epsilon^{2}\partial_{t}(\rho^{2})+\nabla(\rho^{2})\cdot A\big)\operatorname{Re}(e^{-i\varphi}\phi)^{2}\\[3pt]
    &-\frac{4}{\epsilon^{2}}\int_{\mathbb{R}^{2}}\operatorname{Re}\big(\overline{u}_{*}\big(\epsilon^{2}\partial_{t}\varphi+\nabla\varphi\cdot A-B\big)i\phi\big)\operatorname{Re}(\overline{u}_{*}\phi).
\end{split}    
\end{equation}    
\end{lemma}

\begin{proof}
    The proof follows by direct calculation and integration by parts, using the equation
    \begin{equation*}
        \epsilon^{2}i\phi_{t}+\Delta\phi+(1-\abs{u_{*}}^{2})\phi-2\operatorname{Re}(\overline{u}_{*}\phi)u_{*}=f    
    \end{equation*}to rewrite \(\phi_{t}\) in terms of \(\Delta\phi\), \(\phi\) and \(f\). We leave the details to the reader. Note that \(\nabla A\) denotes the \(2\times2\) matrix consisting of the partial \(y\) derivatives of the two components of \(A\).
\end{proof}

For \(A=B=0\), the remainder \(\mathcal{R}(\phi,\phi)\) in (\ref{timederivquadform}) contains terms \(\epsilon^{-1}\dot{\xi}_{j}(t)\) of size \(O(\epsilon^{-1})\) which are too large to derive any reasonable control on \(\phi(y,t)\) using Gr\"{o}nwall's inequality. To address this issue we will make a nonzero choice of \(A\) and \(B\) which eliminates the largest contributions from \(\partial_{t}(\rho^{2})\) and \(\partial_{t}\varphi\) in (\ref{quadformremainder}). 

\smallskip

Before doing so, let us record an alternative expression for \(\mathcal{R}(\phi,\phi)\) in terms of the function \(\psi=\psi_{1}+i\psi_{2}\) defined via \(\phi=iu_{*}\psi\). Recall (\textit{cf}. Remark \ref{zeroesofustar}) that the points \(\xi(t)\approx\xi^{0}(t)\), \(\xi(t)=\big(\xi_{1}(t),\ldots,\xi_{n}(t)\big)\) described in Proposition \ref{arbitraryapproxprop} correspond (after multiplication by \(\epsilon^{-1}\)) precisely to the zeroes of our \(n\)-vortex approximation \(u_{*}(y,t)\).

\begin{lemma}\label{remainderforminpsi}
    Let \(\phi=iu_{*}\psi\) be a smooth, compactly supported function vanishing in a neighbourhood of the points \(\tilde{\xi}_{j}(t)=\epsilon^{-1}\xi_{j}(t)\) for \(j=1,\ldots,n\). Then we can express
\begin{equation*}
    \mathcal{R}(\phi,\phi)=\mathcal{R}(iu_{*}\psi,iu_{*}\psi)=\sum_{k=1}^{6}\mathcal{R}_{k}(\psi,\psi)\end{equation*}where
    \begin{align}
    \begin{split}
    \mathcal{R}_{1}(\psi,\psi)&:=-\frac{2}{\epsilon^{2}}\operatorname{Re}\int_{\mathbb{R}^{2}}\rho^{2}\big(\nabla A\cdot\nabla\psi\big)\cdot\overline{\nabla\psi},
    \end{split}\label{calR1def}\\[5pt]
    \begin{split}
    \mathcal{R}_{2}(\psi,\psi)&:=\frac{2}{\epsilon^{2}}\int_{\mathbb{R}^{2}}\rho^{2}\bigg(\epsilon^{2}\partial_{t}A-2\nabla B+2(\nabla A+\nabla A^{T})\cdot\nabla\varphi\bigg)\cdot(\nabla\psi_{1})\psi_{2},
    \end{split}\label{calR2def}\\[5pt]
    \begin{split}
    \mathcal{R}_{3}(\psi,\psi)&:=\frac{1}{\epsilon^{2}}\int_{\mathbb{R}^{2}}\bigg(\epsilon^{2}\partial_{t}\mu+\operatorname{div}(\mu2\nabla\varphi)-\epsilon^{2}\partial_{t}\big(2\rho^{2}\mathcal{S}_{2}^{*}\big)\\
    &\hspace{5em}-\operatorname{div}(2\rho^{2}\mathcal{S}_{2}^{*}A)-\operatorname{div}(2\rho^{2}\nabla\lambda)+4\rho^{4}\lambda\bigg)\psi_{1}\psi_{2},
    \end{split}\label{calR3def}\\[5pt]
    \begin{split}
        \mathcal{R}_{4}(\psi,\psi)&:=\frac{1}{\epsilon^{2}}\int_{\mathbb{R}^{2}}\rho^{2}\bigg(-\epsilon^{2}\partial_{t}\mathcal{S}^{*}_{1}-\nabla \mathcal{S}_{1}^{*}\cdot A+\epsilon^{2}\partial_{t}\lambda+2\nabla\lambda\cdot\nabla\varphi\bigg)\abs{\psi}^{2},
    \end{split}\label{calR4def}\\[5pt]
    \begin{split}
        \mathcal{R}_{5}(\psi,\psi)&:=-\frac{1}{\epsilon^{2}}\int_{\mathbb{R}^{2}}\rho^{2}\nabla\bigg(\frac{\mu}{2\rho^{2}}\bigg)\cdot\nabla\abs{\psi}^{2},
    \end{split}\label{calR5def}\\[5pt]
    \begin{split}
        \mathcal{R}_{6}(\psi,\psi)&:=\frac{2}{\epsilon^{2}}\int_{\mathbb{R}^{2}}\bigg(\epsilon^{2}\partial_{t}(\rho^{2})+\nabla(\rho^{2})\cdot A\bigg)\rho^{2}\abs{\psi_{2}}^{2}
    \end{split}\label{calR6def}
    \end{align}and
    \begin{gather}
        \mu:=\epsilon^{2}\partial_{t}(\rho^{2})+\operatorname{div}(\rho^{2}A),\label{mudef}\\[3pt]
        \lambda:=B-\epsilon^{2}\partial_{t}\varphi-\nabla\varphi\cdot A,\label{lambdadef}\\[3pt]
        \mathcal{S}_{1}^{*}:=\operatorname{Re}\bigg(\frac{S(u_{*})}{u_{*}}\bigg),\quad\quad \mathcal{S}_{2}^{*}:=\operatorname{Im}\bigg(\frac{S(u_{*})}{u_{*}}\bigg).\notag
    \end{gather}
\end{lemma}

\begin{proof}
    For \(\phi=iu_{*}\psi=i\rho e^{i\varphi}\psi\) as stated, let us rewrite the integrals involving \(\nabla\phi\) in (\ref{quadformremainder}) in terms of \(\nabla\rho\), \(\nabla\varphi\) and \(\nabla\psi\). We find, after direct computation and integration by parts,
    \begingroup
    \allowdisplaybreaks
        \begin{align*}
        \operatorname{Re}\int_{\mathbb{R}^{2}}\big(\nabla A\cdot\nabla\phi\big)\cdot\overline{\nabla\phi}=&\quad\operatorname{Re}\int_{\mathbb{R}^{2}}\rho^{2}\big(\nabla A\cdot\nabla\psi\big)\cdot\overline{\nabla\psi} \\
        &-\operatorname{Re}\int_{\mathbb{R}^{2}}\rho^{2}\big(i(\nabla A+\nabla A^{T})\cdot\nabla\varphi\big)\cdot(\nabla\psi)\overline{\psi}\\
        &+\int_{\mathbb{R}^{2}}(\nabla A\cdot\nabla\rho)\cdot\nabla\rho\:\abs{\psi}^{2}\\
        &+\int_{\mathbb{R}^{2}}\rho^{2}(\nabla A\cdot\nabla\varphi)\cdot\nabla\varphi\:\abs{\psi}^{2}\\
        &-\frac{1}{4}\int_{\mathbb{R}^{2}}\operatorname{div}\big((\nabla A+\nabla A^{T})\cdot\nabla(\rho^{2})\big)\abs{\psi}^{2}
        \end{align*}
        \endgroup
        and

        \begin{equation*}
        \begin{split}\operatorname{Re}\int_{\mathbb{R}^{2}}i\big(\epsilon^{2}\partial_{t}A-2\nabla B\big)\cdot\nabla\phi\overline{\phi}  & = \operatorname{Re}\int_{\mathbb{R}^{2}}\rho^{2}i\big(\epsilon^{2}\partial_{t}A-2\nabla B\big)\cdot(\nabla\psi)\overline{\psi}\\
        &\phantom{=}-\int_{\mathbb{R}^{2}}\rho^{2}\big(\epsilon^{2}\partial_{t}A-2\nabla B\big)\cdot\nabla\varphi\:\abs{\psi}^{2}.
        \end{split}\end{equation*}
        
        Next, we observe that the difference
        \begin{equation*}
            -2\operatorname{Re}\int_{\mathbb{R}^{2}}\big(\nabla A\cdot\nabla\phi\big)\cdot\overline{\nabla\phi}+\operatorname{Re}\int_{\mathbb{R}^{2}}i\big(\epsilon^{2}\partial_{t}A-2\nabla B\big)\cdot\nabla\phi\overline{\phi}
            \end{equation*}involves an integral of the form 
        \begin{equation*}
            \operatorname{Re}\int_{\mathbb{R}^{2}}\rho^{2}i\big(\mathcal{G}\cdot\nabla\psi\big)\overline{\psi}=2\int_{\mathbb{R}^{2}}\rho^{2}\big(\mathcal{G}\cdot\nabla\psi_{1}\big)\psi_{2}+\int_{\mathbb{R}^{2}}\operatorname{div}\big(\rho^{2}\mathcal{G}\big)\psi_{1}\psi_{2}
        \end{equation*}where
        \begin{equation*}
            \mathcal{G}:=\epsilon^{2}\partial_{t}A-2\nabla B+2(\nabla A+\nabla A^{T})\cdot\nabla\varphi.
        \end{equation*}
We have
\begin{equation*}
            \operatorname{div}\big(\rho^{2}\mathcal{G}\big)=\epsilon^{2}\partial_{t}\mu+\operatorname{div}\big(\mu2\nabla\varphi\big)-\epsilon^{2}\partial_{t}\big(2\rho^{2}\mathcal{S}_{2}^{*}\big)-\operatorname{div}\big(2\rho^{2}\mathcal{S}_{2}^{*}A\big)-\operatorname{div}\big(\rho^{2}2\nabla\lambda\big)
        \end{equation*}using the identity  \begin{equation*}
            \mathcal{S}_{2}^{*}:=\operatorname{Im}\bigg(\frac{S(u_{*})}{u_{*}}\bigg)=\frac{1}{2\rho^{2}}\bigg(\epsilon^{2}\partial_{t}(\rho^{2})+\operatorname{div}\big(\rho^{2}2\nabla\varphi\big)\bigg)
        \end{equation*}and the expressions (\ref{mudef}), (\ref{lambdadef}) for \(\mu\), \(\lambda\). It then follows
        \begin{equation*}
            \mathcal{R}(iu_{*}\psi,iu_{*}\psi)=\sum_{k=1}^{3}\mathcal{R}_{k}(\psi,\psi)+\widetilde{\mathcal{R}}_{4}(\psi,\psi)+\mathcal{R}_{6}(\psi,\psi)
        \end{equation*}where \(\mathcal{R}_{1}\), \(\mathcal{R}_{2}\), \(\mathcal{R}_{3}\), \(\mathcal{R}_{6}\) are defined by (\ref{calR1def}), (\ref{calR2def}), (\ref{calR3def}), (\ref{calR6def}) respectively, and 
        \begin{equation}\label{calR4tilde}
        \begin{split}
        \widetilde{\mathcal{R}}_{4}(\psi,\psi):=\frac{1}{\epsilon^{2}}\int_{\mathbb{R}^{2}}\rho^{2}\bigg(&\epsilon^{2}\partial_{t}(\rho^{2})+\nabla(\rho^{2})\cdot A+\epsilon^{2}\partial_{t}B+\tfrac{1}{2}\Delta(\operatorname{div}A)\\
        &-(\epsilon^{2}\partial_{t}A-2\nabla B)\cdot\nabla\varphi-2(\nabla A\cdot\nabla\varphi)\cdot\nabla\varphi\\
        &-2\rho^{-2}(\nabla A\cdot\nabla\rho)\cdot\nabla\rho+\tfrac{1}{2}\rho^{-2}\operatorname{div}\big((\nabla A+\nabla A^{T})\cdot\nabla(\rho^{2})\big)\bigg)\abs{\psi}^{2}.
        \end{split}
        \end{equation}

        To conclude we note that the expression in parentheses in (\ref{calR4tilde}) is equal to 
        \begin{equation*}
            \big(-\epsilon^{2}\partial_{t}\mathcal{S}_{1}^{*}-\nabla\mathcal{S}_{1}^{*}\cdot A\big)+\big(\epsilon^{2}\partial_{t}\lambda+2\nabla\lambda\cdot\nabla\varphi\big)+\rho^{-2}\operatorname{div}\bigg(\rho^{2}\nabla\bigg(\frac{\mu}{2\rho^{2}}\bigg)\bigg),
        \end{equation*}
        as can be seen using the identity
        \begin{equation*}
            \mathcal{S}_{1}^{*}:=\operatorname{Re}\bigg(\frac{S(u_{*})}{u_{*}}\bigg)=-\epsilon^{2}\partial_{t}\varphi-\abs{\nabla\varphi}^{2}+\frac{\Delta\rho}{\rho}+(1-\rho^{2}).
        \end{equation*}Integration by parts then gives \(\widetilde{\mathcal{R}}_{4}(\psi,\psi)=\mathcal{R}_{4}(\psi,\psi)+\mathcal{R}_{5}(\psi,\psi)\) where \(\mathcal{R}_{4}\) and \(\mathcal{R}_{5}\) are defined by (\ref{calR4def}) and (\ref{calR5def}). The proof of the lemma is complete.
\end{proof}

\subsection{Definition of \texorpdfstring{\(A\)}{A} and \texorpdfstring{\(B\)}{B}}\label{AandBdefsubsect} Inspecting the expression for \(\mathcal{R}(\phi,\phi)=\mathcal{R}(iu_{*}\psi,iu_{*}\psi)\) computed in Lemma \ref{remainderforminpsi}, we see that the integrals (\ref{calR1def})-(\ref{calR6def}) consist of various terms which all involve some combination of \(\mathcal{S}_{1}^{*}\), \(\mathcal{S}_{2}^{*}\), \(\mu\), \(\lambda\) and the derivatives of \(A\) and \(B\). The quantities \(\mathcal{S}_{1}^{*}\) and \(\mathcal{S}_{2}^{*}\) have size \(O(\epsilon^{m})\) for \(m\geq9\) and can essentially be neglected: our goal is to find \(A\) and \(B\) so that \(\mu\), \(\lambda\) and \(\nabla A\) etc. are as small as possible. 

\medskip

\textbf{Modulus and phase of \(\boldsymbol{u_{*}}\).} For \(\tilde{\xi}_{j}(t)=\epsilon^{-1}\xi_{j}(t)\) and \(y_{j}=y-\tilde{\xi}_{j}(t)\), we recall the notation
\begin{equation*}
   W_{j}(y,t)=\begin{cases}
    W(y_{j}),\quad\text{for }j\in I_{+},\\
    \overline{W}(y_{j}),\quad\text{for }j\in I_{-}.
   \end{cases}
\end{equation*}

With the lower order corrections \(\phi_{j}=iW_{j}\psi_{j}\) and \(\psi^{out}\) described in Proposition \ref{arbitraryapproxprop}, the expression (\ref{newapprox}) for our \(n\)-vortex approximation \(u_{*}\) can be written in the form
\begin{equation*}
u_{*}(y,t)=e^{i\psi^{out}}\prod_{j=1}^{n}W_{j}\bigg(\eta_{j}(1+i\tilde{\eta}_{j}\psi_{j})+(1-\eta_{j})e^{i\tilde{\eta}_{j}\psi_{j}}\bigg)
\end{equation*}where \(\psi^{out}(x,\tau)=O(\epsilon^{2}\abs{\log\epsilon}^{2})\) and \(\psi_{j}(y_{j},t)=O(\epsilon^{2}\abs{\log\epsilon})\). We now claim that there exists a bounded function \(\psi_{*}^{in}(y,t)\), smooth away from the points \(\tilde{\xi}_{j}(t)\), such that 
\begin{equation}\label{psistarineq}
    e^{i\psi_{*}^{in}}=\prod_{j=1}^{n}\bigg(\eta_{j}(1+i\tilde{\eta}_{j}\psi_{j})+(1-\eta_{j})e^{i\tilde{\eta}_{j}\psi_{j}}\bigg).
\end{equation}

This assertion follows easily from the boundedness of the functions \(\psi_{j}=\psi_{j1}+i\psi_{j2}\) in a neighbourhood of \(y_{j}=0\) (\textit{cf}. Remark \ref{zeroesofustar}) and the estimate \(\psi_{j}(y_{j},t)=O(\epsilon^{2}\abs{\log\epsilon})\). In fact, since \(\eta_{j}=1\) for \(\abs{y_{j}}\leq1\) and \(\eta_{j}=0\) for \(\abs{y_{j}}\geq2\), we can find \(\psi_{*}^{in}(y,t)=O(\epsilon^{2}\abs{\log\epsilon})\) satisfying (\ref{psistarineq}) with
\begin{gather*}
    \psi_{*1}^{in}(y,t):=\operatorname{Re}(\psi_{*}^{in})=\tan^{-1}\bigg(\frac{\psi_{j1}}{1-\psi_{j2}}\bigg)\quad\text{for}\quad\abs{y_{j}}\leq1,\\[3pt]
    \psi^{in}_{*2}(y,t):=\operatorname{Im}(\psi_{*}^{in})=-\tfrac{1}{2}\log\bigg((1-\psi_{j2})^{2}+\psi_{j1}^{2}\bigg)\quad\text{for}\quad\abs{y_{j}}\leq1,
    \end{gather*}and
    \begin{equation*}
        \psi_{*}^{in}(y,t)=\sum_{j=1}^{n}\tilde{\eta}_{j}\psi_{j}\quad\text{if}\quad\abs{y_{j}}\geq2\quad\text{for all }j.
    \end{equation*}

\medskip

Now taking the modulus of
\begin{equation*}
    u_{*}(y,t)=e^{i(\psi^{out}+\psi^{in}_{*})}\prod_{j=1}^{n}W_{j}
\end{equation*}in \(\mathbb{R}^{2}\times[0,T]\), we deduce 
\begin{equation}\label{rhoexp}
    \rho(y,t):=\abs{u_{*}(y,t)}=e^{-(\psi^{out}_{2}+\psi_{*2}^{in})}\prod_{j=1}^{n}w_{j}
\end{equation}where \(\psi^{out}_{2}\), \(\psi^{in}_{*2}\) denote the imaginary parts of \(\psi^{out}\), \(\psi^{in}_{*}\), and \(w_{j}=\abs{W_{j}}\). For the phase \(\varphi(y,t)\) of \(u_{*}\), we find
    \begin{equation}\label{phaseustarexp}
        \varphi(y,t)=\sum_{j=1}^{n}d_{j}\theta(y_{j})+\psi_{1}^{out}+\psi^{in}_{*1}
    \end{equation}where \(\theta(y)\) denotes the polar angle of \(y\in\mathbb{R}^{2}\), and \(\psi_{1}^{out}\), \(\psi_{*1}^{in}\) are the real parts of \(\psi^{out}\), \(\psi^{in}_{*}\). Note that the angles \(\theta(y_{j})\) are only defined up to a multiple of \(2\pi\), but the derivatives of \(\varphi(y,t)\) are single-valued and smooth away from the points \(\tilde{\xi}_{j}(t)\).

    \medskip

    \textbf{Definition of \(A\).} For \(\rho(y,t)\) given by (\ref{rhoexp}), we seek a vector field \(A(y,t)\) which approximately annihilates the quantity
    \begin{equation*}
    \mu(y,t)=\epsilon^{2}\partial_{t}(\rho^{2})+\operatorname{div}(\rho^{2}A).
    \end{equation*}

    Let us first make some heuristic observations. In the region close to \(\tilde{\xi}_{j}(t)=\epsilon^{-1}\xi_{j}(t)\) we have \(\rho^{2}\approx w_{j}^{2}\) and
    \begin{equation}
    \label{approxmuclose}
        \mu(y,t)\approx 2w_{j}^{2}\bigg(\frac{\nabla_{y}w_{j}}{w_{j}}\cdot(-\epsilon\dot{\xi}_{j}+A)+\tfrac{1}{2}\operatorname{div}A\bigg).
    \end{equation}Asking that the right hand side of (\ref{approxmuclose}) vanishes, it is natural to set \(A(y,t)=\epsilon\dot\xi_{j}(t)\). 
    
    Far from all vortices, on the other hand, we have \(\rho^{2}\approx 1\) with \(\epsilon^{2}\partial_{t}(\rho^{2})\) small. Then it is natural to set \(A(y,t)=0\). We will define \(A(y,t)\) so these properties hold, with \(A(y,t)\) interpolating smoothly between \(\epsilon\dot{\xi}_{j}(t)\) and \(0\) in a certain intermediate region. 

Recall that \(\tilde{\eta}_{j}\) is a smooth cut-off such that
\begin{equation}
    \tilde{\eta}_{j}(y,t)=\begin{cases}
   1,\quad\text{if}\quad\abs{y-\tilde{\xi}_{j}}\leq\delta\epsilon^{-1},\\
   0,\quad\text{if}\quad\abs{y-\tilde{\xi}_{j}}\geq2\delta\epsilon^{-1}.\label{tildeetadef}
   \end{cases}
\end{equation}Taking \(\tilde{\eta}_{j}^{(1/2)}:=\eta_{0}\big(2\epsilon\delta^{-1}(y-\tilde{\xi}_{j})\big)\) a new smooth cut-off such that
\begin{equation*}
    \tilde{\eta}_{j}^{(1/2)}(y,t):=\begin{cases}
   1,\quad\text{if}\quad\abs{y-\tilde{\xi}_{j}}\leq\frac{1}{2}\delta\epsilon^{-1},\\
   0,\quad\text{if}\quad\abs{y-\tilde{\xi}_{j}}\geq\delta\epsilon^{-1},
   \end{cases}
\end{equation*}we define
\begin{equation}
\label{Adef}    A(y,t):=\sum_{j=1}^{n}\tilde{\eta}_{j}\bigg(\epsilon\dot{\xi}_{j}-\rho^{-2}\nabla^{\perp}\bigg((1-\tilde{\eta}_{j}^{(1/2)})\epsilon\dot{\xi}_{j}\cdot y_{j}^{\perp}\bigg)\bigg).
\end{equation}

The following result holds. 

\begin{lemma}
    For A(y,t) defined by (\ref{Adef}), we have
    \begin{equation}\label{muest}
        \mu(y,t)=\epsilon^{2}\partial_{t}(\rho^{2})+\operatorname{div}(\rho^{2}A)=O(\epsilon^{4}\abs{\log\epsilon}^{2})
    \end{equation}uniformly in \(\mathbb{R}^{2}\times[0,T]\).
\end{lemma}

\begin{proof}We consider the three regions
\begin{align*}
    \text{(i)}&\:\:\:\abs{y-\tilde{\xi}_{j}}\leq \tfrac{1}{2}\delta\epsilon^{-1}\text{ for some }j,\\
    \text{(ii)}&\:\:\:\tfrac{1}{2}\delta\epsilon^{-1}\leq\abs{y-\tilde{\xi}_{j}}\leq 2\delta\epsilon^{-1}\text{ for some }j,\\
    \text{(iii)}&\:\:\:\abs{y-\tilde{\xi}_{j}}\geq 2\delta\epsilon^{-1}\text{ for all }j.    \end{align*}

    In region (i), expression (\ref{Adef}) takes the form \(A(y,t)=\epsilon\dot\xi_{j}(t)\). We compute
    \begin{equation*}
    \begin{split}
        \mu(y,t)&=2\rho^{2}\Bigg(\frac{\nabla_{y} w_{j}}{w_{j}}\cdot\big(-\epsilon\dot{\xi}_{j}+A\big)+\sum_{k\neq j}\frac{\nabla_{y} w_{k}}{w_{k}}\cdot\big(-\epsilon\dot{\xi}_{k}+A\big)\\[3pt]
        &\hspace{4em}-\epsilon^{2}\big(\partial_{t}\psi_{2}^{out}+\partial_{t}\psi_{*2}^{in}\big)-\big(\nabla_{y}\psi_{2}^{out}+\nabla_{y}\psi_{*2}^{in}\big)\cdot A\Bigg)\\[3pt]
        &=2\rho^{2}\Bigg(\sum_{k\neq j}\frac{\nabla_{y}w_{k}}{w_{k}}\cdot\big(-\epsilon\dot{\xi}_{k}+\epsilon\dot{\xi}_{j}\big)-\epsilon^{2}\big(\partial_{t}\psi_{2}^{out}+\partial_{t}\psi_{*2}^{in}\big)-\big(\nabla_{y}\psi_{2}^{out}+\nabla_{y}\psi_{*2}^{in}\big)\cdot \epsilon\dot{\xi}_{j}\Bigg).
    \end{split}
    \end{equation*}Since \(w'(r)\sim r^{-3}\) as \(r\to\infty\), the terms \(\nabla w_{k}/w_{k}\cdot(-\epsilon\dot{\xi}_{k}+\epsilon\dot{\xi}_{j})\) have size \(O(\epsilon^{4})\) in this region. Moreover, using the fact that
    \begin{equation*}
    \psi_{2}^{out}(\epsilon y,\tau)=\epsilon(\sqrt{2})^{-1}\partial_{\tau}\psi_{1}^{out,1}+O(\epsilon^{4}\abs{\log\epsilon}^{M_{2}+5})\quad\text{for}\quad\abs{y_{j}}\leq\tfrac{1}{2}\delta\epsilon^{-1}
    \end{equation*}(see Lemma \ref{psiout1construction} and (\ref{psi2out2est})) and
    \begin{equation*}
        \partial_{\tau\tau}\psi_{1}^{out,1}=O(\epsilon^{2}\abs{\log\epsilon})\quad\text{for}\quad\abs{y_{j}}\leq\tfrac{1}{2}\delta\epsilon^{-1}
    \end{equation*}(as follows from an integral formula similar to (\ref{dtaupsi})), we find
    \begin{equation*}
        \epsilon^{2}\partial_{t}\psi_{2}^{out}+\nabla_{y}\psi_{2}^{out}\cdot(\epsilon\dot{\xi}_{j})=O(\epsilon^{4}\abs{\log\epsilon})\quad\text{for}\quad\abs{y_{j}}\leq\tfrac{1}{2}\delta\epsilon^{-1}.
    \end{equation*}
    The expression 
    \begin{equation*}
        \epsilon^{2}\partial_{t}\psi_{*2}^{in}+\nabla_{y}\psi^{in}_{*2}\cdot\epsilon\dot{\xi}_{j}
    \end{equation*}involves terms of the form
    \begin{equation*}
        \epsilon^{2}\partial_{t}\psi_{j2}+\nabla_{y}\psi_{j2}\cdot(-\epsilon\dot{\xi}_{j})+\nabla_{y}\psi_{j2}\cdot(\epsilon\dot{\xi}_{j})=\epsilon^{2}\partial_{t}\psi_{j2}
    \end{equation*}(plus lower order contributions from \(\psi_{j1}\)) and has size \(O(\epsilon^{4})\). We thus conclude
    \begin{equation*}
        \mu(y,t)=O(\epsilon^{4}\abs{\log\epsilon})\quad\text{for}\quad\abs{y_{j}}\leq\tfrac{1}{2}\delta\epsilon^{-1}.
    \end{equation*}

    In region (iii) expression (\ref{Adef}) reads \(A(y,t)=0\), and we check
    \begin{equation*}
        \mu(y,t)=\epsilon^{2}\partial_{t}(\rho^{2})=O(\epsilon^{4}\abs{\log\epsilon}^{2}).
    \end{equation*}

    It then remains to consider the intermediate region (ii). Here we have
    \begin{equation*}
    A(y,t)=\tilde{\eta}_{j}\bigg(\epsilon\dot{\xi}_{j}-\rho^{-2}\nabla^{\perp}\bigg((1-\tilde{\eta}_{j}^{(1/2)})\epsilon\dot{\xi}_{j}\cdot y_{j}^{\perp}\bigg)\bigg)
    \end{equation*}and
    \begin{equation*}
    \begin{split}
        \operatorname{div}(\rho^{2}A)&=\nabla\tilde{\eta}_{j}\cdot\bigg(\rho^{2}\epsilon\dot{\xi}_{j}-\nabla^{\perp}\big(\epsilon\dot{\xi}_{j}\cdot y_{j}^{\perp}\big)\bigg)+O(\epsilon^4)\\
        &=\nabla\tilde{\eta}_{j}\cdot\bigg((\rho^{2}-1)\epsilon\dot{\xi}_{j}\bigg)+O(\epsilon^{4})=O(\epsilon^{4}).
    \end{split}    
    \end{equation*}Combined with the estimate \(\epsilon^{2}\partial_{t}(\rho^{2})=O(\epsilon^{4}\abs{\log\epsilon})\) in this region, we find
    \begin{equation*}
        \mu(y,t)=O(\epsilon^{4}\abs{\log\epsilon})\quad\text{for}\quad\tfrac{1}{2}\delta\epsilon^{-1}\leq\abs{y_{j}}\leq 2\delta\epsilon^{-1}.
    \end{equation*}The proof of the lemma is complete.
    \end{proof}

    \begin{remark}\label{nablaAremark}
        For \(A(y,t)\) defined by (\ref{Adef}), direct verification shows \(\nabla A=0\) in the regions (i) and (iii), and \(\nabla A=O(\epsilon^{2})\) in region (ii).
    \end{remark}

    \textbf{Definition of \(B\).} With \(\mu(y,t)\) small thanks to our choice of \(A(y,t)\), we now seek a scalar (real-valued) function \(B(y,t)\) which approximately annihilates the quantity
    \begin{equation*}
        \lambda(y,t)=B-\epsilon^{2}\partial_{t}\varphi-\nabla\varphi\cdot A.
    \end{equation*} 
    
    The natural choice for this purpose is, of course, \(B(y,t)=\epsilon^{2}\partial_{t}\varphi+\nabla\varphi\cdot A\). We will modify this expression slightly to ensure \(B(y,t)\) is globally smooth in \(\mathbb{R}^{2}\times[0,T]\).

    \smallskip

    In the region \(\abs{y_{j}}\leq\tfrac{1}{2}\delta\epsilon^{-1}\) for some \(j\), we compute using (\ref{phaseustarexp}) and \(A=\epsilon\dot\xi_{j}(t)\):

    \begin{equation}\label{innerdtphasenablaphase}
    \begin{split}
        \epsilon^{2}\partial_{t}\varphi+\nabla\varphi\cdot A=&\:\sum_{k\neq j}d_{k}\nabla_{y}\theta_{k}\cdot\big(-\epsilon\dot{\xi}_{k}+\epsilon\dot{\xi}_{j}\big)\\[4pt]
        &+\big(\epsilon^{2}\partial_{t}\psi_{1}^{out}+\nabla_{y}\psi_{1}^{out}\cdot\epsilon\dot{\xi}_{j}\big)+\big(\epsilon^{2}\partial_{t}\psi_{*1}^{in}+\nabla_{y}\psi_{*1}^{in}\cdot\epsilon\dot{\xi}_{j}\big),\\[4pt]
    \end{split}
    \end{equation}where \(\nabla_{y}\theta_{k}=(y-\tilde{\xi}_{k})^{\perp}/\abs{y-\tilde{\xi}_{k}}^{2}\). The terms involving \(\nabla_{y}\theta_{k}\) (\(k\neq j\)) and \(\psi_{1}^{out}\) are smooth in this region, whereas \begin{equation*}\epsilon^{2}\partial_{t}\psi^{in}_{*1}+\nabla_{y}\psi_{*}^{in}\cdot\epsilon\dot{\xi}_{j}
    \end{equation*}
    involves terms of the form 
    \begin{equation*}
    \epsilon^{2}\partial_{t}\psi_{j1}+\nabla_{y}\psi_{j1}\cdot(-\epsilon\dot{\xi}_{j})+\nabla_{y}\psi_{j1}\cdot(\epsilon\dot{\xi}_{j})=\epsilon^{2}\partial_{t}\psi_{j1}
    \end{equation*}which are bounded but not necessarily smooth in a neighbourhood of \(y_{j}=0\).

   Now for \(\eta_{j}\) our smooth cut-off with
    \begin{equation*}
        \eta_{j}=1\quad\text{if}\quad\abs{y_{j}}\leq1,\quad\quad\eta_{j}=0\quad\text{if}\quad\abs{y_{j}}\geq2,
    \end{equation*}
    we define
    \begin{equation}\label{Bdef}
    B(y,t):=\epsilon^{2}\partial_{t}\varphi+\nabla\varphi\cdot A-\sum_{j=1}^{n}\eta_{j}\big(\epsilon^{2}\partial_{t}\psi^{in}_{*1}+\nabla_{y}\psi^{in}_{*1}\cdot\epsilon\dot{\xi}_{j}\big).
    \end{equation}

    Expression (\ref{Bdef}) corresponds to \(\epsilon^{2}\partial_{t}\varphi+\nabla\varphi\cdot A\) with singular terms removed, thus \(B(y,t)\) is smooth in \(\mathbb{R}^{2}\times[0,T]\). Moreover:

    \begin{lemma}
        For \(B(y,t)\) defined by (\ref{Bdef}), we have \(B(y,t)=O(\epsilon^{2})\) and 
        \begin{equation}\label{lambdaest}
            \lambda(y,t)=B-\epsilon^{2}\partial_{t}\varphi-\nabla\varphi\cdot A=O(\epsilon^{4})
        \end{equation}
        uniformly in \(\mathbb{R}^{2}\times[0,T]\).
    \end{lemma}

    \begin{proof}
        Using the estimates in Remarks \ref{derivestremark} and \ref{derivestremark2} and similar for \(\psi_{j}^{(k)}\) (\(k\geq2\)), we check that  \(\partial_{t}\psi_{j1}=O(\epsilon^{2})\) in the region where \(\eta_{j}\) is supported. Hence
        \begin{equation*}
            \lambda(y,t)=-\sum_{j=1}^{n}\eta_{j}\big(\epsilon^{2}\partial_{t}\psi_{*1}^{in}+\nabla_{y}\psi_{*1}^{in}\cdot\epsilon\dot{\xi}_{j}\big)=O(\epsilon^{4})
        \end{equation*}in \(\mathbb{R}^{2}\times[0,T]\). We observe that \(\lambda(y,t)=0\) if \(\abs{y_{j}}\geq2\) for all \(j\). 

        The expression \(\epsilon^{2}\partial_{t}\varphi+\nabla\varphi\cdot A\), on the other hand, has size \(O(\epsilon^{2})\) in \(\mathbb{R}^{2}\times[0,T]\), as follows from (\ref{innerdtphasenablaphase}) and direct verification in the region \(\abs{y_{j}}\geq\tfrac{1}{2}\delta\epsilon^{-1}\) for all \(j\). We deduce 
        \begin{equation*}
        B(y,t)=O(\epsilon^{2})-O(\epsilon^{4})=O(\epsilon^{2})
        \end{equation*}for all \((y,t)\in\mathbb{R}^{2}\times[0,T]\).
    \end{proof}

    \textbf{Another estimate.} The final result of this subsection records an estimate for the quantity \(\epsilon^{2}\partial_{t}A-2\nabla B\) appearing in (\ref{calR2def}). 

    \begin{lemma}\label{AnablaBlemma}
       For \(A(y,t)\) defined by (\ref{Adef}) and \(B(y,t)\) defined by (\ref{Bdef}), we have
        \begin{equation*}
            \epsilon^{2}\partial_{t}A-2\nabla B=O\left(\epsilon^{4}\abs{\log\epsilon}(1+\abs{y_{j}})\right),\quad\text{if}\quad\abs{y_{j}}\leq\tfrac{1}{2}\delta\epsilon^{-1}\text{ for some }j,
        \end{equation*}and
        \begin{equation}
            \epsilon^{2}\partial_{t}A-2\nabla B=O(\epsilon^{3}),\quad\text{if}\quad\abs{y_{j}}\geq\tfrac{1}{2}\delta\epsilon^{-1}\text{ for all }j.\label{dtAnablaBoutest}
        \end{equation}
    \end{lemma}

    \begin{proof}In the region \(\abs{y_{j}}\leq\tfrac{1}{2}\delta\epsilon^{-1}\) we compute
    \begin{align}
        \epsilon^{2}\partial_{t}A-2\nabla B&=\epsilon^{3}\ddot{\xi}_{j}-2\nabla_{y}\bigg(\sum_{k\neq j}d_{k}\nabla_{y}\theta_{k}\cdot(-\epsilon\dot{\xi}_{k}+\epsilon\dot{\xi}_{j})\bigg)+O(\epsilon^{4}\abs{\log\epsilon})\notag\\
        \begin{split}
        &=\epsilon^{2}\partial_{t}\bigg(\epsilon\dot{\xi}_{j}-2\sum_{k\neq j}d_{k}\nabla_{y}\theta_{k}\bigg)+\nabla_{y}\bigg(\epsilon\dot{\xi}_{j}-2\sum_{k\neq j}d_{k}\nabla_{y}\theta_{k}\bigg)\cdot(\epsilon\dot{\xi}_{j})\\[3pt]
     &\quad +O(\epsilon^{4}\abs{\log\epsilon}).
     \end{split}\label{dtAnablaBexp1}   
        \end{align}
        We have (as in the proof of Lemma \ref{refinederrorexp})        \begin{align}
            \epsilon\dot{\xi}_{j}-2\sum_{k\neq j}d_{k}\nabla_{y}\theta_{k}&=\epsilon\bigg(\dot{\xi}_{j}-2\sum_{k\neq j}d_{k}\frac{(\epsilon y_{j}+\xi_{j}-\xi_{k})^{\perp}}{\abs{\epsilon y_{j}+\xi_{j}-\xi_{k}}^{2}}\bigg)\notag\\[3pt]
            &=\epsilon\bigg(\dot{\xi}_{j}-2\sum_{k\neq j}d_{k}\frac{(\xi_{j}-\xi_{k})^{\perp}}{\abs{\xi_{j}-\xi_{k}}^{2}}\bigg)+\epsilon\mathsf{M}_{j}(t)[\epsilon y_{j}]+O(\epsilon^{3}\abs{y_{j}}^{2}) \label{ximinusphaseexp}   
        \end{align}for \(\abs{y_{j}}\leq\tfrac{1}{2}\delta\epsilon^{-1}\), where \(\mathsf{M}_{j}(t)\) denotes a smooth \(2\times2\) matrix with
        \begin{equation*}
            \abs{\mathsf{M}_{j}(t)}+\abs{\partial_{t}\mathsf{M}_{j}(t)}\leq C.
        \end{equation*}Since 
        \begin{equation*}
            \partial_{t}\bigg(\dot{\xi}_{j}-2\sum_{k\neq j}d_{k}\frac{(\xi_{j}-\xi_{k})^{\perp}}{\abs{\xi_{j}-\xi_{k}}^{2}}\bigg)=O(\epsilon\abs{\log\epsilon}),
        \end{equation*}we can use (\ref{ximinusphaseexp}) in (\ref{dtAnablaBexp1}) to deduce
        \begin{equation*}
        \begin{split}
            \epsilon^{2}\partial_{t}A-2\nabla B&=\epsilon^{4}\big(\partial_{t}\mathsf{M}_{j}(t)\big)[y_{j}]+\epsilon^{3}\mathsf{M}_{j}(t)[-\dot{\xi}_{j}]+\epsilon^{3}\mathsf{M}_{j}(t)[\dot{\xi}_{j}]\\[3pt]
            &\quad\:+O\big(\epsilon^{4}\abs{\log\epsilon}(1+\abs{y_{j}})\big)\\[3pt]            &=O(\epsilon^{4}\abs{\log\epsilon}(1+\abs{y_{j}})\big), \quad\text{for}\quad\abs{y_{j}}\leq\tfrac{1}{2}\delta\epsilon^{-1}.
            \end{split}    
        \end{equation*}
        
        The proof of the first claim is thus complete. In the region \(\abs{y_{j}}\geq\tfrac{1}{2}\delta\epsilon^{-1}\) for all \(j\), we directly verify that \(\partial_{t}A=O(\epsilon)\) and \(\nabla B=O(\epsilon^{3})\), thus (\ref{dtAnablaBoutest}) holds.
        \end{proof}

    To conclude this subsection we recall that the vanishing of \(\mu(y,t)\) and \(\lambda(y,t)\) is directly connected with \(u_{*}\) being a critical point for the functional \(\mathcal{F}(u)\) described in \(\S\)\ref{energyestimatesoutline}. In certain symmetric situations one can find \(A\), \(B\) so that the integrals (\ref{calR1def})-(\ref{calR6def}) are all equal to zero. For example, if \(\vec{c}\) is a constant vector and \(u_{*}\) is a travelling wave \(u_{*}(y,t)=v_{*}(y-\vec{c}t)\) the natural choice is \(A(y,t)=\vec{c}\) and \(B(y,t)=0\). A quadratic form similar to \(\mathcal{B}[\phi,\phi]\) obtained by linearization around multi-vortex travelling waves has been studied in \cite{chironpacherie2023}.
    
    \subsection{Coercivity of the quadratic form}\label{coercivitysubsect} From now on we consider the quadratic form \(\mathcal{B}[\phi,\phi]\) with \(A(y,t)\) given by (\ref{Adef}) and \(B(y,t)\) given by (\ref{Bdef}). The goal of this section is to prove that \(\mathcal{B}[\phi,\phi]\) is coercive, in a suitable sense, for functions satisfying finitely many orthogonality conditions. We start by recording an expression for \(\mathcal{B}[\phi,\phi]\) in terms of \(\psi=(iu_{*})^{-1}\phi\).

    \begin{lemma}\label{quadforminpsi}
        Let \(\phi=iu_{*}\psi\) be a smooth, compactly supported function vanishing in a neighbourhood of the points \(\tilde{\xi}_{j}(t)=\epsilon^{-1}\xi_{j}(t)\) for \(j=1,\ldots,n\). Then
         \begin{equation}\label{quadforminpsiexp}
    \begin{split}
        \mathcal{B}[\phi,\phi]
        &=\int_{\mathbb{R}^{2}}\rho^{2}\bigg(\abs{\nabla\psi}^{2}+2(A-2\nabla\varphi)\cdot(\nabla\psi_{1})\psi_{2}+2\rho^{2}\abs{\psi_{2}}^{2}\bigg)\\
        &\phantom{=}+\int_{\mathbb{R}^{2}}\bigg(\mu-2\rho^{2}\mathcal{S}_{2}^{*}\bigg)\psi_{1}\psi_{2}+\int_{\mathbb{R}^{2}}\bigg(\lambda-\mathcal{S}_{1}^{*}\bigg)\rho^{2}\abs{\psi}^{2},\\
    \end{split}
    \end{equation}where \(\psi=\psi_{1}+i\psi_{2}\).    
    \end{lemma}
    
    \begin{proof}
        By direct computation and integration by parts we find
        \begin{equation*}
        \begin{split}
            \mathcal{B}[iu_{*}\psi,iu_{*}\psi]&=\int_{\mathbb{R}^{2}}\rho^{2}\bigg(\abs{\nabla\psi}^{2}+\operatorname{Re}\big(i(A-2\nabla\varphi)\cdot(\nabla\psi)\overline{\psi}\big)+2\rho^{2}\abs{\psi_{2}}^{2}\bigg)\\
            &\phantom{=}+\int_{\mathbb{R}^{2}}\bigg(\abs{\nabla\varphi}^{2}-\frac{\Delta\rho}{\rho}-(1-\rho^{2})-A\cdot\nabla\varphi+B\bigg)\rho^{2}\abs{\psi}^{2}.
        \end{split}    
        \end{equation*}Since
            \begin{equation*}
            \mathcal{S}_{1}^{*}=\operatorname{Re}\bigg(\frac{S(u_{*})}{u_{*}}\bigg)=-\epsilon^{2}\partial_{t}\varphi-\abs{\nabla\varphi}^{2}+\frac{\Delta\rho}{\rho}+(1-\rho^{2})
      \end{equation*}we have
      \begin{equation*}
        \abs{\nabla\varphi}^{2}-\frac{\Delta\rho}{\rho}-(1-\rho^{2})-A\cdot\nabla\varphi+B=-\mathcal{S}_{1}^{*}+\lambda.
      \end{equation*}Moreover,
      \begin{equation*}
      \begin{split}
          \operatorname{Re}\int_{\mathbb{R}^{2}}\rho^{2}i(A-2\nabla\varphi)\cdot(\nabla\psi)\overline{\psi}&=2\int_{\mathbb{R}^{2}}\rho^{2}(A-2\nabla\varphi)\cdot(\nabla\psi_{1})\psi_{2}\\
          &\phantom{=}+\int_{\mathbb{R}^{2}}\operatorname{div}\big(\rho^{2}(A-2\nabla\varphi)\big)\psi_{1}\psi_{2}\\
      \end{split}    
      \end{equation*}where
      \begin{equation*}
          \operatorname{div}\big(\rho^{2}(A-2\nabla\varphi)\big)=\bigg(\epsilon^{2}\partial_{t}(\rho^{2})+\operatorname{div}(\rho^{2}A)\bigg)-\bigg(\epsilon^{2}\partial_{t}(\rho^{2})+\operatorname{div}(\rho^{2}2\nabla\varphi)\bigg).
      \end{equation*}
      The first term in parentheses above is equal to \(\mu\), while the second is equal to \(2\rho^{2}\mathcal{S}_{2}^{*}=2\rho^{2}\operatorname{Im}\big(\frac{S(u_{*})}{u_{*}}\big)\). The proof of the lemma is complete.
    \end{proof}

    For \(\phi=iu_{*}\psi\) supported close to one of the points \(\tilde{\xi}_{j}(t)=\epsilon^{-1}\xi_{j}(t)\), expression (\ref{quadforminpsiexp}) corresponds (up to conjugation and translation) to a small perturbation of the quadratic form (\ref{basicquadform}) discussed in \(\S\)\ref{onevortexqfsubsect}. Indeed, suppose that \(\phi=iu_{*}\psi\) satisfies the assumptions of Lemma \ref{quadforminpsi} with \(\psi\) supported in the region \(\abs{y_{j}}\leq\tfrac{1}{4}\delta\epsilon^{-1}\) for some \(j\). In this region we have \(\rho^{2}\approx \abs{W_{j}}^{2}\) and
    \begin{equation*}
        A-2\nabla\varphi=\epsilon\dot{\xi}_{j}-2\nabla\varphi=-2d_{j}\nabla\theta_{j}+O\big(\epsilon^{2}(\abs{y_{j}}+\abs{y_{j}}^{-1})\big),
    \end{equation*}thus (recalling (\ref{muest}) and (\ref{lambdaest}))
    \begin{equation*}
        \mathcal{B}[iu_{*}\psi,iu_{*}\psi]\approx\int_{\mathbb{R}^{2}}\abs{W_{j}}^{2}\bigg(\abs{\nabla\psi}^{2}-4d_{j}\nabla\theta_{j}\cdot(\nabla\psi_{1})\psi_{2}+2\abs{W_{j}}^{2}\abs{\psi_{2}}^{2}\bigg)=\mathbb{B}_{W_{j}}[\psi,\psi]   
    \end{equation*}where 
    \begin{equation*}
        \mathbb{B}_{W_{j}}[\psi,\psi]=\mathcal{B}_{W_{j}}[iW_{j}\psi,iW_{j}\psi]
    \end{equation*}and
    \begin{equation}\label{BWjdef}
        \mathcal{B}_{W_{j}}[\phi,\phi]:=\int_{\mathbb{R}^{2}}\abs{\nabla\phi}^{2}-(1-\abs{W_{j}}^{2})\abs{\phi}^{2}+2\operatorname{Re}(\overline{W}_{j}\phi)^{2}.
    \end{equation}

    Let \(r_{j}:=\abs{y_{j}}\), \(\theta_{j}:=\theta(y_{j})\), \(j=1,\ldots,n\), and recall that \(w(r)\) denotes the modulus of the degree-one vortex. It follows from the discussion in \(\S\)\ref{onevortexqfsubsect} and Remark \ref{conjugatesymmetry} that \(\mathbb{B}_{W_{j}}[\psi,\psi]\) has a two-dimensional kernel in its natural energy space, spanned by the functions \(\psi=(iu_{*})^{-1}\mathcal{Z}_{j1}\) and \(\psi=(iu_{*})^{-1}\mathcal{Z}_{j2}\) where
    \begin{align}
        \mathcal{Z}_{j1}&:=iu_{*}\bigg(\frac{d_{j}}{r_{j}}\sin\theta_{j}+i\frac{w'(r_{j})}{w(r_{j})}\cos\theta_{j}\bigg),\\[4pt]
        \mathcal{Z}_{j2}&:=iu_{*}\bigg(\frac{d_{j}}{r_{j}}\cos\theta_{j}-i\frac{w'(r_{j})}{w(r_{j})}\sin\theta_{j}\bigg).
    \end{align}
   Defining
    \begin{align}
        \mathcal{Z}_{j1}^{*}&:=\phantom{-}i\mathcal{Z}_{j2}=iu_{*}\bigg(\frac{w'(r_{j})}{w(r_{j})}\sin\theta_{j}+i\frac{d_{j}}{r_{j}}\cos\theta_{j}\bigg),\label{calZj1def}\\[4pt]
        \mathcal{Z}_{j2}^{*}&:=-i\mathcal{Z}_{j1}=iu_{*}\bigg(\frac{w'(r_{j})}{w(r_{j})}\cos\theta_{j}-i\frac{d_{j}}{r_{j}}\sin\theta_{j}\bigg),\label{calZj2def}
    \end{align}
    and \(\hat{\eta}_{j}\) a new smooth cut-off such that
    \begin{equation*}
    \hat{\eta}_{j}(y,t):=\begin{cases}
   1,\quad\text{if}\quad\abs{y-\tilde{\xi}_{j}}\leq\frac{1}{8}\delta\epsilon^{-1},\\
   0,\quad\text{if}\quad\abs{y-\tilde{\xi}_{j}}\geq\tfrac{1}{4}\delta\epsilon^{-1},
   \end{cases}
\end{equation*}we have  
    \begin{equation}
    \operatorname{Re}\int_{\mathbb{R}^{2}}\hat{\eta}_{j}\mathcal{Z}_{j1}\overline{\hat{\eta}_{j}\mathcal{Z}_{j2}^{*}}=\operatorname{Re}\int_{\mathbb{R}^{2}}\hat{\eta}_{j}\mathcal{Z}_{j2}\overline{\hat{\eta}_{j}\mathcal{Z}_{j1}^{*}}=0,\quad\quad j=1,\ldots,n,\label{zjlorthog1}
    \end{equation}and
    \begin{equation}
    \operatorname{Re}\int_{\mathbb{R}^{2}}\hat{\eta}_{j}\mathcal{Z}_{j1}\overline{\hat{\eta}_{j}\mathcal{Z}_{j1}^{*}}=\operatorname{Re}\int_{\mathbb{R}^{2}}\hat{\eta}_{j}\mathcal{Z}_{j2}\overline{\hat{\eta}_{j}\mathcal{Z}_{j2}^{*}}=d_{j}\beta_{j},\quad\quad j=1,\ldots,n,\label{zjlorthog2}
    \end{equation}where \(\beta_{j}>0\) is bounded above and below by positive constants independent of \(\epsilon\).

    \medskip

    Returning now to \(\mathcal{B}[\phi,\phi]\), let us consider functions \(\phi(y)\) satisfying the orthogonality conditions
    \begin{equation}\label{quadformorthogconds}
        \operatorname{Re}\int_{\mathbb{R}^{2}}\phi\:\overline{\hat{\eta}_{j}\mathcal{Z}_{jl}^{*}}=0,\quad\text{for all}\quad j=1,\ldots,n,\quad l=1,2.
    \end{equation} 
    
    We write \(\phi=iu_{*}\psi\) as usual, and decompose \(\psi(y)\) according to its Fourier modes with respect to \(y_{j}=(r_{j}\cos\theta_{j},r_{j}\sin\theta_{j})\) in the region where \(\hat{\eta}_{j}\) is supported. Let \(\psi_{j}^{0}\) denote the mode zero component of \(\psi(y)\) close to \(\tilde{\xi}_{j}(t)\) and \(\psi_{j}^{\scriptscriptstyle\geq1}\) denote the mode \(\geq1\) components of \(\psi(y)\) close to \(\tilde{\xi}_{j}(t)\). More precisely, let
    \begin{equation}\label{psijmode0}
        \psi_{j}^{0}(z):=\big(P^{0}\psi(\cdot+\tilde{\xi}_{j})\big)(z)
    \end{equation}and
     \begin{equation}\label{psijmodegeq1}
        \psi_{j}^{\scriptscriptstyle\geq1}(z):=\sum_{k=1}^{\infty}\big(P_{k}^{1}\psi(\cdot+\tilde{\xi}_{j})\big)(z)+\sum_{k=1}^{\infty}\big(P_{k}^{2}\psi(\cdot+\tilde{\xi}_{j})\big)(z)
    \end{equation}for \(\abs{z}\leq\tfrac{1}{4}\delta\epsilon^{-1}\), where \(\psi(\cdot+\tilde{\xi}_{j})(z):=\psi(z+\tilde{\xi}_{j})\) and \(P^{0}\), \(P_{k}^{1}\), \(P_{k}^{2}\) denote the Fourier projection operators introduced in (\ref{fourierprojop}). We write
    \begin{equation*}
    \phi_{j}^{\scriptscriptstyle\geq1}(z):=iW_{j}(z+\tilde{\xi}_{j})\psi_{j}^{\scriptscriptstyle\geq1}(z).    
    \end{equation*}

    The following result holds.

    \begin{proposition}\label{basiccoerciveest}
        There exist constants \(C_{1}\), \(C_{2}>0\) (independent of \(\epsilon>0\)) such that for any \(\phi(y)\) satisfying (\ref{quadformorthogconds}) with \(\phi(y)\) smooth and compactly supported away from \(\tilde{\xi}_{j}(t)\), \(j=1,\ldots,n\), we have
    \begin{equation}\label{basicquadformloweresteq}
        \mathcal{B}[\phi,\phi]+C_{2}\epsilon^{4}\norm{\phi}_{L^{2}_{y}}^{2}\geq\frac{C_{1}}{\abs{\log\epsilon}}\Bigg(\sum_{j=1}^{n}\int_{\abs{y_{j}}\leq\tfrac{1}{4}\delta\epsilon^{-1}}\frac{\abs{\phi_{j}^{\scriptscriptstyle\geq1}(y_{j})}^{2}}{1+\abs{y_{j}}^{2}}\Bigg).
    \end{equation}
    \end{proposition}

    \begin{proof}
    Since \(\lambda=O(\epsilon^{4})\) we can find a constant \(C_{2}>0\) independent of \(\epsilon>0\) such that
    \begin{equation*}
        C_{2}\epsilon^{4}>2\sup_{\mathbb{R}^{2}\times[0,T]}\abs{\lambda}.
    \end{equation*}

      We now claim that there exists a constant \(C_{1}>0\), independent of \(\epsilon>0\), such that for \(C_{2}>0\) as above and \(\phi(y)\) as in the statement of the proposition, estimate (\ref{basicquadformloweresteq}) holds. We argue by contradiction. If the claim does not hold there exists a sequence \(\epsilon_{k}\to0\) and smooth functions \(\phi^{k}(y)\), \(k=1,2,\ldots\) supported away from \(\tilde{\xi}_{j}(t)\), \(j=1,\ldots,n\) such that
        \begin{gather}
            \operatorname{Re}\int_{\mathbb{R}^{2}}\phi^{k}\:\overline{\hat{\eta}_{j}\mathcal{Z}_{jl}^{*}}=0,\quad\quad j=1,\ldots,n,\quad l=1,2,\label{seqorthogconds}\\[4pt]
            \abs{\log\epsilon_{k}}\bigg(\mathcal{B}[\phi^{k},\phi^{k}]+C_{2}\epsilon_{k}^{4}\norm{\phi^{k}}_{L^{2}_{y}}^{2}\bigg)<\frac{1}{k},\quad\text{for all}\quad k\geq1,\label{quadformto0}\\[4pt]
            \sum_{j=1}^{n}\int_{\abs{y_{j}}\leq\tfrac{1}{4}\delta\epsilon_{k}^{-1}}\frac{\abs{(\phi^{k})_{j}^{\scriptscriptstyle\geq1}}^{2}}{1+\abs{y_{j}}^{2}}=1  ,\quad\text{for all}\quad k\geq 1.  \label{normequal1}    \end{gather}

   \textbf{Step 1: Inner-outer decomposition.} Rewriting (\ref{quadformto0}) in terms of \(\psi^{k}=(iu_{*})^{-1}\phi^{k}\) and splitting the integrals over \(\mathbb{R}^{2}\) into the regions
    \begin{gather*}
        \mathcal{D}_{j,k}:=\{y:\abs{y_{j}}\leq\tfrac{1}{4}\delta\epsilon_{k}^{-1}\},\\
        \mathcal{D}_{k}^{out}:=\{y:\abs{y_{j}}>\tfrac{1}{4}\delta\epsilon_{k}^{-1}\quad\text{for all } j\},
    \end{gather*}we have (using (\ref{quadforminpsiexp}))
    \begin{equation*}
        \abs{\log\epsilon_{k}}\bigg(\sum_{j=1}^{n}\mathbb{B}_{j}^{in}[\psi^{k},\psi^{k}]+\mathbb{B}^{out}[\psi^{k},\psi^{k}]\bigg)<\frac{1}{k}
    \end{equation*}where
     \begin{align*}
    \begin{split}
    \mathbb{B}_{j}^{in}[\psi^{k},\psi^{k}]&:=\int_{\mathcal{D}_{j,k}}\rho^{2}\bigg(\abs{\nabla\psi^{k}}^{2}+2(A-2\nabla\varphi)\cdot(\nabla\psi_{1}^{k})\psi_{2}^{k}+2\rho^{2}\abs{\psi_{2}^{k}}^{2}\bigg)\\
     &\phantom{=}+\int_{\mathcal{D}_{j,k}}\bigg(\mu-2\rho^{2}\mathcal{S}_{2}^{*}\bigg)\psi_{1}^{k}\psi_{2}^{k}+\int_{\mathcal{D}_{j,k}}\bigg(C_{2}\epsilon_{k}^{4}+(\lambda-\mathcal{S}_{1}^{*})\bigg)\rho^{2}\abs{\psi^{k}}^{2},
    \end{split}\\    
    \begin{split}
     \mathbb{B}^{out}[\psi^{k},\psi^{k}]&:=\int_{\mathcal{D}_{k}^{out}}\rho^{2}\bigg(\abs{\nabla\psi^{k}}^{2}+2(A-2\nabla\varphi)\cdot(\nabla\psi_{1}^{k})\psi_{2}^{k}+2\rho^{2}\abs{\psi_{2}^{k}}^{2}\bigg)\\
     &\phantom{=}+\int_{\mathcal{D}_{k}^{out}}\bigg(\mu-2\rho^{2}\mathcal{S}_{2}^{*}\bigg)\psi_{1}^{k}\psi_{2}^{k}+\int_{\mathcal{D}_{k}^{out}}\bigg(C_{2}\epsilon_{k}^{4}+(\lambda-\mathcal{S}_{1}^{*})\bigg)\rho^{2}\abs{\psi^{k}}^{2},   
    \end{split}
    \end{align*}and \(\psi^{k}=\psi^{k}_{1}+i\psi^{k}_{2}\). 

    Now using the estimates \(1-\rho^{2}=O(\epsilon_{k}^{2}\abs{\log\epsilon_{k}})\), \(A-2\nabla\varphi=O(\epsilon_{k})\) and \(\mu=O(\epsilon_{k}^{4}\abs{\log\epsilon_{k}}^{2})\) in the region \(\mathcal{D}_{k}^{out}\), we check that
    \begin{equation*}
            \mathbb{B}^{out}[\psi^{k},\psi^{k}]\geq C\int_{
            \mathcal{D}_{k}^{out}}\rho^{2}\bigg(\abs{\nabla\psi^{k}}^{2}+\abs{\psi_{2}^{k}}^{2}+\epsilon_{k}^{4}\abs{\psi^{k}_{1}}^{2}\bigg).
    \end{equation*} 
    
    Concerning \(\mathbb{B}_{j}^{in}[\psi^{k},\psi^{k}]\), let us decompose
    \begin{equation*}
        \mathbb{B}_{j}^{in}[\psi^{k},\psi^{k}]=\widetilde{\mathbb{B}}_{j}^{in}[\psi^{k},\psi^{k}]+\mathbb{B}_{j}^{in,\#}[\psi^{k},\psi^{k}]
    \end{equation*}where
    \begin{equation}\label{bbBtildejin}
        \widetilde{\mathbb{B}}_{j}^{in}[\psi^{k},\psi^{k}]:=\int_{\mathcal{D}_{j,k}}\abs{W_{j}}^{2}\bigg(\abs{\nabla\psi^{k}}^{2}-4d_{j}\nabla\theta_{j}\cdot(\nabla\psi_{1}^{k})\psi_{2}^{k}+2\abs{W_{j}}^{2}\abs{\psi_{2}^{k}}^{2}\bigg)
    \end{equation}and
    \begin{equation*}
        \mathbb{B}_{j}^{in,\#}[\psi^{k},\psi^{k}]:=\mathbb{B}_{j}^{in}[\psi^{k},\psi^{k}]-\widetilde{\mathbb{B}}_{j}^{in}[\psi^{k},\psi^{k}].
    \end{equation*}
    
    Since \(\abs{W_{j}}^{2}\) is radial with respect to \(y_{j}\) and \(\nabla\theta_{j}\cdot\nabla\psi=r_{j}^{-2}\partial_{\theta_{j}}\psi\) the quadratic form (\ref{bbBtildejin}) separates Fourier modes. Thus using the notation introduced in (\ref{psijmode0})-(\ref{psijmodegeq1}) with
        \begin{equation*}
            \psi_{j}^{0,k}:=(\psi^{k})_{j}^{0},\quad\quad\psi_{j}^{\scriptscriptstyle\geq1,k}:=(\psi^{k})_{j}^{\scriptscriptstyle\geq1},
        \end{equation*}we have
        \begin{equation*}
            \widetilde{\mathbb{B}}_{j}^{in}[\psi^{k},\psi^{k}]=\widetilde{\mathbb{B}}_{j}^{in}[\psi_{j}^{0,k},\psi_{j}^{0,k}]+\widetilde{\mathbb{B}}_{j}^{in}[\psi_{j}^{\scriptscriptstyle\geq1,k},\psi_{j}^{\scriptscriptstyle\geq1,k}]
        \end{equation*}where
        \begin{equation*}
           \widetilde{\mathbb{B}}_{j}^{in}[\psi_{j}^{0,k},\psi_{j}^{0,k}]=\int_{\mathcal{D}_{j,k}}\abs{W_{j}}^{2}\bigg(\abs{\nabla\psi_{j}^{0,k}}^{2}+2\abs{W_{j}}^{2}\abs{\psi_{j2}^{0,k}}^{2}\bigg) 
        \end{equation*}since \(\partial_{\theta_{j}}\psi_{j}^{0,k}=0\). Writing \(\phi_{j}^{\scriptscriptstyle\geq1,k}=iW_{j}\psi_{j}^{\scriptscriptstyle\geq1,k}\), we also have
        \begin{equation*}
        \widetilde{\mathbb{B}}_{j}^{in}[\psi_{j}^{\scriptscriptstyle\geq1,k},\psi_{j}^{\scriptscriptstyle\geq1,k}]=\widetilde{\mathcal{B}}_{j}^{in}[\phi_{j}^{\scriptscriptstyle\geq1,k},\phi_{j}^{\scriptscriptstyle\geq1,k}]-\int_{\mathcal{D}_{j,k}}\tfrac{1}{2}\operatorname{div}\big(\abs{\psi_{j}^{\scriptscriptstyle\geq1,k}}^{2}\nabla(\abs{W_{j}}^{2})\big)
        \end{equation*}where
        \begin{equation*}
            \widetilde{\mathcal{B}}_{j}^{in}[\phi_{j}^{\scriptscriptstyle\geq1,k},\phi_{j}^{\scriptscriptstyle\geq1,k}]:=\int_{\mathcal{D}_{j,k}}\abs{\nabla\phi_{j}^{\scriptscriptstyle\geq1,k}}^{2}-(1-\abs{W_{j}}^{2})\abs{\phi_{j}^{\scriptscriptstyle\geq1,k}}^{2}+2\operatorname{Re}(\overline{W}_{j}\phi_{j}^{\scriptscriptstyle\geq1,k})^{2}.
        \end{equation*}

        Now turning to the ``remainder'' quadratic form \(\mathbb{B}_{j}^{in,\#}[\psi^{k},\psi^{k}]\), the estimates \(\rho^{2}-\abs{W_{j}}^{2}=O(\epsilon_{k}^{2})\) and
        \begin{equation*}
             A-2\nabla\varphi=-2d_{j}\nabla\theta_{j}+O\big(\epsilon_{k}^{2}(\abs{y_{j}}+\abs{y_{j}}^{-1})\big)
        \end{equation*}
        in \(\mathcal{D}_{j,k}\) combined with (\ref{muest}), (\ref{lambdaest}) etc. imply that
        \begin{equation*}
        \begin{gathered}
            \mathbb{B}_{j}^{in,\#}[\psi^{k},\psi^{k}]-\int_{\mathcal{D}_{j,k}}\tfrac{1}{2}\operatorname{div}\big(\abs{\psi_{j}^{\scriptscriptstyle\geq1,k}}^{2}\nabla(\abs{W_{j}}^{2})\big)\\
            \geq C\epsilon_{k}^{4}\int_{\mathcal{D}_{j,k}}\rho^{2}\abs{\psi^{k}}^{2}-O(\epsilon_{k})\bigg(\widetilde{\mathbb{B}}_{j}^{in}[\psi_{j}^{0,k},\psi_{j}^{0,k}]+\norm{\phi_{j}^{\scriptscriptstyle\geq1,k}}_{\mathcal{H}_{j,k}}^{2}\bigg)   
        \end{gathered}    
        \end{equation*}where
        \begin{equation*}
            \norm{\phi_{j}^{\scriptscriptstyle\geq1,k}}_{\mathcal{H}_{j,k}}^{2}:=\int_{\abs{y_{j}}\leq\tfrac{1}{4}\delta\epsilon_{k}^{-1}}\abs{\nabla\phi_{j}^{\scriptscriptstyle\geq1,k}}^{2}+\frac{\abs{\phi_{j}^{\scriptscriptstyle\geq1,k}}^{2}}{(1+\abs{y_{j}}^{2})}+\operatorname{Re}(\overline{W}_{j}\phi_{j}^{\scriptscriptstyle\geq1,k})^{2}.
        \end{equation*}
        
        The observations made so far thus combine to show
        \begin{equation}\label{innerouterqfsto0}
        \begin{gathered}
        \abs{\log\epsilon_{k}}\Bigg(\sum_{j=1}^{n}\big(\widetilde{\mathbb{B}}_{j}^{in}[\psi_{j}^{0,k},\psi_{j}^{0,k}]+\widetilde{\mathcal{B}}_{j}^{in}[\phi_{j}^{\scriptscriptstyle\geq1,k},\phi_{j}^{\scriptscriptstyle\geq1,k}]\big)+\mathbb{B}^{out}[\psi^{k},\psi^{k}]\Bigg)\\
            <\frac{1}{k}+O(\epsilon_{k}\abs{\log\epsilon_{k}})\Bigg(\sum_{j=1}^{n}\widetilde{\mathbb{B}}_{j}^{in}[\psi_{j}^{0,k},\psi_{j}^{0,k}]+\norm{\phi_{j}^{\scriptscriptstyle\geq1,k}}_{\mathcal{H}_{j,k}}^{2}\Bigg).
        \end{gathered}    
        \end{equation}Since \(\widetilde{\mathbb{B}}_{j}^{in}[\psi_{j}^{0,k},\psi_{j}^{0,k}]\) and \(\mathbb{B}^{out}[\psi^{k},\psi^{k}]\) are nonnegative and 
        \begin{align}
            \norm{\phi_{j}^{\scriptscriptstyle\geq1,k}}_{\mathcal{H}_{j,k}}^{2}&\leq\widetilde{\mathcal{B}}_{j}^{in}[\phi_{j}^{\scriptscriptstyle\geq1,k},\phi_{j}^{\scriptscriptstyle\geq1,k}]+C\int_{\abs{y_{j}}\leq\tfrac{1}{4}\delta\epsilon_{k}^{-1}}\frac{\abs{\phi_{j}^{\scriptscriptstyle\geq1,k}}^{2}}{(1+\abs{y_{j}}^{2})}\label{HjklessBjinest}\\
            &\leq\widetilde{\mathcal{B}}_{j}^{in}[\phi_{j}^{\scriptscriptstyle\geq1,k},\phi_{j}^{\scriptscriptstyle\geq1,k}]+C  \notag  
        \end{align}using (\ref{normequal1}), we deduce from (\ref{innerouterqfsto0}) the bound
        \begin{equation}\label{hjkuniformbound}
            \sum_{j=1}^{n}\norm{\phi_{j}^{\scriptscriptstyle\geq1,k}}_{\mathcal{H}_{j,k}}^{2}\leq C
        \end{equation}for some constant \(C>0\) independent of \(k\). Moreover,
        \begin{equation}\label{innerqfsto0}
            \abs{\log\epsilon_{k}}\sum_{j=1}^{n}\widetilde{\mathcal{B}}_{j}^{in}[\phi_{j}^{\scriptscriptstyle\geq1,k},\phi_{j}^{\scriptscriptstyle\geq1,k}]<\frac{1}{k}+O(\epsilon_{k}\abs{\log\epsilon_{k}}).
        \end{equation}

        \smallskip

        \textbf{Step 2: Convergence of \(\phi_{j}^{\geq1,k}\).} In view of (\ref{hjkuniformbound}) and weak compactness, there exist functions \(\phi_{1}^{lim},\ldots,\phi_{n}^{lim}\) defined on \(\mathbb{R}^{2}\) such that (after passing to a subsequence) \(\phi_{j}^{\scriptscriptstyle\geq1,k}\rightharpoonup\phi_{j}^{lim}\) locally weakly in \(H^{1}\) and \(\phi_{j}^{\scriptscriptstyle\geq1,k}\to\phi_{j}^{lim}\) locally strongly in \(L^{2}\), for all \(j=1,\ldots,n\). Moreover, \(\phi_{j}^{lim}(z)\) contains only modes \(\geq1\) in its Fourier expansion and satisfies \(\norm{\phi_{j}^{lim}}_{\mathcal{H}_{j}}<\infty\) where
        \begin{equation*}
            \norm{\phi}_{\mathcal{H}_{j}}^{2}:=\int_{\mathbb{R}^{2}}\abs{\nabla\phi}^{2}+\frac{\abs{\phi}^{2}}{(1+\abs{z}^{2})}+\operatorname{Re}(\overline{W}_{j}\phi)^{2}.
        \end{equation*}
        
        Here with slight abuse of notation, we have written \(W_{j}=W(z)\) if \(j\in I_{+}\) and \(W_{j}=\overline{W}(z)\) if \(j\in I_{-}\). 

        Using (\ref{innerqfsto0}), strong \(L^{2}\)-convergence over compacts and weak lower semicontinuity of \(H^{1}\)-norms we deduce that the (nonnegative) one-vortex quadratic form (\ref{BWjdef}) evaluated at \(\phi_{j}^{lim}\) satisfies
        \begin{equation*}
            \mathcal{B}_{W_{j}}[\phi_{j}^{lim},\phi_{j}^{lim}]=0,\quad\text{for all}\quad j=1,\ldots,n.
        \end{equation*}
        By \cite{delpinofelmerkowalczyk2004}*{Theorem 1.1}, we then conclude
        \begin{equation}\label{phijlimcomb}
            \phi_{j}^{lim}(z)=\mathsf{c}_{j1}\partial_{1}W_{j}+\mathsf{c}_{j2}\partial_{2}W_{j}
        \end{equation}for real numbers \(\mathsf{c}_{j1}\), \(\mathsf{c}_{j2}\), where 
        \begin{equation*}
            \partial_{1}W_{j}=iW_{j}\bigg(-\frac{d_{j}}{r}\sin\theta-i\frac{w'}{w}\cos\theta\bigg), \quad\partial_{2}W_{j}=iW_{j}\bigg(\frac{d_{j}}{r}\cos\theta-i\frac{w'}{w}\sin\theta\bigg).
        \end{equation*}(Here \((r,\theta)\) denote polar coordinates with respect to \(z\)). 

    \medskip    

    \textbf{Step 3: Conclusion for fixed orthogonality conditions.} At this point we can quickly conclude the contradiction argument if (\ref{seqorthogconds}) is replaced by orthogonality conditions on a \textit{fixed ball} independent of \(\epsilon_{k}>0\). Indeed, let us suppose that 
    \begin{equation}\label{fixedballorthogk}
        \operatorname{Re}\int_{\mathbb{R}^{2}}\phi^{k}\:\overline{\eta_{j}\mathcal{Z}_{jl}^{*}}=0,\quad\quad j=1,\ldots,n,\quad l=1,2,
    \end{equation}where \(\eta_{j}\) is our smooth cut-off with \(\eta_{j}=1\) if \(\abs{y_{j}}\leq1\) and \(\eta_{j}=0\) if \(\abs{y_{j}}\geq2\). Since the support of \(\eta_{j}\) does not depend on \(\epsilon_{k}\) we can pass to the limit in (\ref{fixedballorthogk}) to deduce
    \begin{equation*}
     \operatorname{Re}\int_{\mathbb{R}^{2}}\phi_{j}^{lim}\:\overline{\eta_{j}(\partial_{\ell}W_{j})^{*}}=0,\quad\quad j=1,\ldots,n,\quad l=1,2,   
    \end{equation*}where \((\partial_{1}W_{j})^{*}:=i\partial_{2}W_{j}\) and \((\partial_{2}W_{j})^{*}:=i\partial_{1}W_{j}\). Equation (\ref{phijlimcomb}) and identities similar to (\ref{zjlorthog1})-(\ref{zjlorthog2}) then imply \(\phi_{j}^{lim}=0\) for \(j=1,\ldots,n\). By an argument analogous to that in \mbox{\cite{delpinofelmerkowalczyk2004}*{Lemma 3.1}} we can find \(R>0\) independent of \(\epsilon_{k}\) such that
    \begin{equation}\label{Rqfest}
       \norm{\phi_{j}^{\scriptscriptstyle\geq1,k}}_{\mathcal{H}_{j,k}}^{2}\leq\widetilde{\mathcal{B}}_{j}^{in}[\phi_{j}^{\scriptscriptstyle\geq1,k},\phi_{j}^{\scriptscriptstyle\geq1,k}]+C\int_{\abs{y_{j}}\leq R}\abs{\phi_{j}^{\scriptscriptstyle\geq1,k}}^{2}+O(\epsilon_{k}). 
    \end{equation}Combining (\ref{innerqfsto0}), (\ref{Rqfest}) and \(\phi_{j}^{\scriptscriptstyle\geq1,k}\to0\) in \(L^{2}_{loc}(\mathbb{R}^{2})\) we then conclude
    \begin{equation*}
        \sum_{j=1}^{n}\int_{\abs{y_{j}}\leq\tfrac{1}{4}\delta\epsilon_{k}^{-1}}\frac{\abs{\phi_{j}^{\scriptscriptstyle\geq1,k}}^{2}}{1+\abs{y_{j}}^{2}}\to0\quad\text{as}\quad k\to\infty,
    \end{equation*}which contradicts (\ref{normequal1}). 

    \smallskip

    The steps above do not require the \(\abs{\log\epsilon_{k}}\) factor in (\ref{quadformto0}). We have thus established the following fact: for \(\phi(y)\) as in the statement of Proposition \ref{basiccoerciveest} with (\ref{quadformorthogconds}) replaced by 
    \begin{equation}\label{fixedballorthogcond}
       \operatorname{Re}\int_{\mathbb{R}^{2}}\phi\:\overline{\eta_{j}\mathcal{Z}_{jl}^{*}}=0,\quad\text{for all}\quad j=1,\ldots,n,\quad l=1,2,    \end{equation}we have
    \begin{equation}
         \mathcal{B}[\phi,\phi]+C_{2}\epsilon^{4}\norm{\phi}_{L^{2}_{y}}^{2}\geq C_{1}\sum_{j=1}^{n}\int_{\abs{y_{j}}\leq\tfrac{1}{4}\delta\epsilon^{-1}}\frac{\abs{\phi_{j}^{\scriptscriptstyle\geq1}}^{2}}{1+\abs{y_{j}}^{2}}.
    \end{equation}

    \smallskip

    \textbf{Step 4: Conclusion of the main argument.} Returning now to the sequence \(\phi^{k}\) satisfying (\ref{seqorthogconds})-(\ref{normequal1}), we can decompose
    \begin{equation}\label{phiktildephikdecomp}
        \phi^{k}=\tilde{\phi}^{k}+\sum_{\substack{j=1,\ldots,n\\l=1,2}}\mathsf{c}_{jl}^{k}\:\hat\eta_{j}\mathcal{Z}_{jl}    \end{equation}where \(\tilde{\phi}^{k}\) satisfies the ``fixed ball'' orthogonality conditions (\ref{fixedballorthogcond}) for all \(k\), and \(\mathsf{c}_{jl}^{k}\) are real constants defined by
        \begin{equation*}
            \operatorname{Re}\int_{\mathbb{R}^{2}}\phi^{k}\:\overline{\eta_{j}\mathcal{Z}_{jl}^{*}}=\mathsf{c}_{jl}^{k}\:\operatorname{Re}\int_{\mathbb{R}^{2}}\hat\eta_{j}\mathcal{Z}_{jl}\:\overline{\eta_{j}\mathcal{Z}_{jl}^{*}}.
        \end{equation*}
        
        By the result proved in Step 3, we have
        \begin{equation*}
           \mathcal{B}[\tilde{\phi}^{k},\tilde{\phi}^{k}]+C_{2}\epsilon_{k}^{4}\norm{\tilde{\phi}^{k}}_{L_{y}^{2}}^{2}\geq C_{1}\sum_{j=1}^{n}\int_{\abs{y_{j}}\leq\tfrac{1}{4}\delta\epsilon_{k}^{-1}}\frac{\abs{\tilde{\phi}_{j}^{\scriptscriptstyle\geq1,k}}^{2}}{1+\abs{y_{j}}^{2}}.
        \end{equation*}Then using the decomposition of \(\mathcal{B}[\cdot,\cdot]+C_{2}\epsilon_{k}^{4}\norm{\cdot}_{L_{y}^{2}}^{2}\) employed in Step 1 combined with (\ref{HjklessBjinest}) we deduce
        \begin{equation}\label{coercivebtildephik}
            \mathcal{B}[\tilde{\phi}^{k},\tilde{\phi}^{k}]+C_{2}\epsilon_{k}^{4}\norm{\tilde{\phi}^{k}}_{L_{y}^{2}}^{2}\geq C\norm{\tilde{\phi}^{k}}_{\mathcal{H}_{\epsilon_k}}^{2}
        \end{equation}where
        \begin{equation*}
        \begin{split}
            \norm{\tilde{\phi}^{k}}_{\mathcal{H}_{\epsilon_{k}}}^{2}&:=\sum_{j=1}^{n}\widetilde{\mathbb{B}}_{j}^{in}[\tilde{\psi}_{j}^{0,k},\tilde{\psi}_{j}^{0,k}]+\norm{\tilde{\phi}_{j}^{\scriptscriptstyle\geq1,k}}_{\mathcal{H}_{j,k}}^{2} \\
            &\phantom{=}+\int_{\mathcal{D}_{k}^{out}}\rho^{2}\bigg(\abs{\nabla\tilde{\psi}^{k}}^{2}+\abs{\tilde{\psi}_{2}^{k}}^{2}\bigg)+\epsilon_{k}^{4}\int_{\mathbb{R}^{2}}\abs{\tilde{\phi}^{k}}^{2}.
            \end{split}    
        \end{equation*}

        In Lemma \ref{qfatzjllemma} below we estimate the projection of \(\mathcal{B}[\cdot,\cdot]\) onto the approximate kernel \(\hat\eta_{j}\mathcal{Z}_{jl}\), from which it follows
        \begin{equation*}
            \mathcal{B}[\hat\eta_{j}\mathcal{Z}_{jl},\hat\eta_{j}\mathcal{Z}_{jl}]=O(\epsilon_{k}^{2}),\quad\quad\mathcal{B}[\phi^{k},\hat{\eta}_{j}\mathcal{Z}_{jl}]=O(\epsilon_{k})\norm{\tilde{\phi}^{k}}_{\mathcal{H}_{\epsilon_{k}}}.
        \end{equation*}Using this result, (\ref{quadformto0}), (\ref{phiktildephikdecomp}) and (\ref{coercivebtildephik}), we deduce 
        \begin{equation}\label{tildephikboundconsts}
        \abs{\log\epsilon_{k}}\norm{\tilde{\phi}^{k}}_{\mathcal{H}_{\epsilon_{k}}}^{2}<\frac{C}{k}+O(\epsilon_{k}^{2}\abs{\log\epsilon_{k}})\sum_{\substack{j=1,\ldots,n\\l=1,2}}(\mathsf{c}_{jl}^{k})^{2}.
        \end{equation}

        To complete the proof it remains to estimate the constants \(\mathsf{c}_{jl}^{k}\). Testing equation (\ref{phiktildephikdecomp}) with \(\hat{\eta}_{j}\mathcal{Z}_{jl}^{*}\) and using the orthogonality conditions (\ref{seqorthogconds}) we find
        \begin{equation*}
            0=\operatorname{Re}\int_{\mathbb{R}^{2}}\tilde{\phi}^{k}\:\overline{\hat{\eta}_{j}\mathcal{Z}_{jl}^{*}}+d_{j}\beta_{j}\:\mathsf{c}_{jl}^{k}
        \end{equation*}where \(\beta_{j}>0\) is defined by (\ref{zjlorthog2}). Then for \(r_{j}=\abs{y_{j}}\) and \(w_{j}=w(r_{j})\), we have
        \begin{equation*}
        \begin{split}
            \abs{c_{jl}^{k}}&\leq\int_{r_{j}\leq\tfrac{1}{4}\delta\epsilon_{k}^{-1}}w_{j}^{2}\bigg(\tfrac{w'_{j}}{w_{j}}\abs{\tilde{\psi}_{j1}^{\scriptscriptstyle\geq1,k}}+\tfrac{1}{r_{j}}\abs{\tilde{\psi}_{j2}^{\scriptscriptstyle\geq1,k}}\bigg)+O(\epsilon_{k}^{2})\norm{\tilde{\phi}^{k}}_{\mathcal{H}_{\epsilon_{k}}}\\[5pt]
            &\leq C\bigg(\int_{0}^{\tfrac{1}{4}\delta\epsilon_{k}^{-1}}\frac{w^{2}}{r^{2}}r\:dr\bigg)^{1/2}\norm{\tilde{\phi}^{k}}_{\mathcal{H}_{\epsilon_{k}}}+C\norm{\tilde{\phi}^{k}}_{\mathcal{H}_{\epsilon_{k}}}\\[5pt]
            &\leq C\abs{\log\epsilon_{k}}^{1/2}\norm{\tilde{\phi}^{k}}_{\mathcal{H}_{\epsilon_{k}}}.
        \end{split}    
        \end{equation*}
        
        This estimate combined with (\ref{tildephikboundconsts}) implies \(\abs{\log\epsilon_{k}}\norm{\tilde{\phi}^{k}}_{\mathcal{H}_{\epsilon_{k}}}^{2}\to0\) and \(\mathsf{c}_{jl}^{k}\to0\). Thus
        \begin{equation*}
        \sum_{j=1}^{n}\int_{\abs{y_{j}}\leq\tfrac{1}{4}\delta\epsilon_{k}^{-1}}\frac{\abs{\phi_{j}^{\scriptscriptstyle\geq1,k}}^{2}}{1+\abs{y_{j}}^{2}}\to0\quad\text{as}\quad k\to\infty,
    \end{equation*}which contradicts (\ref{normequal1}).
    \end{proof}

    As observed in (\ref{coercivebtildephik}) above, the lower bound (\ref{basicquadformloweresteq}) leads to a full coercivity estimate for \(\mathcal{B}[\cdot,\cdot]\) using the decomposition in Step 1 of the previous proof. We have:

    \begin{corollary}\label{fullcoercivecor}
        Under the assumptions of Proposition \ref{basiccoerciveest},
        \begin{equation}
            \mathcal{B}[\phi,\phi]+C_{2}\epsilon^{4}\norm{\phi}_{L_{y}^{2}}^{2}\geq\frac{C_{1}}{\abs{\log\epsilon}}\norm{\phi}_{\mathcal{H}}^{2}
        \end{equation}where
        \begin{equation*}
        \begin{split}
            \norm{\phi}_{\mathcal{H}}^{2}&:=\sum_{j=1}^{n}\int_{\abs{y_{j}}\leq\tfrac{1}{4}\delta\epsilon^{-1}}\rho^{2}\bigg(\abs{\nabla\psi_{j}^{0}}^{2}+2\rho^{2}\abs{\psi_{j2}^{0}}^{2}\bigg)\\
            &\phantom{=}+\sum_{j=1}^{n}\int_{\abs{y_{j}}\leq\tfrac{1}{4}\delta\epsilon^{-1}}\abs{\nabla\phi_{j}^{\scriptscriptstyle\geq1}}^{2}+\frac{\abs{\phi_{j}^{\scriptscriptstyle\geq1}}^{2}}{(1+\abs{y_{j}}^{2})}+\operatorname{Re}(\overline{W}_{j}\phi_{j}^{\scriptscriptstyle\geq1})^{2}\\
            &\phantom{=}+\int_{\substack{\scriptscriptstyle\abs{y_{j}}>\tfrac{1}{4}\delta\epsilon^{-1}\\\text{for all }j}}\rho^{2}\bigg(\abs{\nabla\psi}^{2}+\abs{\psi_{2}}^{2}\bigg)+\epsilon^{4}\int_{\mathbb{R}^{2}}\abs{\phi}^{2}.
        \end{split}    
        \end{equation*}
        (Here \(\phi=iu_{*}\psi\) with \(\psi_{2}=\operatorname{Im}\psi\), and \(\psi_{j}^{0}\), \(\phi_{j}^{\scriptscriptstyle\geq1}=iW_{j}\psi_{j}^{\scriptscriptstyle\geq1}\) are defined via (\ref{psijmode0}), (\ref{psijmodegeq1}) with \(\psi_{j2}^{0}=\operatorname{Im}\psi_{j}^{0}\)).
    \end{corollary}
    
\subsection{Proof of the linear estimate}\label{linearestproofsubsect} The goal of this section is to prove Proposition \ref{mainlinearizedprop}. For a solution \(\phi(y,t)\) of (\ref{fulllinprobpensect}), we can write
\begin{equation}\label{phiperpdecomp}
    \phi(y,t)=\phi^{\perp}(y,t)+\sum_{\substack{j=1,\ldots,n\\l=1,2}}\mathsf{c}_{jl}(t)\hat{\eta}_{j}\mathcal{Z}_{jl}(y,t)
\end{equation}where \(\phi^{\perp}(y,t)\) satisfies the orthogonality conditions (\ref{quadformorthogconds}) for all \(t\in[0,T]\), and \(\mathsf{c}_{jl}(t)\) are time-dependent (real-valued) functions such that
\begin{equation*}
    \operatorname{Re}\int_{\mathbb{R}^{2}}\phi\:\overline{\hat{\eta}_{j}\mathcal{Z}_{jl}^{*}}=\mathsf{c}_{jl}(t)\operatorname{Re}\int_{\mathbb{R}^{2}}\hat{\eta}_{j}\mathcal{Z}_{jl}\:\overline{\hat{\eta}_{j}\mathcal{Z}_{jl}^{*}}
\end{equation*}for \(j=1,\ldots,n\), \(l=1,2\).

\smallskip

Our first objective is to find a lower bound for \(\mathcal{B}[\phi,\phi]\) in terms of \(\norm{\phi^{\perp}}_{\mathcal{H}}^{2}\). The following lemma records, for this purpose, the size of \(\mathcal{B}[\cdot,\cdot]\) on the subspace spanned by \(\hat{\eta}_{j}\mathcal{Z}_{jl}\). 

\begin{lemma}\label{qfatzjllemma}
    We have
    \begin{equation*}
        \mathcal{B}[\hat{\eta}_{j}\mathcal{Z}_{jl},\hat{\eta}_{j}\mathcal{Z}_{jl}]=O(\epsilon^{2})\quad\text{for all}\quad j=1,\ldots,n,\quad l=1,2.
    \end{equation*}Moreover, for \(\phi(y)\) smooth and compactly supported away from \(\tilde{\xi}_{j}(t)\), \(j=1,\ldots,n\) we have
    \begin{equation*}
        \abs{\mathcal{B}[\phi,\hat{\eta}_{j}\mathcal{Z}_{jl}]}\leq C\epsilon\norm{\phi}_{\mathcal{H}}\quad\text{for all}\quad j=1,\ldots,n,\quad l=1,2.
    \end{equation*}
\end{lemma}

\begin{proof}
Let \(\widetilde{\mathcal{Z}}_{jl}:=(iu_{*})^{-1}\mathcal{Z}_{jl}\). By Lemma \ref{quadforminpsi} and the decomposition    
\(
    \mathbb{B}_{j}^{in}[\cdot,\cdot]=\widetilde{\mathbb{B}}_{j}^{in}[\cdot,\cdot]+\mathbb{B}_{j}^{in,\#}[\cdot,\cdot]
\) used in the proof of Proposition \ref{basiccoerciveest}, we have
\begin{equation*}
    \mathcal{B}[\hat{\eta}_{j}\mathcal{Z}_{jl},\hat{\eta}_{j}\mathcal{Z}_{jl}]=\widetilde{\mathbb{B}}_{j}^{in}[\hat{\eta}_{j}\widetilde{\mathcal{Z}}_{jl},\hat{\eta}_{j}\widetilde{\mathcal{Z}}_{jl}]+O(\epsilon^{2})
\end{equation*}where
\begin{equation*}
    \widetilde{\mathbb{B}}_{j}^{in}[\psi,\psi]=\int_{\abs{y_{j}}\leq\tfrac{1}{4}\delta\epsilon^{-1}}\abs{W_{j}}^{2}\bigg(\abs{\nabla\psi}^{2}-2d_{j}\operatorname{Re}\big(i\nabla\theta_{j}\cdot(\nabla\psi)\overline{\psi}\big)+2\abs{W_{j}}^{2}\operatorname{Im}(\psi)^{2}\bigg).
\end{equation*}The identities
\begin{equation*}
    \abs{W_{j}}^{2}\abs{\nabla\psi}^{2}=\operatorname{div}\bigg(\abs{W_{j}}^{2}(\nabla\psi)\overline{\psi}\bigg)-\nabla\abs{W_{j}}^{2}\cdot(\nabla\psi)\overline{\psi}-\abs{W_{j}}^{2}(\Delta\psi)\overline{\psi}
\end{equation*}and
\begin{equation*}
    \Delta\widetilde{\mathcal{Z}}_{jl}+2\tfrac{\nabla\abs{W_{j}}}{\abs{W_{j}}}\cdot\nabla\widetilde{\mathcal{Z}}_{jl}+2id_{j}\nabla\theta_{j}\cdot\nabla\widetilde{\mathcal{Z}}_{jl}-2i\abs{W_{j}}^{2}\operatorname{Im}(\widetilde{\mathcal{Z}}_{jl})=0
\end{equation*}then imply
\begin{equation*}
\begin{split}
    \mathcal{B}[\hat{\eta}_{j}\mathcal{Z}_{jl},\hat{\eta}_{j}\mathcal{Z}_{jl}]&=\operatorname{Re}\int\abs{W_{j}}^{2}\bigg(-(\Delta\hat{\eta}_{j})\hat{\eta}_{j}\abs{\widetilde{\mathcal{Z}}_{jl}}^{2}-2(\nabla\hat{\eta}_{j}\cdot\nabla\widetilde{\mathcal{Z}}_{jl})(\hat{\eta}_{j}\widetilde{\mathcal{Z}}_{jl})\bigg)+O(\epsilon^{2})\\[4pt]
    &=O(\epsilon^{2}),
\end{split}    
\end{equation*}where the integral above is supported on the region \(\tfrac{1}{8}\delta\epsilon^{-1}\leq\abs{y_{j}}\leq\tfrac{1}{4}\delta\epsilon^{-1}\). 

\smallskip

The proof of the second claim follows using a similar argument.
\end{proof}

\begin{remark}\label{cauchyschwarzest}
    If \(\phi(y)\) satisfies the orthogonality conditions (\ref{quadformorthogconds}), we have in fact
    \begin{equation*}
     \abs{\mathcal{B}[\phi,\hat{\eta}_{j}\mathcal{Z}_{jl}]}\leq C\epsilon\big(\mathcal{B}[\phi,\phi]+C_{2}\epsilon^{4}\norm{\phi}_{L_{y}^{2}}^{2}\big)^{1/2}  
    \end{equation*} for all \(j=1,\ldots,n\), \(l=1,2\). This estimate can be deduced, for example, using the previous lemma and a decomposition of the form (\ref{phiktildephikdecomp}).
\end{remark}

In view of Corollary \ref{fullcoercivecor}, Lemma \ref{qfatzjllemma} and Remark \ref{cauchyschwarzest}, there exists a constant \(C_{3}>0\) independent of \(\epsilon\) such that for \(\phi(y,t)\) decomposed according to (\ref{phiperpdecomp}), we have
\begin{equation}\label{Qlowerbound}
\begin{gathered}
    \mathcal{Q}(t):=\mathcal{B}[\phi(\cdot,t),\phi(\cdot,t)]+C_{2}\epsilon^{4}\norm{\phi(\cdot,t)}_{L_{y}^{2}}^{2}+C_{3}\epsilon^{2}\sum_{\substack{j=1,\ldots,n\\l=1,2}}\mathsf{c}_{jl}(t)^{2}\\
    \geq C\bigg(\abs{\log\epsilon}^{-1}\norm{\phi^{\perp}(\cdot,t)}_{\mathcal{H}}^{2}+\epsilon^{2}\sum_{\substack{j=1,\ldots,n\\l=1,2}}\mathsf{c}_{jl}(t)^{2}\bigg).
\end{gathered}
\end{equation}

The proof of Proposition \ref{mainlinearizedprop} relies on the following differential inequality for \(\mathcal{Q}(t)\).

\begin{proposition}\label{Qdiffineqprop}
    For any smooth solution \(\phi(y,t)\) of (\ref{fulllinprobpensect}) with sufficient decay as \({\abs{y}\to\infty}\), we have for some constant \(\kappa>0\):
    \begin{equation}\label{Qdiffineq}
        \frac{d}{dt}\mathcal{Q}(t)\leq \kappa\bigg(\abs{\log\epsilon}\mathcal{Q}(t)+\epsilon^{-4}\norm{f(\cdot,t)}_{H_{y}^{1}}^{2}\bigg).
    \end{equation}
\end{proposition}

We need two preliminary results.

\begin{lemma}\label{constsdiffineq}
    For \(\phi(y,t)\) satisfying (\ref{fulllinprobpensect}) we have
    \begin{equation*}
        \epsilon^{2}\sum_{\substack{j=1,\ldots,n\\l=1,2}}\frac{d}{dt}\mathsf{c}_{jl}(t)^{2}\leq C\Bigg(\epsilon^{2}\sum_{\substack{j=1,\ldots,n\\l=1,2}}\mathsf{c}_{jl}(t)^{2}+\norm{\phi^{\perp}}_{\mathcal{H}}^{2}+\epsilon^{-2}\abs{\log\epsilon}\norm{f}_{L_{y}^{2}}^{2}\Bigg).
    \end{equation*}
\end{lemma}

\begin{proof}
    Testing equation (\ref{fulllinprobpensect}) with \(\hat{\eta}_{j}\mathcal{Z}_{jl}\) we find
    \begin{equation}\label{linprobtestzjl}
    \begin{gathered}
        \epsilon^{2}\:\operatorname{Re}\int_{\mathbb{R}^{2}}i\phi_{t}\:\overline{\hat{\eta}_{j}\mathcal{Z}_{jl}} -\operatorname{Re}\int_{\mathbb{R}^{2}}\nabla\phi\cdot\overline{\nabla(\hat{\eta}_{j}\mathcal{Z}_{jl})}+\int_{\mathbb{R}^{2}}(1-\abs{u_{*}}^{2})\phi\:\overline{\hat{\eta}_{j}\mathcal{Z}_{jl}}\\
        -\int_{\mathbb{R}^{2}}2\operatorname{Re}(\overline{u}_{*}\phi)\operatorname{Re}(\overline{u}_{*}\hat{\eta}_{j}\mathcal{Z}_{jl})=\operatorname{Re}\int_{\mathbb{R}^{2}}f\:\overline{\hat{\eta}_{j}\mathcal{Z}_{jl}}.
    \end{gathered}
    \end{equation}
    
    After translation by \(\tilde{\xi}_{j}\), the first integral in (\ref{linprobtestzjl}) can be expressed in the form 
    \begin{equation*}
        \epsilon^{2}\operatorname{Re}\int_{\mathbb{R}^{2}}i\phi_{t}(y+\tilde{\xi}_{j},t)\:\overline{(\hat{\eta}_{j}\mathcal{Z}_{jl})}(y+\tilde{\xi}_{j},t)\:dy
    \end{equation*}where
    \begin{equation*}
        \epsilon^{2}\phi_{t}(y+\tilde{\xi}_{j},t)=\epsilon^{2}\frac{d}{dt}\bigg(\phi(y+\tilde{\xi}_{j},t)\bigg)-\nabla\phi(y+\tilde{\xi}_{j},t)\cdot(\epsilon\dot{\xi}_{j}).
    \end{equation*}
    and
    \begin{equation*}
    \begin{split}
        \epsilon^{2}\frac{d}{dt}\bigg((\hat{\eta}_{j}\mathcal{Z}_{jl})(y+\tilde{\xi}_{j},t)\bigg)&=\pm\epsilon^{2}\frac{d}{dt}\bigg(e^{i\sum_{k\neq j}d_{k}\theta(y+\tilde{\xi}_{j}-\tilde{\xi}_{k})}\eta_{0}(\tfrac{8\epsilon}{\delta}y)\partial_{l}W_{j}(y)\bigg)\\[4pt]
        &\phantom{=}\quad+O(\epsilon^{4}\abs{\log\epsilon}^{2})(\hat\eta_{j}\mathcal{Z}_{jl})(y+\tilde{\xi}_{j},t)\\[4pt]
        &=i(B\hat{\eta}_{j}\mathcal{Z}_{jl})(y+\tilde{\xi}_{j},t)\\
        &\phantom{=}\quad+O(\epsilon^{4}\abs{\log\epsilon}^{2})(\hat\eta_{j}\mathcal{Z}_{jl})(y+\tilde{\xi}_{j},t).
        \end{split}
    \end{equation*}Here we have used (\ref{innerdtphasenablaphase}) and (\ref{Bdef}), and have written (again with slight abuse of notation) \(W_{j}(y)=W(y)\) if \(j\in I_{+}\) and \(W_{j}(y)=\overline{W}(y)\) if \(j\in I_{-}\).

    Since \(A=\epsilon\dot{\xi}_{j}\) in the region where \(\hat{\eta}_{j}\) is supported, the above observations combine to show
    \begin{equation*}
        \epsilon^{2}\frac{d}{dt}\operatorname{Re}\int_{\mathbb{R}^{2}}i\phi\:\overline{\hat{\eta}_{j}\mathcal{Z}_{jl}}=\mathcal{B}[\phi,\hat{\eta}_{j}\mathcal{Z}_{jl}]+\operatorname{Re}\int_{\mathbb{R}^{2}}f\:\overline{\hat{\eta}_{j}\mathcal{Z}_{jl}}+O(\epsilon^{4}\abs{\log\epsilon}^{2})\operatorname{Re}\int_{\mathbb{R}^{2}}i\phi\:\overline{\hat{\eta}_{j}\mathcal{Z}_{jl}}.
    \end{equation*}

    Using this identity, Lemma \ref{qfatzjllemma}, the estimate \(\int_{\mathbb{R}^{2}}\abs{f}\abs{\hat{\eta}_{j}\mathcal{Z}_{jl}}\leq\norm{f}_{L_{y}^{2}}\abs{\log\epsilon}^{1/2}\) and the fact that
    \begin{equation*}
        \operatorname{Re}\int_{\mathbb{R}^{2}}i\phi\:\overline{\hat{\eta}_{j}\mathcal{Z}_{jl}}=\pm \mathsf{c}_{jl'}(t)\:d_{j}\beta_{j}
    \end{equation*}where \(\beta_{j}>0\) is defined by (\ref{zjlorthog2}) and \(l':=1\) if \(l=2\), \(l':=2\) if \(l=1\), we then deduce 
    \begin{equation*}
        \big|\epsilon^{2}\frac{d}{dt}\mathsf{c}_{jl'}(t)\big|\leq C\bigg(\epsilon^{2}\sum_{l=1,2}\abs{c_{jl}(t)}+\epsilon\norm{\phi^{\perp}}_{\mathcal{H}}+\abs{\log\epsilon}\norm{f}_{L_{y}^{2}}\bigg).
    \end{equation*}The conclusion of the lemma readily follows.
\end{proof}

\begin{lemma}\label{l2normdiffineq}
    For \(\phi(y,t)\) satisfying (\ref{fulllinprobpensect}) we have
    \begin{equation*}
        \epsilon^{4}\frac{d}{dt}\norm{\phi}_{L_{y}^{2}}^{2}\leq C\Bigg(\epsilon^{2}\sum_{\substack{j=1,\ldots,n\\l=1,2}}\mathsf{c}_{jl}(t)^{2}+\norm{\phi^{\perp}}_{\mathcal{H}}^{2}+\norm{f}_{L_{y}^{2}}^{2}\Bigg).
    \end{equation*}
\end{lemma}

\begin{proof}
    Multiplying equation (\ref{fulllinprobpensect}) by \(\epsilon^{2}\:\overline{i\phi}\), taking real part and integrating over \(\mathbb{R}^{2}\) we obtain
    \begin{equation}\label{dbydtl2sqidentity}
        \epsilon^{4}\frac{d}{dt}\norm{\phi}_{L_{y}^{2}}^{2}=-4\epsilon^{2}\int_{\mathbb{R}^{2}}\operatorname{Re}(\overline{u}_{*}\phi)\operatorname{Im}(\overline{u}_{*}\phi)+2\epsilon^{2}\operatorname{Re}\int_{\mathbb{R}^{2}}i\phi\:\overline{f}.
    \end{equation}Now writing \(\phi(y,t)\) in the form (\ref{phiperpdecomp}), we have
    \begin{gather}
            \bigg|\int_{\mathbb{R}^{2}}\operatorname{Re}(\overline{u}_{*}\phi^{\perp})\operatorname{Im}(\overline{u}_{*}\phi^{\perp})\bigg|\leq C\norm{\phi^{\perp}}_{\mathcal{H}}\cdot\norm{\phi^{\perp}}_{L_{y}^{2}}\leq C\epsilon^{-2}\norm{\phi^{\perp}}_{\mathcal{H}}^{2},\notag\\
            \bigg|\int_{\mathbb{R}^{2}}\operatorname{Re}(\overline{u}_{*}\mathsf{c}_{jl}\hat{\eta}_{j}\mathcal{Z}_{jl})\operatorname{Im}(\overline{u}_{*}\phi^{\perp})\bigg|\leq C\abs{\mathsf{c}_{jl}}\norm{\phi^{\perp}}_{\mathcal{H}},\label{rezjimphiperpint}\\
            \bigg|\int_{\mathbb{R}^{2}}\operatorname{Re}(\overline{u}_{*}\phi^{\perp})\operatorname{Im}(\overline{u}_{*}\mathsf{c}_{jl}\hat{\eta}_{j}\mathcal{Z}_{jl})\bigg|\leq C\abs{\log\epsilon}^{1/2}\abs{\mathsf{c}_{jl}}\norm{\phi^{\perp}}_{\mathcal{H}},\notag\\
            \bigg|\int_{\mathbb{R}^{2}}\operatorname{Re}(\overline{u}_{*}\mathsf{c}_{jl}\hat{\eta}_{j}\mathcal{Z}_{jl})\operatorname{Im}(\overline{u}_{*}\mathsf{c}_{jl}\hat{\eta}_{j}\mathcal{Z}_{jl})\bigg|\leq C\abs{\mathsf{c}_{jl}}^{2},\notag
    \end{gather}where (\ref{rezjimphiperpint}) holds since the integral involves at main order only the mode \(\geq1\) terms \((\phi^{\perp})_{j}^{\scriptscriptstyle\geq1}\) in the Fourier expansion of \(\phi^{\perp}\). These inequalities multiplied by \(\epsilon^{2}\) imply that the first term on the right hand side of (\ref{dbydtl2sqidentity}) is controlled by \(\norm{\phi^{\perp}}_{\mathcal{H}}^{2}+\epsilon^{2}\sum\mathsf{c}_{jl}(t)^{2}\).

    \smallskip

    Concerning the second term on the right hand side of (\ref{dbydtl2sqidentity}), the decomposition (\ref{phiperpdecomp}) can again be used to show 
    \begin{equation*}
        \epsilon^{2}\bigg|\operatorname{Re}\int_{\mathbb{R}^{2}}i\phi\:\overline{f}\bigg|\leq C\bigg(\epsilon^{2}\norm{\phi^{\perp}}_{L_{y}^{2}}+\epsilon^{2}\abs{\log\epsilon}^{1/2}\sum_{j,l}\abs{\mathsf{c}_{jl}}\bigg)\norm{f}_{L_{y}^{2}}.
    \end{equation*}Using the inequality \(\epsilon^{2}\norm{\phi^{\perp}}_{L_{y}^{2}}\leq\norm{\phi^{\perp}}_{\mathcal{H}}\) we then deduce the claim of the lemma.
\end{proof}

\begin{proof}[Proof of Proposition \ref{Qdiffineqprop}]
    In view of Lemma \ref{constsdiffineq}, Lemma \ref{l2normdiffineq} and (\ref{Qlowerbound}) it suffices to establish the inequality
    \begin{equation}\label{dbydtqfineq}
        \frac{d}{dt}\mathcal{B}[\phi,\phi]\leq \kappa\bigg(\abs{\log\epsilon}\mathcal{Q}(t)+\epsilon^{-4}\norm{f(\cdot,t)}_{H_{y}^{1}}^{2}\bigg)
    \end{equation}for some constant \(\kappa>0\). By Lemma \ref{lemmtimederivqf} and Lemma \ref{remainderforminpsi} we can write
    \begin{equation*}
        \frac{d}{dt}\mathcal{B}[\phi,\phi]=\frac{2}{\epsilon^{2}}\mathcal{B}[\phi,-if]+\mathcal{R}(\phi,\phi)
    \end{equation*}where \(\mathcal{R}(\phi,\phi)=\sum_{k=1}^{6}\mathcal{R}_{k}(\psi,\psi)\), \(\phi=iu_{*}\psi\) and \(\mathcal{R}_{1},\ldots,\mathcal{R}_{6}\) are defined by (\ref{calR1def})-(\ref{calR6def}). Then using the decomposition (\ref{phiperpdecomp}), estimates (\ref{muest}), (\ref{lambdaest}) for \(\mu\), \(\lambda\) (and similar for their derivatives), Remark \ref{nablaAremark} and Lemma \ref{AnablaBlemma} we check that 
    \begin{equation*}
        \abs{\mathcal{R}(\phi,\phi)}\leq C\bigg(\norm{\phi^{\perp}}_{\mathcal{H}}^{2}+\epsilon^{2}\sum_{\substack{j=1,\ldots,n\\l=1,2}}\mathsf{c}_{jl}(t)^{2}\bigg)\leq C\abs{\log\epsilon}\mathcal{Q}(t).
    \end{equation*}

    Concerning \(\mathcal{B}[\phi,-if]\), we verify using (\ref{quadforminpsiexp}), (\ref{phiperpdecomp}) and Lemma \ref{qfatzjllemma} that
    \begin{equation*}
    \begin{split}
        \epsilon^{-2}\abs{\mathcal{B}[\phi,if]}&\leq C\bigg(\epsilon^{-2}\norm{\phi^{\perp}}_{\mathcal{H}}\cdot\norm{if}_{\mathcal{H}}+\epsilon^{-1}\sum_{j,l}\abs{\mathsf{c}_{jl}}\cdot\norm{if}_{\mathcal{H}}\bigg)\\
        &\leq C\bigg(\norm{\phi^{\perp}}_{\mathcal{H}}^{2}+\epsilon^{2}\sum_{j,l}\abs{\mathsf{c}_{jl}}^{2}+\epsilon^{-4}\norm{if}_{\mathcal{H}}^{2}\bigg)\\
        &\leq C\bigg(\abs{\log\epsilon}\mathcal{Q}(t)+\epsilon^{-4}\norm{f(\cdot,t)}_{H_{y}^{1}}^{2}\bigg).
    \end{split}    
    \end{equation*}

    These estimates (with \(C\) relabelled as \(\kappa\)) imply (\ref{dbydtqfineq}) for \(\phi(y,t)\) smooth, compactly supported and vanishing in a neighbourhood of the points \(\tilde{\xi}_{j}(t)\), \(j=1,\ldots,n\). Since such functions are dense in \(H^{1}(\mathbb{R}^{2})\) \cite{tartar2007}*{Lecture 17}, we then deduce (\ref{dbydtqfineq}) and (\ref{Qdiffineq}) for all \(\phi(y,t)\) as in the statement of the proposition. The proof is complete.
\end{proof}

We now complete the proof of the main result of this section.

\begin{proof}[Proof of Proposition \ref{mainlinearizedprop}]
   The existence and uniqueness of solutions to (\ref{fulllinprobpensect}) is standard, thus it remains to prove estimate (\ref{fulllinearizedH1est}). By Proposition \ref{Qdiffineqprop} and Gr\"{o}nwall's inequality we have 
   \begin{equation*}
   \begin{split}
       \mathcal{Q}(t)&\leq Ce^{\kappa\abs{\log\epsilon}t}\epsilon^{-4}\norm{f}_{L^{\infty}_{t}H^{1}_{y}}^{2}\\
       &\leq C\epsilon^{-4-\kappa T}\norm{f}_{L^{\infty}_{t}H^{1}_{y}}^{2}
   \end{split}    
   \end{equation*}for any sufficiently smooth solution of (\ref{fulllinprobpensect}) with \(t\in[0,T]\). Enlarging \(C_{2}>0\) in (\ref{Qlowerbound}) if necessary we can ensure
   \begin{equation*}
       \mathcal{Q}(t)\geq C\epsilon^{4}\norm{\phi(\cdot,t)}_{L^{2}_{y}}^{2},
   \end{equation*}then the inequalities
   \begin{equation*}
      \epsilon^{4}\norm{\phi}_{H_{y}^{1}}^{2}\leq \epsilon^{4}\mathcal{B}[\phi,\phi]+\epsilon^{4}\int_{\mathbb{R}^{2}}(1-\abs{u_{*}}^{2})\abs{\phi}^{2}+\epsilon^{4}\int_{\mathbb{R}^{2}}\abs{\phi}^{2}+O(\epsilon^{5})\norm{\phi}_{H^{1}_{y}}^{2}
   \end{equation*}and
   \begin{equation*}
       \abs{\mathsf{c}_{jl}(t)}\leq C\abs{\log\epsilon}^{1/2}\norm{\phi(\cdot,t)}_{L^{2}_{y}}
   \end{equation*}imply
   \begin{equation*}
       \norm{\phi(\cdot,t)}_{H^{1}_{y}}^{2}\leq C\epsilon^{-4}\mathcal{Q}(t)\leq C\epsilon^{-8-\kappa T}\norm{f}_{L^{\infty}_{t}H^{1}_{y}}^{2}.
   \end{equation*}

   Taking a square root on both sides of the above and relabelling \(\kappa/2\to\kappa\) we deduce (\ref{fulllinearizedH1est}). By density this estimate then holds for all solutions in \(L^{\infty}_{t}H^{1}_{y}\).
\end{proof}

Higher-order estimates for \(S'(u_{*})[\phi]=f\) can be deduced from (\ref{fulllinearizedH1est}) by differentiating the equation with respect to \(y\) and passing terms involving the derivatives of \(u_{*}\) onto the right-hand side. In particular, we have the following (crude) estimate in \(H^{3}(\mathbb{R}^{2})\).

\begin{corollary}
For any \(f(y,t)\) satisfying \(\norm{f}_{L^{\infty}_{t}H^{3}_{y}}<\infty\), problem (\ref{fulllinprobpensect}) has a unique solution \(\phi(y,t)\) such that
    \begin{equation}\label{fulllinearizedH3est}
        \norm{\phi}_{L^{\infty}_{t}H^{3}_{y}}\leq C\epsilon^{-12-3\kappa T}\norm{f}_{L^{\infty}_{t}H^{3}_{y}}.
    \end{equation}
\end{corollary}

\section{Solving the full problem}\label{solvefullprobsect} 

For \(\kappa>0\) the constant in (\ref{fulllinearizedH3est}), we now choose \(m\geq26+6\kappa T\) in Proposition \ref{arbitraryapproxprop} so that
\begin{equation}\label{approxerrorlinftyh3}
    \norm{S(u_{*})}_{L^{\infty}_{t}H^{3}_{y}}\leq C\epsilon^{26+6\kappa T}.
\end{equation}

The results of \(\S\)\ref{mainresultssubsect} are established as follows.

\begin{proof}[Proof of Theorems \ref{mainthm}, \ref{phasethm} and \ref{kirchhoffcorrthm}]
    All three statements follow readily from the following claim: there exists \(\phi(y,t)\) with 
    \begin{equation}
     \norm{\phi}_{L^{\infty}_{t}H^{3}_{y}}\leq \epsilon^{13+3\kappa T}   
    \end{equation}such that \(S(u_{*}+\phi)=0\). Indeed, once this fact has been established the embedding \(H^{3}(\mathbb{R}^{2})\hookrightarrow C^{1,\alpha}(\mathbb{R}^{2})\) and the construction of \(\phi_{j}\), \(\psi^{out}\) achieved in \(\S\)\ref{Firstimprovementsect}, \(\S\)\ref{arbitraryapproxsect} quickly imply the existence of an \(n\)-vortex solution with the profile described in Theorems \ref{mainthm} and \ref{phasethm}. Setting \(\xi^{*}(t):=\xi^{0}(t)+\xi^{1}(t)\) where \(\xi^{1}(t)=(\xi^{1}_{1}(t),\ldots,\xi_{n}^{1}(t))\) is defined by (\ref{xi1ode}), we also deduce the statement of \mbox{Theorem \ref{kirchhoffcorrthm}}.

    \medskip

    Now to prove the claim, let us write \(\phi=\mathcal{T}[f]\) for the unique solution of the linear problem (\ref{fulllinprobpensect}) with right-hand side \(f(y,t)\). The nonlinear problem \(S(u_{*}+\phi)=0\) can be expressed in the form
    \begin{equation}\label{fullnonlinprob}
        S(u_{*})+S'(u_{*})[\phi]+\mathcal{N}(\phi)=0
    \end{equation}where
    \begin{equation*}
    \mathcal{N}(\phi):=-2\operatorname{Re}(\overline{u}_{*}\phi)\phi-\abs{\phi}^{2}u_{*}-\abs{\phi}^{2}\phi.
    \end{equation*}Solutions of (\ref{fullnonlinprob}) with zero initial data then correspond to fixed points \(\phi=\mathcal{A}(\phi)\) of the map
        \begin{equation*}
        \mathcal{A}(\phi):=-\mathcal{T}\big[S(u_{*})+\mathcal{N}(\phi)\big].
    \end{equation*}  

    For \(\phi(y,t)\) in the closed ball
    \begin{equation*}
         \mathcal{X}:=\bigg\{\phi(y,t):\norm{\phi}_{L^{\infty}_{t}H^{3}_{y}}\leq\epsilon^{13+3\kappa T}\bigg\},
    \end{equation*}we check using (\ref{fulllinearizedH3est}), (\ref{approxerrorlinftyh3}) and the property 
    \begin{equation*}
    \norm{\mathcal{N}(\phi)}_{L^{\infty}_{t}H^{3}_{y}}\leq C\norm{\phi}_{L^{\infty}_{t}H^{3}_{y}}^{2},\quad\text{for all}\quad \phi\in\mathcal{X},
    \end{equation*}that
    \begin{equation*}
    \begin{split}
         \norm{\mathcal{A}(\phi)}_{L^{\infty}_{t}H^{3}_{y}}&\leq C\epsilon^{-12-3\kappa T}\big(\epsilon^{26+6\kappa T}+\epsilon^{26+6\kappa T} \big)\\
         &\leq C\epsilon^{14+3\kappa T}\ll\epsilon^{13+3\kappa T},\quad\text{for all}\quad\phi\in\mathcal{X},
    \end{split}
    \end{equation*}i.e. \(\mathcal{A}\) is a well-defined mapping of \(\mathcal{X}\) into itself. 

    In addition we have for \(\phi_{1}\), \(\phi_{2}\in\mathcal{X}\)
    \begin{equation*}
    \begin{split}
        \norm{\mathcal{A}(\phi_{1})-\mathcal{A}(\phi_{2})}_{L^{\infty}_{t}H^{3}_{y}}&\leq C\epsilon^{-12-3\kappa T}\cdot\epsilon^{13+3\kappa T}\norm{\phi_{1}-\phi_{2}}_{L^{\infty}_{t}H^{3}_{y}}\\
        &\leq C\epsilon\norm{\phi_{1}-\phi_{2}}_{L^{\infty}_{t}H^{3}_{y}}.
    \end{split}    
    \end{equation*}It follows that \(\mathcal{A}\) is a contraction on \(\mathcal{X}\) for \(\epsilon>0\) sufficiently small, and by the contraction mapping theorem \(\mathcal{A}\) has a unique fixed point in \(\mathcal{X}\). The proof is complete.
\end{proof}

\subsection*{Acknowledgments}
M.~del Pino has been supported by the Royal Society Research Professorship grant RP-R1-180114 and by the ERC/UKRI Horizon Europe grant ASYMEVOL, EP/Z000394/1. R.~Juneman has been supported by RP-R1-180114 and a University Research Studentship at the University of Bath.
    
\bibliography{GPrefs}

\begin{bibdiv}
\begin{biblist}

\bib{aohuangliuwei2021}{article}{
      author={Ao, W.},
      author={Huang, Y.},
      author={Liu, Y.},
      author={Wei, J.},
       title={Generalized {A}dler-{M}oser polynomials and multiple vortex rings for the {G}ross-{P}itaevskii equation},
        date={2021},
     journal={SIAM J. Math. Anal.},
      volume={53},
       pages={6959–6992},
}

\bib{arnold1965}{article}{
      author={Arnold, V.~I.},
       title={Conditions for nonlinear stability of plane stationary curvilinear flows of an ideal fluid},
        date={1965},
     journal={Dokl. Akad. Nauk SSSR},
      volume={162},
       pages={975–978},
}

\bib{bethueljerrardsmets2008}{article}{
      author={Bethuel, F.},
      author={Jerrard, R.~L.},
      author={Smets, D.},
       title={On the {NLS} dynamics for infinite energy vortex configurations on the plane},
        date={2008},
     journal={Rev. Mat. Iberoam.},
      volume={24},
       pages={671\ndash 702},
}

\bib{bethuelorlandismets2005}{article}{
      author={Bethuel, F.},
      author={Orlandi, G.},
      author={Smets, D.},
       title={Collisions and phase-vortex interactions in dissipative {G}inzburg-{L}andau dynamics},
        date={2005},
     journal={Duke Math. J.},
      volume={130},
       pages={523–614},
}

\bib{bethuelorlandismets2007multipledegree}{article}{
      author={Bethuel, F.},
      author={Orlandi, G.},
      author={Smets, D.},
       title={Dynamics of multiple degree {G}inzburg-{L}andau vortices},
        date={2007},
     journal={Comm. Math. Phys.},
      volume={272},
       pages={229–261},
}

\bib{bethuelorlandismets2007}{article}{
      author={Bethuel, F.},
      author={Orlandi, G.},
      author={Smets, D.},
       title={Quantization and motion law for {G}inzburg-{L}andau vortices},
        date={2007},
     journal={Arch. Ration. Mech. Anal.},
      volume={183},
       pages={315–370},
}

\bib{bethuelsmets2007}{article}{
      author={Bethuel, F.},
      author={Smets, D.},
       title={A remark on the {C}auchy problem for the 2{D} {G}ross-{P}itaevskii equation with nonzero degree at infinity},
        date={2007},
     journal={Differential Integral Equations},
      volume={20},
       pages={325–338},
}

\bib{carlson}{article}{
      author={Carlson, B.~C.},
       title={Elliptic {I}ntegrals},
     journal={Digital Library of Mathematical Functions, section 19.12, https://dlmf.nist.gov/19.12},
}

\bib{chenelliottqi1994}{article}{
      author={Chen, X.},
      author={Elliott, C.~M.},
      author={Qi, T.},
       title={Shooting method for vortex solutions of a complex-valued {G}inzburg-{L}andau equation},
        date={1994},
     journal={Proc. Roy. Soc. Edinburgh Sect. A},
      volume={124},
       pages={1075–1088},
}

\bib{chironpacherie2021}{article}{
      author={Chiron, D.},
      author={Pacherie, E.},
       title={Smooth branch of travelling waves for the {G}ross-{P}itaevskii equation in \(\mathbb{R}^{2}\) for small speed},
        date={2021},
     journal={Ann. Sc. Norm. Super. Pisa Cl. Sci.},
      volume={22},
       pages={1937–2038},
}

\bib{chironpacherie2023}{article}{
      author={Chiron, D.},
      author={Pacherie, E.},
       title={Coercivity for travelling waves in the {G}ross-{P}itaevskii equation in \(\mathbb{R}^{2}\) for small speed},
        date={2023},
     journal={Publ. Mat.},
      volume={67},
       pages={277–410},
}

\bib{collianderjerrard1998}{article}{
      author={Colliander, J.~E.},
      author={Jerrard, R.~L.},
       title={Vortex dynamics for the {G}inzburg-{L}andau-{S}chr\"{o}dinger equation},
        date={1998},
     journal={Int. Math. Res. Notices},
      number={7},
       pages={333–358},
}

\bib{collotgermainpacherie2025}{article}{
      author={Collot, C.},
      author={Germain, P.},
      author={Pacherie, E.},
       title={Estimates for the {G}ross-{P}itaevskii equation linearized around a vortex},
     journal={Preprint arXiv:2503.02953},
}

\bib{daviladelpinomedinarodiac2022}{article}{
      author={D\'{a}vila, J.},
      author={del Pino, M.},
      author={Medina, M.},
      author={Rodiac, R.},
       title={Interacting helical vortex filaments in the three-dimensional {G}inzburg-{L}andau equation},
        date={2022},
     journal={J. Eur. Math. Soc.},
      volume={24},
       pages={4143–4199},
}

\bib{daviladelpinomussowei2020}{article}{
      author={D\'{a}vila, J.},
      author={del Pino, M.},
      author={Musso, M.},
      author={Wei, J.},
       title={Gluing methods for vortex dynamics in {E}uler flows},
        date={2020},
     journal={Arch. Ration. Mech. Anal.},
      volume={235},
       pages={1467–1530},
}

\bib{delpinofelmerkowalczyk2004}{article}{
      author={del Pino, M.},
      author={Felmer, P.},
      author={Kowalczyk, M.},
       title={Minimality and non-degeneracy of degree-one {G}inzburg–{L}andau vortex as a {H}ardy’s type inequality},
        date={2004},
     journal={Int. Math. Res. Not.},
      volume={30},
       pages={1511\ndash 1527},
}

\bib{delpinojunemanmusso2025}{article}{
      author={del Pino, M.},
      author={Juneman, R.},
      author={Musso, M.},
       title={Solvability for the {G}inzburg-{L}andau equation linearized at the degree-one vortex},
        date={2025},
     journal={J. Funct. Anal.},
      volume={289},
       pages={111105},
}

\bib{delpinokowalczykmusso2006}{article}{
      author={del Pino, M.},
      author={Kowalczyk, M.},
      author={Musso, M.},
       title={Variational reduction for {G}inzburg-{L}andau vortices},
        date={2006},
     journal={J. Funct. Anal.},
      volume={239},
       pages={497–541},
}

\bib{dolcegallay2024}{article}{
      author={Dolce, M.},
      author={Gallay, T.},
       title={The long way of a viscous vortex dipole},
     journal={Preprint arXiv:2407.13562},
}

\bib{gallay2011}{article}{
      author={Gallay, T.},
       title={Interaction of vortices in weakly viscous planar flows},
        date={2011},
     journal={Arch. Ration. Mech. Anal.},
      volume={200},
       pages={445\ndash 490},
}

\bib{gallaysverak2024}{article}{
      author={Gallay, T.},
      author={\v{S}ver\'{a}k, V.},
       title={Arnold's variational principle and its application to the stability of planar vortices},
        date={2024},
     journal={Anal. PDE},
      volume={17},
       pages={681–722},
}

\bib{gravejatpacheriesmets2022}{article}{
      author={Gravejat, P.},
      author={Pacherie, E.},
      author={Smets, D.},
       title={On the stability of the {G}inzburg-{L}andau vortex},
        date={2022},
     journal={Proc. Lond. Math. Soc.},
      volume={125},
       pages={1015–1065},
}

\bib{herveherve1994}{article}{
      author={Herv\'{e}, R.-M.},
      author={Herv\'{e}, M.},
       title={\'{E}tude qualitative des solutions r\'{e}elles d'une \'{e}quation diff\'{e}rentielle li\'{e}e \`{a} l'\'{e}quation de {G}inzburg-{L}andau},
        date={1994},
     journal={Ann. Inst. H. Poincar\'{e} C Anal. Non Lin\'{e}aire},
      volume={11},
       pages={427–440},
}

\bib{jerrardsoner1998}{article}{
      author={Jerrard, R.~L.},
      author={Soner, H.~M.},
       title={Dynamics of {G}inzburg-{L}andau vortices},
        date={1998},
     journal={Arch. Rational Mech. Anal.},
      volume={142},
       pages={99–125},
}

\bib{jerrardspirn2008}{article}{
      author={Jerrard, R.~L.},
      author={Spirn, D.},
       title={Refined {J}acobian estimates and {G}ross-{P}itaevsky vortex dynamics},
        date={2008},
     journal={Arch. Ration. Mech. Anal.},
      volume={190},
       pages={425–475},
}

\bib{jerrardspirn2015}{article}{
      author={Jerrard, R.~L.},
      author={Spirn, D.},
       title={Hydrodynamic limit of the {G}ross-{P}itaevskii equation},
        date={2015},
     journal={Comm. Partial Differential Equations},
      volume={40},
       pages={135\ndash 190},
}

\bib{kurzkemelchermoserspirn2009}{article}{
      author={Kurzke, M.},
      author={Melcher, C.},
      author={Moser, R.},
      author={Spirn, D.},
       title={Dynamics for {G}inzburg-{L}andau vortices under a mixed flow},
        date={2009},
     journal={Indiana Univ. Math. J.},
      volume={58},
       pages={2597–2621},
}

\bib{lin1996}{article}{
      author={Lin, F.-H.},
       title={Some dynamical properties of {G}inzburg-{L}andau vortices},
        date={1996},
     journal={Comm. Pure Appl. Math.},
      volume={49},
       pages={323–359},
}

\bib{linxin1999heatflow}{article}{
      author={Lin, F.-H.},
      author={Xin, J.~X.},
       title={On the dynamical law of the {G}inzburg-{L}andau vortices on the plane},
        date={1999},
     journal={Comm. Pure Appl. Math.},
      volume={52},
       pages={1189–1212},
}

\bib{linxin1999}{article}{
      author={Lin, F.-H.},
      author={Xin, J.~X.},
       title={On the incompressible fluid limit and the vortex motion law of the nonlinear {S}chr\"{o}dinger equation},
        date={1999},
     journal={Comm. Math. Phys.},
      volume={200},
       pages={249–274},
}

\bib{liuwei2020}{article}{
      author={Liu, Y.},
      author={Wei, J.},
       title={Multivortex traveling waves for the {G}ross-{P}itaevskii equation and the {A}dler-{M}oser polynomials},
        date={2020},
     journal={SIAM J. Math. Anal.},
      volume={52},
       pages={3546–3579},
}

\bib{luhrmannschlagshahshahani2025}{article}{
      author={L\"{u}hrmann, J.},
      author={Schlag, W.},
      author={Shahshahani, S.},
       title={On the {G}ross-{P}itaevskii evolution linearized around the degree-one vortex},
     journal={Preprint arXiv:2503.07345},
}

\bib{mantonsutcliffe2004}{book}{
      author={Manton, N.},
      author={Sutcliffe, P.},
       title={Topological solitons},
      series={Cambridge Monographs on Mathematical Physics},
   publisher={Cambridge University Press},
     address={Cambridge},
        date={2004},
}

\bib{marchioropulvirenti1994}{book}{
      author={Marchioro, C.},
      author={Pulvirenti, M.},
       title={Mathematical theory of incompressible nonviscous fluids},
      series={Applied Mathematical Sciences},
   publisher={Springer-Verlag},
     address={New York},
        date={1994},
      volume={96},
}

\bib{miot2009}{article}{
      author={Miot, E.},
       title={Dynamics of vortices for the complex {G}inzburg-{L}andau equation},
        date={2009},
     journal={Anal. PDE},
      volume={2},
       pages={159–186},
}

\bib{mironescu1995}{article}{
      author={Mironescu, P.},
       title={On the stability of radial solutions of the {G}inzburg-{L}andau equation},
        date={1995},
     journal={J. Funct. Anal.},
      volume={130},
       pages={334–344},
}

\bib{mironescu1996}{article}{
      author={Mironescu, P.},
       title={Les minimiseurs locaux pour l'\'{e}quation de {G}inzburg-{L}andau sont \`{a} sym\'{e}trie radiale},
        date={1996},
     journal={C. R. Acad. Sci. Paris S\'{e}r. I Math.},
      volume={323},
       pages={593–598},
}

\bib{neu1990}{article}{
      author={Neu, J.~C.},
       title={Vortices in complex scalar fields},
        date={1990},
     journal={Phys. D},
      volume={43},
       pages={385\ndash 406},
}

\bib{ovchinnikovsigal1997}{incollection}{
      author={Ovchinnikov, Y.~N.},
      author={Sigal, I.~M.},
       title={Ginzburg-{L}andau equation. {I}. {S}tatic vortices},
        date={1997},
   booktitle={Partial differential equations and their applications ({T}oronto, {ON}, 1995), {CRM} {P}roc. {L}ecture {N}otes, vol. 12, {A}mer. {M}ath. {S}oc., {P}rovidence, {RI}},
      series={CRM Proc. Lecture Notes},
      volume={12},
   publisher={Amer. Math. Soc., Providence, RI},
       pages={199\ndash 220},
}

\bib{ovchinnikovsigal1998}{article}{
      author={Ovchinnikov, Y.~N.},
      author={Sigal, I.~M.},
       title={The {G}inzburg-{L}andau equation. {III}. {V}ortex dynamics},
        date={1998},
     journal={Nonlinearity},
      volume={11},
       pages={1277–1294},
}

\bib{ovchinnikovsigal1998longtime}{article}{
      author={Ovchinnikov, Y.~N.},
      author={Sigal, I.~M.},
       title={Long-time behaviour of {G}inzburg-{L}andau vortices},
        date={1998},
     journal={Nonlinearity},
      volume={11},
       pages={1295–1309},
}

\bib{pacardriviere2000}{book}{
      author={Pacard, F.},
      author={Rivi\`{e}re, T.},
       title={Linear and nonlinear aspects of vortices. {T}he {G}inzburg-{L}andau model},
      series={Progress in Nonlinear Differential Equations and their Applications, 39},
   publisher={Birkh\"{a}user Boston},
     address={Boston, MA},
        date={2000},
}

\bib{palaciospusateri2024}{article}{
      author={Palacios, J.~M.},
      author={Pusateri, F.},
       title={Linearized dynamic stability for vortices of {G}inzburg-{L}andau evolutions},
     journal={Preprint arXiv:2409.04393},
}

\bib{sandierserfaty2004}{article}{
      author={Sandier, E.},
      author={Serfaty, S.},
       title={Gamma-convergence of gradient flows with applications to {G}inzburg-{L}andau},
        date={2004},
     journal={Comm. Pure Appl. Math.},
      volume={57},
       pages={1627–1672},
}

\bib{serfaty2007no1}{article}{
      author={Serfaty, S.},
       title={Vortex collisions and energy-dissipation rates in the {G}inzburg-{L}andau heat flow. {I}. {S}tudy of the perturbed {G}inzburg-{L}andau equation},
        date={2007},
     journal={J. Eur. Math. Soc.},
      volume={9},
       pages={177–217},
}

\bib{serfaty2007no2}{article}{
      author={Serfaty, S.},
       title={Vortex collisions and energy-dissipation rates in the {G}inzburg-{L}andau heat flow. {II}. {T}he dynamics},
        date={2007},
     journal={J. Eur. Math. Soc.},
      volume={9},
       pages={383–426},
}

\bib{serfaty2017}{article}{
      author={Serfaty, S.},
       title={Mean field limits of the {G}ross-{P}itaevskii and parabolic {G}inzburg-{L}andau equations},
        date={2017},
     journal={J. Amer. Math. Soc.},
      volume={30},
       pages={713\ndash 768},
}

\bib{tartar2007}{incollection}{
      author={Tartar, L.},
       title={An introduction to {S}obolev spaces and interpolation spaces},
        date={2007},
   booktitle={Lecture {N}otes of the {U}nione {M}atematica {I}taliana, 3. {S}pringer, {B}erlin, {U}{M}{I}, {B}ologna},
      volume={3},
   publisher={Springer, Berlin; UMI, Bologna},
       pages={xxvi+218},
}

\bib{venkatraman2017}{article}{
      author={Venkatraman, R.},
       title={Periodic orbits of {G}ross-{P}itaevskii in the disc with vortices following point vortex flow},
        date={2017},
     journal={Calc. Var. Partial Differential Equations},
      volume={56},
       pages={Paper No. 64, 35 pp.},
}

\end{biblist}
\end{bibdiv}

\end{document}